\pdfoutput=1
\documentclass[11pt]{article}

\usepackage[margin=1in]{geometry}
\usepackage{amsmath,amssymb,amsthm,mathtools}
\usepackage{mathrsfs}
\usepackage{aliascnt}
\usepackage{enumitem}
\usepackage{array}
\usepackage{microtype}
\usepackage[colorlinks=true,linkcolor=blue,citecolor=blue,urlcolor=blue]{hyperref}
\usepackage[nameinlink,noabbrev]{cleveref}

\numberwithin{equation}{section}

\newtheorem{theorem}{Theorem}[section]

\newaliascnt{proposition}{theorem}
\newtheorem{proposition}[proposition]{Proposition}
\aliascntresetthe{proposition}

\newaliascnt{lemma}{theorem}
\newtheorem{lemma}[lemma]{Lemma}
\aliascntresetthe{lemma}

\newaliascnt{corollary}{theorem}
\newtheorem{corollary}[corollary]{Corollary}
\aliascntresetthe{corollary}

\newaliascnt{claim}{theorem}

\aliascntresetthe{claim}

\theoremstyle{definition}
\newaliascnt{definition}{theorem}
\newtheorem{definition}[definition]{Definition}
\aliascntresetthe{definition}

\newaliascnt{remark}{theorem}
\newtheorem{remark}[remark]{Remark}
\aliascntresetthe{remark}

\newaliascnt{assumption}{theorem}

\aliascntresetthe{assumption}
\newaliascnt{example}{theorem}
\newtheorem{example}[example]{Example}
\aliascntresetthe{example}

\crefname{theorem}{Theorem}{Theorems}
\crefname{proposition}{Proposition}{Propositions}
\crefname{lemma}{Lemma}{Lemmas}
\crefname{corollary}{Corollary}{Corollaries}
\crefname{definition}{Definition}{Definitions}
\crefname{remark}{Remark}{Remarks}
\crefname{example}{Example}{Examples}
\Crefname{theorem}{Theorem}{Theorems}
\Crefname{proposition}{Proposition}{Propositions}
\Crefname{lemma}{Lemma}{Lemmas}
\Crefname{corollary}{Corollary}{Corollaries}
\Crefname{definition}{Definition}{Definitions}
\Crefname{remark}{Remark}{Remarks}
\Crefname{example}{Example}{Examples}

\newcommand{\Hn}{\mathbb H^n}
\newcommand{\Sn}{\mathbb S^{n-1}}

\newcommand{\Om}{\Omega}
\newcommand{\Gap}{\Gamma}
\newcommand{\Braw}{B_R^{\mathrm{raw}}}

\newcommand{\calA}{\mathcal A}
\newcommand{\calB}{\mathcal B}
\newcommand{\calH}{\mathcal H}

\newcommand{\eps}{\varepsilon}
\DeclareMathOperator{\diam}{diam}
\DeclareMathOperator{\spec}{spec}
\DeclareMathOperator{\dist}{dist}

\DeclareMathOperator{\Lip}{Lip}

\title{Effective Angular Asymptotics and the Sharp \texorpdfstring{$D^{-3}$}{D^-3} Horoconvex Gap Scale}
\author{SangHyun Park\\Yonsei University}
\date{June 2026}

\begin{document}
\maketitle

\begin{abstract}
Let \(\Om\subset \mathbb H^n\), \(n\ge2\), be a bounded horoconvex domain and let
\(\Gap(\Om)=\lambda_2(\Om)-\lambda_1(\Om)\) be its Dirichlet fundamental gap.  We identify the
large-diameter first-band asymptotic expansion for the first two Dirichlet eigenvalues.  If a divergent sequence is written, after
Chebyshev centering, as a radial graph
\[
  \Om_{R,V}=\{(r,\theta):0\le r\le R-V(\theta)\},\qquad R\to\infty,
\]
with deficits converging to an admissible horoconvex profile \(V\in\calA_n\), then, for \(j=0,1\),
\[
  \lambda_{j+1}(\Om_{R,V})=
  \alpha^2+\frac{\pi^2}{R^2}
  +\frac{2\pi^2}{R^3}\Bigl(\eta_j(T_n+V)-b_0\Bigr)+o(R^{-3}),
  \qquad \alpha=\frac{n-1}{2}.
\]
Here \(b_0=-\gamma_E-\psi(\alpha)\), and \(T_n\) is the nonlocal spherical operator
\[
  T_nY_{\ell m}=\left[\psi\!\left(\ell+\frac{n-1}{2}\right)
  -\psi\!\left(\frac{n-1}{2}\right)\right]Y_{\ell m}.
\]
The operator \(T_n\) is the threshold scattering contribution of the hyperbolic radial equation: all
angular sectors have the same leading radial energy \(\pi^2/R^2\), and the first splitting occurs
through the zero-energy scattering lengths
\(b_\ell=-\gamma_E-\psi(\ell+\alpha)\).  The boundary deficit \(V\) shifts the same endpoint phase and
therefore appears as a multiplication potential.

As consequences, the horoconvex fundamental gap has the sharp polynomial scale \(D^{-3}\), and its
leading constant at this scale is characterized by the compact variational formula
\[
  \lim_{D\to\infty}D^3
  \inf_{\substack{\Om\ \mathrm{horoconvex}\\ \diam\Om=D}}
  \Gap(\Om)
  =16\pi^2\inf_{V\in\calA_n}\bigl[\eta_1(T_n+V)-\eta_0(T_n+V)\bigr].
\]
Geodesic balls realize the polynomial scale, and an explicit admissible axial perturbation lowers the
reduced leading-constant value at first order by \(2(n-1)/(n(n+2))\).
\end{abstract}

\tableofcontents

\section{Introduction}

For a bounded domain \(\Om\) in a Riemannian manifold, let
\[
  0<\lambda_1(\Om)<\lambda_2(\Om)\le\cdots
\]
be the Dirichlet eigenvalues of the positive Laplace--Beltrami operator, and set
\[
  \Gap(\Om)=\lambda_2(\Om)-\lambda_1(\Om).
\]
The fundamental gap problem asks how \(\Gap(\Om)\) is controlled by the geometry of \(\Om\), most
classically by its diameter \(D\).  The modern Euclidean theory starts from the one-dimensional
comparison paradigm of Singer--Wong--Yau--Yau \cite{SWYY} and culminates in the sharp
Andrews--Clutterbuck theorem for convex Euclidean domains, where the gap has scale \(D^{-2}\)
\cite{AC}.  Positive curvature admits closely related comparison theorems, including the sharp
spherical estimate of Seto--Wang--Wei \cite{SWW}; related spherical gap estimates and probabilistic
proofs were developed by He--Wei--Zhang and Cho--Wei--Yang \cite{HWZ,CWY}.  Negative curvature changes
the question: McKean's spectral threshold \cite{McKean} accounts for the term
\(\alpha^2=(n-1)^2/4\), while convexity alone no longer gives a diameter-only \(D^{-2}\)-gap bound.
Bourni--Clutterbuck--Nguyen--Stancu--Wei--Wheeler proved that, for convex domains in
\(\mathbb H^n\), the quantity \(D^2\Gap\) can be arbitrarily small \cite{BCNSWW}; Khan--Nguyen later
showed that this instability is a general negative-curvature phenomenon \cite{KN}.

Horoconvexity is the natural stronger convexity condition in hyperbolic space.  Nguyen--Stancu--Wei
showed, using the large-radius behavior of geodesic balls and related hyperbolic ball estimates
\cite{BenguriaLinde}, that even among horoconvex domains \(D^2\Gap\) has no positive lower bound
\cite{NSW}.  Khan--Nguyen--T\"urkoen--Wei, Khan--Saha--T\"urkoen, and Khan--T\"urkoen subsequently
developed conformal log-concavity estimates and diameter-dependent positivity results in
positive-curvature, conformal, and horoconvex settings \cite{KNTW,KST,KT}.  These results establish
positivity in terms of diameter and dimension.  In the large-diameter horoconvex regime, however,
the explicit bound obtained in \cite{KT} decays doubly exponentially in the diameter; the present
paper identifies the polynomial asymptotic model and the leading-constant problem at that scale.
This line builds on comparison and probabilistic gap methods, including Cho--Wei--Yang \cite{CWY},
and is complemented by small-diameter super-log-concavity results of Wei--Xiao \cite{WeiXiao}.

It is useful to separate the scale question from the leading-constant question.  Geodesic balls and
Nguyen--Stancu--Wei's result show that the Euclidean \(D^{-2}\) scale is unavailable and that the
correct polynomial order, if a polynomial lower bound holds, must be \(D^{-3}\).  A contemporaneous
public announcement by J. Ennis reports an independent matching \(D^{-3}\) lower-bound scale.\footnote{J. Ennis (@johnennis), announcement on X, June 1, 2026, accessed June 9, 2026.  At the time of writing, no public manuscript was available to the author.  The present paper is independent and proves the first-band asymptotic expansion and the variational formula for the leading large-diameter constant.}
The announcement brings the scale question into the same large-diameter regime as the Nguyen--Stancu--Wei
geodesic-ball model.  The abstract accompanying the announcement refers to a limiting angular
operator.  The present paper identifies that operator explicitly as the digamma multiplier \(T_n\),
traces it to the sector scattering lengths \(b_\ell=-\gamma_E-\psi(\ell+\alpha)\), proves the
two-sided \(R^{-3}\) first-band expansion for \(\lambda_1\) and \(\lambda_2\), reduces the leading
constant to a compact nonlocal variational problem, and gives an explicit admissible direction
lowering the geodesic-ball reduced value.

After Chebyshev centering, a large horoconvex body has radial boundary
\[
  \rho(\theta)=R-V(\theta),
\]
where the deficits \(V\) converge, along subsequences, in a compact support-envelope class
\(\calA_n\).  The first two Dirichlet eigenvalues then satisfy
\[
  \frac{R^3}{2\pi^2}
  \left(\lambda_{j+1}(\Omega_{R,V})-\alpha^2-\frac{\pi^2}{R^2}\right)
  \longrightarrow \eta_j(T_n+V)-b_0,
  \qquad j=0,1.
\]
The angular operator is determined by the radial calculation.  In the \(\ell\)-th spherical sector, the radial
Liouville equation has zero-energy scattering length
\[
  b_\ell=-\gamma_E-\psi(\ell+\alpha),
\]
and the difference \(b_0-b_\ell=\psi(\ell+\alpha)-\psi(\alpha)\) is precisely the multiplier of
\(T_n\).  The boundary deficit \(V(\theta)\) shifts the same endpoint phase and hence appears as an
ordinary potential.  Thus the large-diameter hyperbolic gap problem reduces to the nonlocal spherical
Hamiltonian
\[
  T_n+V-b_0.
\]

The distinction between polynomial scale and leading constant is essential.  Geodesic balls determine the sharp
polynomial scale \(D^{-3}\), and the large-diameter constant is governed by the reduced variational
problem.  We prove
\[
  \lim_{D\to\infty}D^3
  \inf_{\substack{\Omega\ \mathrm{horoconvex}\\ \diam\Omega=D}}
  \Gap(\Omega)
  =16\pi^2
  \inf_{V\in\calA_n}\bigl(\eta_1(T_n+V)-\eta_0(T_n+V)\bigr),
\]
and an explicit admissible axial perturbation decreases the reduced gap at \(V=0\), the geodesic-ball
profile.  Thus balls are scale-sharp, while the displayed axial direction gives a strict one-sided
decrease for the leading-constant functional.

\begin{center}
\scriptsize
\renewcommand{\arraystretch}{1.18}
\setlength{\tabcolsep}{3pt}
\begin{tabular}{>{\raggedright\arraybackslash}p{.20\textwidth}>{\raggedright\arraybackslash}p{.20\textwidth}>{\raggedright\arraybackslash}p{.27\textwidth}>{\raggedright\arraybackslash}p{.20\textwidth}}
\hline
Work & Class & Main contribution & Relation here \\
\hline
Singer--Wong--Yau--Yau; Andrews--Clutterbuck & Euclidean convex domains & one-dimensional \(D^{-2}\) gap paradigm & baseline contrast \\
Seto--Wang--Wei; He--Wei--Zhang; Cho--Wei--Yang & spherical convex domains & positive-curvature comparison and probabilistic gap proofs & curvature comparison background \\
McKean & negative curvature & bottom spectral threshold \((n-1)^2/4\) & explains the \(\alpha^2\) threshold \\
Bourni--Clutterbuck--Nguyen--Stancu--Wei--Wheeler; Khan--Nguyen & hyperbolic/negative-curvature convex domains & diameter-only convex gap instability & motivates horoconvexity \\
Nguyen--Stancu--Wei & horoconvex domains & \(D^2\Gap\) has no positive lower bound; balls give \(D^{-3}\) upper behavior & source of the large-ball model \\
Khan--Nguyen--T\"urkoen--Wei; Khan--Saha--T\"urkoen; Khan--T\"urkoen; Wei--Xiao & horoconvex/conformal settings & log-concavity, diameter-dependent positivity, and small-diameter concavity estimates; the explicit large-\(D\) bound in \cite{KT} is doubly exponential & existence/positivity and complementary regimes; not a sharp large-\(D\) asymptotic formula \\
This paper & horoconvex domains & effective Hamiltonian \(T_n+V-b_0\), variational formula for the leading constant, explicit axial decrease from the ball profile & sharp asymptotic structure \\
\hline
\end{tabular}
\end{center}

The analytic reduction uses four ingredients.  The moving Dirichlet condition is encoded by a boundary
characteristic pencil, whose characteristic values are counted by analytic Fredholm multiplicity, as
in elliptic boundary and Grushin reductions \cite{SjoestrandZworski,Grubb,McLean,Kato}.  The number
\(b_\ell\) is the threshold phase slope of the regular radial solution, giving the low-energy
scattering input for the first band \cite{JensenKato,DZ}.  The effective angular operator has a
Poisson-semigroup representation, placing it among nonlocal generators obtained from Bernstein
functions of classical operators \cite{SchillingSongVondracek}.  Since its multiplier grows like
\(\log\ell\), the reduced problem is naturally compared with logarithmic-order nonlocal operators,
including the logarithmic Laplacian and its spectral theory \cite{ChenWeth,LaptevWeth}.  Finally,
the passage from finite angular windows to the full pencil is a Schur reduction with compact
multiplication remainder \(\sup_{V\in\mathcal B}\|P_{>L}VP_M\|\to0\), analogous in role to first-band
mode separation in long or thin domain problems \cite{Post,Kuchment}.

\subsection{Main results}

A domain will be called \emph{horoconvex} if it is an intersection of closed horoballs.  In the
smooth case this is equivalent to all principal curvatures of \(\partial\Om\), computed with respect
to the outward unit normal, being at least \(1\).  We freely pass between a bounded horoconvex
open domain and the compact horoconvex body \(K=\overline\Omega\): the support functions and
Chebyshev data are attached to \(K\), while the Dirichlet spectrum is that of \(\Omega=\operatorname{int}K\).
The Dirichlet eigenvalues are unchanged by modifying the boundary of the compact body.

\begin{theorem}[Sharp large-diameter scale]\label{thm:main-lower}
For every \(n\ge2\) there exist constants \(c_n>0\) and \(D_0(n)>0\) such that every bounded
horoconvex domain \(\Om\subset\Hn\) with nonempty interior and
\(D=\diam\Om\ge D_0(n)\) satisfies
\[
  \lambda_2(\Om)-\lambda_1(\Om)\ge c_nD^{-3}.
\]
\end{theorem}

The theorem is sharp in scale because geodesic balls satisfy
\[
  \lambda_2(B_R)-\lambda_1(B_R)=\frac{4\pi^2}{(n-1)R^3}+O(R^{-4}),
  \qquad R\to\infty.
\]
Since \(\diam B_R=2R\), this asymptotic proves the sharpness of the \(D^{-3}\) scale.

The precise theorem is an asymptotic reduction to an effective Hamiltonian on the sphere.  Let
\(\alpha=(n-1)/2\).  For \(\ell=0,1,2,\ldots\), let \(\calH_\ell\) denote the spherical harmonics of
degree \(\ell\) on \(\Sn\).  Define the self-adjoint operator \(T_n\) by
\begin{equation}\label{eq:Tn-intro}
  T_nY=\left[\psi(\ell+\alpha)-\psi(\alpha)\right]Y,
  \qquad Y\in\calH_\ell,
\end{equation}
where \(\psi=\Gamma'/\Gamma\) is the digamma function.  If \(V\in L^\infty(\Sn)\), we write
\[
  \eta_0(T_n+V)<\eta_1(T_n+V)\le\cdots
\]
for the eigenvalues of \(T_n+V\).  The admissible class \(\calA_n\) is defined intrinsically by the
support-envelope model in Definition~\ref{def:admissible-class}; Theorem~\ref{thm:geometry-package}
then identifies this class with the compact family of large-radius radial deficits of horoconvex
bodies.

\begin{theorem}[Effective angular limit]\label{thm:effective-limit}
Let \(\Om_k\subset\Hn\) be horoconvex domains with circumradii \(R_k\to\infty\), after choosing
Chebyshev centers.  After passing to a subsequence there exists \(V\in\calA_n\) such that
\(\diam\Om_k=2R_k+O(1)\) and, for \(j=1,2\),
\begin{equation}\label{eq:eff-intro}
  \lambda_j(\Om_k)=
  \alpha^2+\frac{\pi^2}{R_k^2}
  +\frac{2\pi^2}{R_k^3}\Bigl(\eta_{j-1}(T_n+V)-b_0\Bigr)
  +o(R_k^{-3}),
\end{equation}
where
\[
  b_0=-\gamma_E-\psi(\alpha).
\]
Conversely, for every \(V\in\calA_n\) there is a sequence of horoconvex domains realizing
\eqref{eq:eff-intro}.
\end{theorem}

Combining \cref{thm:effective-limit} with compactness of \(\calA_n\) gives the sharp constant as an
effective variational problem.

\begin{theorem}[Asymptotic variational formula]\label{thm:var-formula}
For every \(n\ge2\),
\begin{equation}\label{eq:main-constant}
  \lim_{D\to\infty}D^3
  \inf_{\substack{\Om\subset\Hn\ \mathrm{horoconvex}\\ \diam\Om=D}}
  \bigl(\lambda_2(\Om)-\lambda_1(\Om)\bigr)
  =16\pi^2\inf_{V\in\calA_n}\mathfrak g(V),
\end{equation}
where
\[
  \mathfrak g(V):=\eta_1(T_n+V)-\eta_0(T_n+V).
\]
Moreover
\[
  \inf_{V\in\calA_n}\mathfrak g(V)>0.
\]
\end{theorem}

Formula \eqref{eq:main-constant} separates the polynomial scale from the remaining compact
variational problem.  The sharp large-diameter constant is characterized by the value of this
nonlocal problem on \(\calA_n\).  Theorem~\ref{thm:ball-not-optimal} gives the first structural information
on that problem by exhibiting an admissible axial direction along which the reduced gap decreases
from the symmetric point \(V\equiv0\).

\begin{theorem}[Explicit axial decrease from the symmetric point]\label{thm:ball-not-optimal}
For each unit vector \(e\in\mathbb R^n\) there is an admissible curve
\(V_\eps\in\calA_n\),
\[
  V_\eps(\theta)=\eps\bigl(1-(e\cdot\theta)^2\bigr)+O(\eps^2)
  \quad\text{in }C(\Sn),
\]
such that
\begin{equation}\label{eq:ball-perturb-intro}
  \mathfrak g(V_\eps)=\frac{2}{n-1}
  -\frac{2(n-1)}{n(n+2)}\eps+O(\eps^2).
\end{equation}
Equivalently, \(d\mathfrak g(V_\eps)/d\eps|_{\eps=0}<0\).  Thus the displayed admissible curve gives a strict one-sided decrease from the large-ball profile in the reduced problem.
\end{theorem}

\subsection{Outline of the proof}

The argument separates the geometry, the one-dimensional scattering calculation, and the moving-boundary
spectral reduction.

\paragraph{1. Compact horoconvex geometry.}
If \(R\) is the circumradius and \(o\) is a Chebyshev center, the normalized horospherical support
profile
\[
  q_R(\xi)=e^{-R}s_K(\xi)
\]
is trapped between \(1/2+o(1)\) and \(1+o(1)\).  The radial function has the form
\[
  \rho_R(\theta)=R-V_R(\theta),
\]
and every subsequential limit is a support envelope
\begin{equation}\label{eq:Vq-intro}
  V_q(\theta)=\sup_{\xi\in\Sn}
  \log\frac{1-\theta\cdot\xi}{2q(\xi)},
  \qquad \frac12\le q\le1.
\end{equation}
Theorem~\ref{thm:geometry-package} characterizes the large-radius compactification, canonical support duality,
horoconvex realization, and exact diameter calibration.

\paragraph{2. Zero-energy scattering.}
The Liouville transform reduces the \(\ell\)-th angular branch to
\[
  -\frac{d^2}{dr^2}+\frac{(\ell+\alpha)(\ell+\alpha-1)}{\sinh^2 r}.
\]
Its zero-energy regular solution has scattering length
\[
  b_\ell=-\gamma_E-\psi(\ell+\alpha).
\]
Thus a large interval of length \(R\) has first radial energy
\[
  \frac{\pi^2}{(R+b_\ell)^2}
  =\frac{\pi^2}{R^2}-\frac{2\pi^2b_\ell}{R^3}+O(R^{-4}),
\]
and the difference \(b_0-b_\ell=\psi(\ell+\alpha)-\psi(\alpha)\) is exactly the multiplier of \(T_n\).

\paragraph{3. Boundary Grushin reduction.}
For a radial graph \(r=R-V(\theta)\), one solves the regular radial equation from the origin and
imposes the Dirichlet condition through the boundary operator
\[
  B_R^{\mathrm{raw}}(k,V)f=\mathcal S(k)f(R-V(\theta),\theta).
\]
The normalized operator
\[
  \mathfrak B_R(\zeta,V)=-\frac R\pi B_R^{\mathrm{raw}}(k_R(\zeta),V),
  \qquad
  k_R(\zeta)=\frac\pi R+\frac{\pi\zeta}{R^2},
\]
satisfies, on finite angular windows,
\[
  \mathfrak B_R(\zeta,V)=\zeta I-(T_n+V-b_0)+o(1).
\]
The Sobolev--Fredholm framework proves that characteristic zeros of \(B_R^{\mathrm{raw}}\) are exactly the
Dirichlet eigenvalues, with multiplicity.  The quantitative part uses ordered angular cutoffs:
\[
  M\text{ fixed}\quad\Longrightarrow\quad L\gg M\quad\Longrightarrow\quad R\to\infty
  \quad\Longrightarrow\quad M\to\infty.
\]
The quantitative high-mode estimate is
\[
  \sup_{V\in\mathcal B}\|P_{>L}VP_M\|\longrightarrow0
  \qquad(L\to\infty)
\]
for each fixed inner cutoff \(M\) and each compact family \(\mathcal B\subset C(\Sn)\).

\paragraph{4. Positivity of the effective gap.}
The identity
\begin{equation}\label{eq:poisson-intro}
  T_n=\int_0^1(I-P_\rho)\frac{\rho^{\alpha-1}}{1-\rho}\,d\rho
\end{equation}
expresses \(T_n\) through the spherical Poisson semigroup.  Hence \(e^{-tT_n}\), and therefore
\(e^{-t(T_n+V)}\), is positivity improving.  The ground state of \(T_n+V\) is simple for every
\(V\), and compactness of \(\calA_n\) gives
\[
  \inf_{V\in\calA_n}\mathfrak g(V)>0.
\]

\paragraph{5. Exact diameter transfer.}
Graph asymptotics are transferred back to actual horoconvex bodies through the realization theorem.
The diameter expansion
\[
  \diam K_R(q)=2R+\Delta(V_q)+O(e^{-2R})
\]
then allows the upper-bound sequence to be calibrated to the exact constraint \(\diam\Om=D\).  The
factor \(16\pi^2\) in \cref{thm:var-formula} is the product of the effective coefficient
\(2\pi^2/R^3\) and \((2R)^3\).

\subsection{Notation}

\begin{center}
\renewcommand{\arraystretch}{1.15}
\begin{tabular}{>{$}l<{$} p{0.72\textwidth}}
\hline
\alpha=(n-1)/2 & the hyperbolic threshold parameter; the bottom of the essential radial energy is \(\alpha^2\).\\
R & large circumradius or radial length parameter, depending on context.\\
D & diameter; for admissible large bodies \(D=2R+O(1)\).\\
q, q_R & normalized horospherical support profiles, with \(1/2\le q\le1\).\\
V, V_R & radial deficits: \(\rho_R(\theta)=R-V_R(\theta)\).\\
\calA_n & compact class of limiting horoconvex deficits.\\
b_\ell & zero-energy scattering length in angular sector \(\ell\): \(b_\ell=-\gamma_E-\psi(\ell+\alpha)\).\\
T_n & nonlocal spherical operator with multiplier \(b_0-b_\ell\).\\
\eta_j(T_n+V) & eigenvalues of the effective angular Hamiltonian.\\
\mathfrak g(V) & effective fundamental gap \(\eta_1(T_n+V)-\eta_0(T_n+V)\).\\
\hline
\end{tabular}
\end{center}

\section{Horoconvex compactness}\label{sec:horoconvex-compactness}

We use the hyperboloid model when convenient.  Fix \(o\in\Hn\).  In polar coordinates
\((r,\theta)\in[0,\infty)\times\Sn\), the Busemann-type horospherical support functions are
\begin{equation}\label{eq:beta}
  \beta_\xi(r,\theta)=\cosh r-\sinh r\,\theta\cdot\xi,
  \qquad \xi\in\Sn.
\end{equation}
A compact horoconvex body \(K\) is described by a support profile \(s_K:\Sn\to(0,\infty)\):
\begin{equation}\label{eq:support-body}
  K=\bigcap_{\xi\in\Sn}\{x:\beta_\xi(x)\le s_K(\xi)\}.
\end{equation}
The profile is lower semicontinuous and is uniquely determined after replacing it by the lower
semicontinuous envelope
\[
  s_K(\xi)=\sup_{x\in K}\beta_\xi(x).
\]

A Chebyshev center of a compact horoconvex body lies in the body.  Indeed, horoballs are
geodesically convex, so \(K\) is a closed convex subset of the Hadamard manifold \(\Hn\).  If a
circumcenter \(o\) were outside \(K\), let \(p\) be the metric projection of \(o\) onto \(K\).  The
CAT\((0)\) projection inequality gives the Pythagorean estimate
\[
  d(p,x)^2+d(o,p)^2\le d(o,x)^2,
  \qquad x\in K .
\]
If \(R=\max_{x\in K}d(o,x)\), then
\[
  d(p,x)\le (R^2-d(o,p)^2)^{1/2}<R
  \qquad x\in K,
\]
so the ball centered at \(p\) has smaller radius than the ball centered at \(o\), a contradiction.
Consequently \(K\) is star-shaped with respect to its Chebyshev center, and it is legitimate to write
\(K\) as a radial graph from that center.

Let \(R\) be the circumradius of \(K\) and choose a Chebyshev center \(o\).  Thus
\(K\subset \overline{B_R(o)}\), and no strictly smaller ball centered at any point contains \(K\).
Define
\begin{equation}\label{eq:qR}
  q_R(\xi)=e^{-R}s_K(\xi).
\end{equation}

The following elementary lemma is the geometric normalization that prevents a large horoconvex body
from having an \(O(R)\) radial loss in some direction.

\begin{lemma}[Support normalization]\label{lem:support-normalization}
Let \(K\) be a horoconvex body with Chebyshev center \(o\), circumradius \(R\), and normalized
support profile \(q_R\).  Then
\begin{equation}\label{eq:q-bounds}
  \frac12\le q_R(\xi)\le1
  \qquad\text{for all }\xi\in\Sn.
\end{equation}
Moreover the contact set
\[
  C_R=\{\theta\in\Sn:(R,\theta)\in K\}
\]
satisfies \(0\in\operatorname{conv} C_R\).
\end{lemma}

\begin{proof}
The upper bound is immediate from \(K\subset B_R(o)\).  Indeed, for \(x=(r,\theta)\in B_R(o)\),
\[
  \beta_\xi(x)=\cosh r-\sinh r\,\theta\cdot\xi\le \cosh R+\sinh R=e^R,
\]
so \(s_K(\xi)\le e^R\).

We next prove the contact condition.  If \(0\notin\operatorname{conv} C_R\), choose a unit vector
\(v\in T_o\Hn\) and a number \(a>0\) such that \(\theta\cdot v\ge a\) for every \(\theta\in C_R\).
Let \(o_t=\exp_o(tv)\).  For \(x=(r,\theta)\), the first variation formula gives
\[
  \left.\frac d{dt}\right|_{t=0}d(o_t,x)=-\theta\cdot v .
\]
Thus all contact points satisfy \(d(o_t,x)\le R-at/2\) for \(t>0\) small.  Points with
\(d(o,x)\le R-\delta\) remain within radius \(R-at/4\) after choosing \(t\ll\delta\), and points with
\(d(o,x)>R-\delta\) have directions close to \(C_R\) by compactness, hence still satisfy the same
negative first variation.  Therefore a ball centered at \(o_t\) has radius strictly smaller than
\(R\), contradicting the Chebyshev property.

Fix \(\xi\in\Sn\).  Since \(0\in\operatorname{conv}C_R\), there is \(\theta\in C_R\) with
\(\theta\cdot\xi\le0\).  As \((R,\theta)\in K\),
\[
  s_K(\xi)\ge \beta_\xi(R,\theta)
  =\frac{e^R}{2}(1-\theta\cdot\xi)+\frac{e^{-R}}2(1+\theta\cdot\xi)
  \ge \frac{e^R}{2}.
\]
Dividing by \(e^R\) gives the lower bound.
\end{proof}

For a star-shaped body about \(o\), let
\[
  \rho_K(\theta)=\sup\{r:(r,\theta)\in K\},
  \qquad V_R(\theta)=R-\rho_K(\theta).
\]
The next calculation is the basic support-to-radial conversion.  If \(r=R-V\), then
\begin{equation}\label{eq:beta-asymp}
  e^{-R}\beta_\xi(R-V,\theta)
  =\frac{e^{-V}}2(1-\theta\cdot\xi)
  +\frac{e^{-2R+V}}2(1+\theta\cdot\xi).
\end{equation}
Thus the inequality \(\beta_\xi(R-V,\theta)\le e^Rq_R(\xi)\) is, up to an exponentially small error,
\[
  V\ge \log\frac{1-\theta\cdot\xi}{2q_R(\xi)}.
\]
This motivates the limiting envelope.

\begin{definition}[Normalized support profiles and admissible deficits]\label{def:admissible-class}
A \emph{normalized limiting support profile} is a lower semicontinuous function
\(q:\Sn\to[1/2,1]\) whose contact set
\begin{equation}\label{eq:Cq-def}
  C(q):=\left\{\theta\in\Sn:
  q(\xi)\ge \frac{1-\theta\cdot\xi}{2}\text{ for every }\xi\in\Sn\right\}
\end{equation}
satisfies
\begin{equation}\label{eq:limiting-cheb-normalization}
  0\in\operatorname{conv} C(q).
\end{equation}
For such a profile define its radial envelope
\begin{equation}\label{eq:Vq}
  V_q(\theta):=\sup_{\xi\in\Sn}\log\frac{1-\theta\cdot\xi}{2q(\xi)} .
\end{equation}
The horoconvex admissible class is
\begin{equation}\label{eq:A-def-profile}
  \calA_n:=\{V_q:\ q\text{ is a normalized limiting support profile}\}\subset C(\Sn).
\end{equation}
We quotient normalized support profiles by the equivalence relation \(q\sim\tilde q\) if
\(V_q=V_{\tilde q}\).  The topology on \(\calA_n\) is the uniform topology of \(C(\Sn)\) through the
envelope representatives.  When a sequence of profiles is said to converge below, it means
convergence of the associated radial envelopes, after passing to the canonical representative of
Lemma~\ref{lem:canonical-dual-support}.
\end{definition}

\begin{remark}[Why the definition uses envelopes]\label{rem:profile-topology}
The normalized profiles \(q_R=e^{-R}s_K\) of finite horoconvex bodies may fail to be equicontinuous.
The radial deficit is the compact representative of the asymptotic shape.  Definition~\ref{def:admissible-class}
therefore builds the admissible class in the envelope topology.  The compactness-realization theorem
below proves that this intrinsic definition is the set of uniform large-radius limits of
Chebyshev-centered horoconvex bodies.
\end{remark}

\begin{remark}[Horospherical Legendre transform]\label{rem:horo-legendre}
The map \(q\mapsto V_q\) is a horospherical analogue of a Legendre transform.  With
\[
  c(\theta,\xi)=-\log\frac{1-\theta\cdot\xi}{2},
\]
the envelope can be written as
\[
  V_q(\theta)=\sup_{\xi\in\Sn}\{-c(\theta,\xi)-\log q(\xi)\}.
\]
Thus \(\calA_n\) is a normalized class of \(c\)-convex envelopes.  The canonical support profile in
Lemma~\ref{lem:canonical-dual-support} is the corresponding biconjugate representative.  This explains why the radial deficits form a compact convexity-type class rather than an arbitrary family of Lipschitz functions.
\end{remark}

\begin{example}[Basic admissible profiles]\label{ex:admissible-profiles}
The support-envelope model contains the elementary large-radius shapes used later.
\begin{enumerate}[label=(\roman*)]
\item If \(q\equiv1\), then
\[
  V_q(\theta)=\sup_\xi\log\frac{1-\theta\cdot\xi}{2}=0,
\]
so \(V\equiv0\) is the large-ball deficit.
\item In dimension \(n=3\), one has \(\alpha=1\) and hence
\[
  T_3Y_{\ell m}=\bigl[\psi(\ell+1)-\psi(1)\bigr]Y_{\ell m}=H_\ell Y_{\ell m},
\]
where \(H_\ell\) is the \(\ell\)-th harmonic number.  Thus, in dimension three, the effective angular operator is the harmonic-number multiplier.
\end{enumerate}
\end{example}

\begin{center}
\scriptsize
\renewcommand{\arraystretch}{1.18}
\begin{tabular}{>{\raggedright\arraybackslash}p{.18\textwidth}>{\raggedright\arraybackslash}p{.24\textwidth}>{\raggedright\arraybackslash}p{.20\textwidth}>{\raggedright\arraybackslash}p{.18\textwidth}}
\hline
Support profile & Deficit envelope & Limiting geometry & Effect in the reduced problem \\
\hline
\(q\equiv1\) & \(V_q\equiv0\) & large geodesic balls; all directions active & baseline value \(\mathfrak g(0)=2/(n-1)\) \\
Axial perturbation \(q_\eps(\xi)=\exp[-\eps(1-(e\cdot\xi)^2)]\) & \(\eps(1-(e\cdot\theta)^2)+O(\eps^2)\) & ball with an axial second-order flattening; contacts include \(\pm e\) & decreases \(\mathfrak g\) to first order \\
Envelope generated by finitely many active directions & maximum of finitely many horospherical affine profiles & nonsmooth limiting horoconvex body with finite contact set & natural candidate type for constrained minimizers \\
\hline
\end{tabular}
\end{center}

\begin{lemma}[An admissible axial curve]\label{lem:axial-admissible-curve}
Let \(e\in\mathbb R^n\) be a unit vector and set
\[
  W(\theta)=1-(e\cdot\theta)^2.
\]
For sufficiently small \(\eps>0\), the profile
\[
  q_\eps(\xi)=\exp[-\eps W(-\xi)]
\]
is normalized in the sense of Definition~\ref{def:admissible-class}, and its envelope satisfies
\begin{equation}\label{eq:axial-envelope-expansion}
  V_{q_\eps}(\theta)=\eps W(\theta)+O(\eps^2)
\end{equation}
uniformly on \(\Sn\).
\end{lemma}

\begin{proof}
For \(0<\eps<\log2\), one has \(1/2\le q_\eps\le1\).  We check the normalization.  If \(\theta=e\), the required contact inequality is
\[
  e^{-\eps(1-t^2)}\ge \frac{1-t}{2},\qquad t=e\cdot\xi\in[-1,1].
\]
It holds for all sufficiently small \(\eps\): near \(t=-1\) the right side has slope \(-1/2\) in \(t\), while the left side has slope \(-2\eps\), which is larger than \(-1/2\) for small \(\eps\).  Away from \(t=-1\) the inequality is strict at \(\eps=0\).  The same argument with \(t\) replaced by \(-t\) gives \(-e\in C(q_\eps)\).  Hence \(0\in\operatorname{conv}\{e,-e\}\subset\operatorname{conv}C(q_\eps)\).

For the envelope expansion, write
\[
  V_{q_\eps}(\theta)
  =\sup_{\xi\in\Sn}
  \left\{\log\frac{1-\theta\cdot\xi}{2}+\eps W(-\xi)\right\}.
\]
At \(\eps=0\), the function \(\xi\mapsto\log((1-\theta\cdot\xi)/2)\) has the unique maximum \(0\) at \(\xi=-\theta\), with a negative definite Hessian in tangent directions.  Uniformly in \(\theta\), the maximizer for small \(\eps\) therefore remains \(O(\eps)\)-close to \(-\theta\).  Evaluating at \(-\theta\) gives the lower bound \(\eps W(\theta)\), and the quadratic loss of the unperturbed maximum gives an upper error \(O(\eps^2)\).  This proves \eqref{eq:axial-envelope-expansion}.

\end{proof}

\begin{lemma}[Uniform modulus for radial envelopes]\label{lem:modulus}
Every \(V\in\calA_n\) satisfies
\begin{equation}\label{eq:V-bounds}
  0\le V\le\log2.
\end{equation}
Moreover the family \(\calA_n\) is uniformly Lipschitz on \(\Sn\), with a constant depending only on
\(n\).
\end{lemma}

\begin{proof}
The upper bound follows from \(1-\theta\cdot\xi\le2\) and \(q\ge1/2\).  The lower bound is intrinsic in the envelope representation: choosing \(\xi=-\theta\) gives
\[
  \frac{1-\theta\cdot(-\theta)}{2q(-\theta)}=\frac1{q(-\theta)}\ge1,
\]
so the supremum in \eqref{eq:Vq} is nonnegative.

For the Lipschitz estimate, fix \(\theta\).  A maximizing direction \(\xi\) in \eqref{eq:Vq} may be
chosen with
\[
  \log\frac{1-\theta\cdot\xi}{2q(\xi)}\ge0.
\]
Since \(q(\xi)\ge1/2\), this implies \(1-\theta\cdot\xi\ge1\), hence \(\theta\cdot\xi\le0\).  On the
closed set \(\{(\theta,\xi):\theta\cdot\xi\le0\}\), the gradient of
\(\log(1-\theta\cdot\xi)\) in the \(\theta\)-variable is uniformly bounded.  A supremum of functions
with a common Lipschitz constant is Lipschitz with the same constant.
\end{proof}

For later use we isolate the uniform class on which the collar and Grushin estimates are valid.
For \(K<\infty\), set
\begin{equation}\label{eq:C-K-def}
  \mathfrak C_K:=
  \{V\in C^{0,1}(\Sn):0\le V\le\log2,\ \Lip(V)\le K\}.
\end{equation}

\begin{lemma}[Admissible Lipschitz hull]\label{lem:admissible-lipschitz-hull}
There is a constant \(K_n<\infty\), depending only on the dimension, such that
\(\calA_n\subset\mathfrak C_{K_n}\).  Moreover each \(\mathfrak C_K\) is compact in
\(C(\Sn)\), and convergence in \(\mathfrak C_K\) will always mean uniform convergence.
\end{lemma}

\begin{proof}
The inclusion \(\calA_n\subset\mathfrak C_{K_n}\) is exactly
Lemma~\ref{lem:modulus}.  The set \(\mathfrak C_K\) is closed in \(C(\Sn)\): if
\(V_j\to V\) uniformly and \(\Lip(V_j)\le K\), then
\(|V(\theta)-V(\theta')|\le Kd_{\Sn}(\theta,\theta')\) after passing to the limit, and the bounds
\(0\le V\le\log2\) also pass to the limit.  Equicontinuity and uniform boundedness give compactness by
Arzela--Ascoli.  The same compactness statement holds, with constants depending also on a fixed
height bound \(A\), for the variant \(0\le V\le A\).  We shall use this fixed-height variant for
finite-\(R\) radial deficits, which satisfy \(0\le V\le\log2+o(1)\) before passing to the limit.
\end{proof}

The finite-radius version of the envelope estimate is the uniform graph control used in all later spectral estimates.
\begin{lemma}[Finite-radius graph control]\label{lem:finite-R-lipschitz}
There are constants \(R_0(n)\), \(K_n<\infty\), and \(A_n<\infty\) with the following property.  Let
\(K\) be a Chebyshev-centered horoconvex body with circumradius \(R\ge R_0\), normalized support
profile \(q_R\), radial function \(\rho_R\), and radial deficit \(V_R=R-\rho_R\).  Then
\begin{equation}\label{eq:finite-R-graph-control}
  0\le V_R\le A_n,
  \qquad
  \Lip(V_R)\le K_n .
\end{equation}
In particular every divergent Chebyshev-centered horoconvex sequence is, after discarding finitely
many terms, a sequence of radial graph domains whose deficits lie in one compact Lipschitz family
of the fixed-height type.
\end{lemma}

\begin{proof}
The boundary value \(W=V_R(\theta)\) is characterized by
\begin{equation}\label{eq:finite-R-root-equation}
  F_R(W,\theta)=1,
\end{equation}
where
\[
  F_R(W,\theta):=
  \sup_{\xi\in\Sn}
  \left\{
  \frac{e^{-W}}{2q_R(\xi)}(1-\theta\cdot\xi)
  +\frac{e^{-2R+W}}{2q_R(\xi)}(1+\theta\cdot\xi)
  \right\}.
\]
The support bounds \(1/2\le q_R\le1\) imply the height estimate.  Since \(K\subset B_R(o)\), one
has \(V_R\ge0\).  Conversely, at \(W=\log2+C e^{-2R}\), the first term is at most
\(2e^{-W}\le1-C'e^{-2R}\), while the second term is \(O(e^{-2R})\).  Choosing \(C\) large gives
\(F_R(W,\theta)\le1\), hence \(V_R\le\log2+Ce^{-2R}\).  A fixed height bound \(A_n\) follows for all
large \(R\).

It remains to prove the Lipschitz bound.  On the strip \(0\le W\le A_n+1\), the functions
\(F_R(W,\cdot)\) are uniformly Lipschitz on \(\Sn\), since \(q_R\ge1/2\) and
\(|(\theta-\theta')\cdot\xi|\le d_{\Sn}(\theta,\theta')\).  Thus
\begin{equation}\label{eq:FR-theta-Lip}
  |F_R(W,\theta)-F_R(W,\theta')|\le C d_{\Sn}(\theta,\theta')
\end{equation}
with \(C\) independent of \(R\).

Near the level set \(F_R=1\), the dependence on \(W\) is uniformly transverse.  More explicitly,
for \(0\le W\le A_n+1\) and small \(h\ge0\), the first term in \(F_R\) is multiplied by \(e^{-h}\)
when \(W\) is replaced by \(W+h\), while the second term is multiplied by \(e^h\) and is
\(O(e^{-2R})\) on the whole height strip.  Hence, whenever \(F_R(W,\theta)=1\),
\[
  F_R(W+h,\theta)\le F_R(W,\theta)-c h+O(h^2)+O(e^{-2R}h),
\]
\[
  F_R(W-h,\theta)\ge F_R(W,\theta)+c h-O(h^2)-O(e^{-2R}h),
\]
with \(c>0\) independent of \(R\).  Taking \(h\) small and then \(R\) large gives the uniform
one-sided inequalities
\[
  F_R(W+h,\theta)\le1-\frac c2h,
  \qquad
  F_R(W-h,\theta)\ge1+\frac c2h .
\]
Combining this transversality with
\eqref{eq:FR-theta-Lip} and the root equation \eqref{eq:finite-R-root-equation} gives
\[
  |V_R(\theta)-V_R(\theta')|\le Cc^{-1}d_{\Sn}(\theta,\theta'),
\]
which proves \eqref{eq:finite-R-graph-control}.  Compactness of the fixed-height family follows from
Arzela--Ascoli, as in Lemma~\ref{lem:admissible-lipschitz-hull}.
\end{proof}

\begin{proposition}[Radial compactness]\label{prop:radial-compactness}
Let \(K_j\) be Chebyshev-centered horoconvex bodies with circumradii \(R_j\to\infty\).  After passing
to a subsequence, there is \(V\in\calA_n\) such that
\[
  \|V_{R_j}-V\|_{C(\Sn)}\to0.
\]
After discarding finitely many terms, the finite-\(R_j\) deficits belong to one fixed compact
Lipschitz graph family with a fixed height bound.
Moreover
\begin{equation}\label{eq:diam-2R}
  \diam K_j=2R_j+O(1).
\end{equation}
\end{proposition}

\begin{proof}
By Lemma~\ref{lem:finite-R-lipschitz}, after discarding finitely many terms the deficits \(V_{R_j}\)
lie in one fixed-height compact Lipschitz family.  Arzela--Ascoli gives a uniformly convergent
subsequence, say \(V_{R_j}\to V\).  Formula \eqref{eq:beta-asymp} and the bounds \eqref{eq:q-bounds}
give the envelope identity
\begin{equation}\label{eq:VR-envelope}
  V_{R_j}(\theta)=\sup_{\xi\in\Sn}\log\frac{1-\theta\cdot\xi}{2q_{R_j}(\xi)}+O(e^{-2R_j})
\end{equation}
uniformly in \(\theta\).  The limiting support profile is obtained by canonical duality: set
\[
  q_V(\xi):=\frac12\sup_{\theta\in\Sn}e^{-V(\theta)}(1-\theta\cdot\xi).
\]
The envelope stability estimate of Lemma~\ref{lem:canonical-dual-support} gives
\(V_{q_V}=V\).  The finite Chebyshev contact condition passes to the limit as follows.  For each
\(R\), choose a finite balanced contact family in \(C(q_R)\) by Caratheodory's theorem.  Passing to a
subsequence of these finite families gives a balanced family in the zero set of \(V\), hence in
\(C(q_V)\).  Thus \(q_V\) satisfies the limiting normalization \eqref{eq:limiting-cheb-normalization},
and \(V\in\calA_n\).

The upper diameter bound is \(\diam K_j\le2R_j\).  For the lower bound, the contact condition gives
two contact directions \(\theta_1,\theta_2\in C_{R_j}\) with \(\theta_1\cdot\theta_2\le0\).  Otherwise
all contact directions would lie in an open hemisphere.  The hyperbolic law of cosines yields
\[
  \cosh d=\cosh^2R_j-\sinh^2R_j\,\theta_1\cdot\theta_2
  \ge \cosh^2R_j,
\]
and hence \(d\ge 2R_j-O(1)\).  Thus \(\diam K_j=2R_j+O(1)\).
\end{proof}

\subsection{Regularization and Mosco stability}

There are two separate approximation issues.  The geometry produces radial deficits that are only
Lipschitz, while the boundary-Grushin calculation in Section~\ref{sec:effective-angular} is first
written for smooth graph functions.  The next lemmas provide the quantitative passage to admissible Lipschitz limits and also fix a canonical support profile for each admissible deficit.

\begin{lemma}[Canonical dual support profile]\label{lem:canonical-dual-support}
Let \(V\in\calA_n\).  Define
\begin{equation}\label{eq:qV-canonical}
  q_V(\xi)=\frac12\sup_{\theta\in\Sn} e^{-V(\theta)}(1-\theta\cdot\xi).
\end{equation}
Then \(q_V\in C(\Sn)\),
\begin{equation}\label{eq:qV-bounds}
  \frac12\le q_V\le1,
\end{equation}
\(q_V\) satisfies the limiting Chebyshev normalization, and
\begin{equation}\label{eq:biconjugacy}
  V(\theta)=\sup_{\xi\in\Sn}\log\frac{1-\theta\cdot\xi}{2q_V(\xi)}.
\end{equation}
Moreover, if \(q,\widetilde q:\Sn\to[1/2,1]\), then their envelopes satisfy
\begin{equation}\label{eq:envelope-stability}
  \|V_q-V_{\widetilde q}\|_{C(\Sn)}
  \le \|\log q-\log\widetilde q\|_{C(\Sn)}
  \le 2\|q-\widetilde q\|_{C(\Sn)}.
\end{equation}
\end{lemma}

\begin{proof}
Continuity of \(q_V\) follows because it is the supremum, over a compact parameter set, of the
continuous family
\((\theta,\xi)\mapsto e^{-V(\theta)}(1-\theta\cdot\xi)/2\).  Since \(0\le V\le\log2\), choosing
\(\theta=-\xi\) gives \(q_V(\xi)\ge e^{-V(-\xi)}\ge1/2\), while
\(e^{-V(\theta)}(1-\theta\cdot\xi)/2\le1\) gives \(q_V\le1\).

Let \(q\) be any support profile representing \(V\) in Definition~\ref{def:admissible-class}.  The
identity \(V=V_q\) implies, for every \(\theta,\xi\),
\[
  \log\frac{1-\theta\cdot\xi}{2q(\xi)}\le V(\theta),
\]
or equivalently
\[
  q(\xi)\ge \frac12e^{-V(\theta)}(1-\theta\cdot\xi).
\]
Taking the supremum over \(\theta\) gives \(q\ge q_V\).  Hence
\(V_{q_V}\ge V_q=V\).  Conversely, the definition of \(q_V\) gives
\[
  \frac{1-\theta\cdot\xi}{2q_V(\xi)}\le e^{V(\theta)}
\]
for every \(\theta,\xi\), and therefore \(V_{q_V}\le V\).  This proves
\eqref{eq:biconjugacy}.

It remains to check that the canonical representative preserves the Chebyshev normalization.  Let
\(q\) be a normalized representative of \(V\), and let
\[
  C_q:=\{\theta:\ q(\xi)\ge (1-\theta\cdot\xi)/2\text{ for every }\xi\in\Sn\}.
\]
By normalization, \(0\in\operatorname{conv}C_q\).  If \(\theta\in C_q\), then the envelope formula gives
\(V(\theta)\le0\), while Lemma~\ref{lem:modulus} gives \(V\ge0\).  Thus \(V(\theta)=0\).  For every
zero of \(V\), the definition of \(q_V\) gives
\[
  q_V(\xi)\ge \frac12(1-\theta\cdot\xi),\qquad \xi\in\Sn.
\]
Hence \(C_q\) is contained in the contact set of \(q_V\), and that contact set also has convex hull
containing the origin.

For stability, write
\[
  V_q(\theta)=\sup_\xi \left\{\log\frac{1-\theta\cdot\xi}{2}-\log q(\xi)\right\}.
\]
The difference of two suprema is bounded by the supremum of the difference of the functions being
maximized.  Since \(q,\widetilde q\ge1/2\), the logarithm is \(2\)-Lipschitz on the relevant
interval.
\end{proof}

\begin{lemma}[Closedness of the support-envelope class]\label{lem:admissible-closed}
Let \(V_j\in\calA_n\) and suppose \(V_j\to V\) uniformly on \(\Sn\).  Then \(V\in\calA_n\).
In particular \(\calA_n\) is compact in \(C(\Sn)\).
\end{lemma}

\begin{proof}
By Lemma~\ref{lem:admissible-lipschitz-hull}, the limit satisfies \(0\le V\le\log2\) and has the
same universal Lipschitz bound.  For each \(j\), replace the representing support profile by its
canonical representative \(q_j=q_{V_j}\) from Lemma~\ref{lem:canonical-dual-support}.  The defining estimate for \(q_V\) gives
\[
  \|q_j-q_V\|_{C(\Sn)}
  \le \sup_{\theta}|e^{-V_j(\theta)}-e^{-V(\theta)}|
  \le \|V_j-V\|_{C(\Sn)}\to0 .
\]
The envelope stability estimate \eqref{eq:envelope-stability} then yields
\[
  V_{q_V}=\lim_{j\to\infty}V_{q_j}=\lim_{j\to\infty}V_j=V
  \quad\text{uniformly.}
\]
It remains to check the limiting Chebyshev normalization for \(q_V\).  Let
\[
  Z_j:=\{\theta:\ V_j(\theta)=0\}.
\]
In the proof of Lemma~\ref{lem:canonical-dual-support} we showed that the contact set of \(q_j\)
has convex hull containing the origin and is contained in \(Z_j\).  By Caratheodory's theorem, for
each \(j\) there are \(N\le n+1\), points \(\theta_{j,1},\ldots,\theta_{j,N}\in Z_j\), and weights
\(a_{j,i}\ge0\), \(\sum_i a_{j,i}=1\), such that \(\sum_i a_{j,i}\theta_{j,i}=0\).  Passing to a
subsequence gives \(a_{j,i}\to a_i\) and \(\theta_{j,i}\to\theta_i\).  Uniform convergence implies
\(V(\theta_i)=0\), and \(\sum_i a_i\theta_i=0\).  For every zero \(\theta\) of \(V\), the definition
of \(q_V\) gives
\[
  q_V(\xi)\ge\frac{1-\theta\cdot\xi}{2},\qquad \xi\in\Sn.
\]
Thus the points \(\theta_i\) lie in \(C(q_V)\), and \(0\in\operatorname{conv}C(q_V)\).  Hence
\(q_V\) is normalized and \(V=V_{q_V}\in\calA_n\).

The compactness follows because \(\calA_n\subset\mathfrak C_{K_n}\) by
Lemma~\ref{lem:admissible-lipschitz-hull}, \(\mathfrak C_{K_n}\) is compact in \(C(\Sn)\), and the
preceding argument shows that \(\calA_n\) is closed.
\end{proof}

\begin{lemma}[Balanced almost-contacts control the circumradius]\label{lem:balanced-almost-contacts}
Let \(\theta_1,\ldots,\theta_m\in\Sn\) and let \(a_i\ge0\), \(\sum_i a_i=1\), satisfy
\[
  \sum_{i=1}^m a_i\theta_i=0.
\]
Let \(x_i=(r_i,\theta_i)\in\Hn\) with \(R-\varepsilon\le r_i\le R\), where
\(0<\varepsilon\le1\).  If a hyperbolic ball of radius \(\rho\) contains all the points \(x_i\),
then
\begin{equation}\label{eq:balanced-circumradius-lower}
  \rho\ge R-C\varepsilon,
\end{equation}
where \(C\) depends only on the chosen weights.  Consequently, if a compact set is contained in
\(\overline B_R(o)\) and contains such a balanced family of points, its circumradius is
\(R+O(\varepsilon)\).
\end{lemma}

\begin{proof}
Use the hyperboloid model.  Write
\[
  X_i=(\cosh r_i,\sinh r_i\,\theta_i),
  \qquad
  C=(\cosh t,\sinh t\,u)
\]
for the points and for an arbitrary candidate center.  Let \(\overline X=\sum_i a_iX_i\).  Since
\(r_i=R+O(\varepsilon)\) and \(\sum_i a_i\theta_i=0\),
\[
  \overline X_0=\cosh R+O(\varepsilon e^R),
  \qquad
  |\overline X_{\rm sp}|\le C\varepsilon e^R.
\]
Hence its Lorentz length satisfies
\[
  \sqrt{\overline X_0^2-|\overline X_{\rm sp}|^2}
  \ge \cosh R-C\varepsilon e^R
  \ge \cosh(R-C'\varepsilon).
\]
The reverse Cauchy--Schwarz inequality in Minkowski space gives, for every future unit timelike
\(C\),
\[
  \sum_i a_i\cosh d(C,x_i)
  =-\langle C,\overline X\rangle_L
  \ge \sqrt{\overline X_0^2-|\overline X_{\rm sp}|^2}.
\]
Therefore \(\max_i d(C,x_i)\ge R-C'\varepsilon\).  Taking the infimum over centers gives
\eqref{eq:balanced-circumradius-lower}.  The final assertion follows because the original center
\(o\) gives the opposite bound \(\rho\le R\).
\end{proof}

\begin{lemma}[Horoconvex realization of an admissible deficit]\label{lem:realization}
Let \(V=V_q\in\calA_n\) with \(q:\Sn\to[1/2,1]\).  For \(R\gg1\), set
\begin{equation}\label{eq:KRq}
  K_R(q)=\bigcap_{\xi\in\Sn}\{x:\beta_\xi(x)\le e^Rq(\xi)\}.
\end{equation}
Let \(\rho_R\) be its radial function and \(V_R=R-\rho_R\).  Then
\begin{equation}\label{eq:realization-error}
  \|V_R-V\|_{C(\Sn)}\le C e^{-2R},
\end{equation}
where \(C\) depends only on the universal bound \(0\le V\le\log2\).  If \(q\) is chosen with the
limiting Chebyshev normalization in Definition~\ref{def:admissible-class}, then the circumradius of
\(K_R(q)\) is \(R+O(e^{-2R})\) and
\begin{equation}\label{eq:realization-diameter}
  \diam K_R(q)=2R+O(1).
\end{equation}
\end{lemma}

\begin{proof}
For fixed \(\theta\), the boundary point \(r=R-W\) is determined by
\[
  \sup_\xi
  \left\{
  \frac{e^{-W}}{2q(\xi)}(1-\theta\cdot\xi)
  +\frac{e^{-2R+W}}{2q(\xi)}(1+\theta\cdot\xi)
  \right\}=1,
\]
which is just \eqref{eq:beta-asymp} divided by \(q(\xi)\).  The first term alone has level one at
\(W=V(\theta)\), by the envelope identity.  The second term is bounded by \(2e^{-2R+W}\le4e^{-2R}\).
Increasing \(W\) by \(Ce^{-2R}\) decreases the first term by a factor
\(1-Ce^{-2R}+O(e^{-4R})\), while decreasing \(W\) by the same amount increases it by the same order.
Choosing \(C\) large enough traps the exact root between \(V(\theta)-Ce^{-2R}\) and
\(V(\theta)+Ce^{-2R}\), uniformly in \(\theta\).  This proves \eqref{eq:realization-error}.

If the limiting Chebyshev normalization holds, the zero set of \(V\) has convex hull containing the
origin.  By Caratheodory's theorem, choose finitely many zero directions
\(\theta_1,\ldots,\theta_m\) and weights \(a_i\ge0\), \(\sum_i a_i=1\), such that
\(\sum_i a_i\theta_i=0\).  For these directions, \eqref{eq:realization-error} gives boundary points
\(x_i=(r_i,\theta_i)\in K_R(q)\) with \(R-Ce^{-2R}\le r_i\le R\).  Since \(q\le1\), the body is
contained in \(\overline B_R(o)\).  Lemma~\ref{lem:balanced-almost-contacts}, with
\(\varepsilon=Ce^{-2R}\), shows that the circumradius is at least \(R-Ce^{-2R}\), while the chosen
center gives the upper bound \(R\).  Hence the circumradius is \(R+O(e^{-2R})\).

For the diameter, the same balanced finite contact set contains two directions with
\(\theta_i\cdot\theta_j\le0\); otherwise all its points would lie in an open hemisphere and their
convex hull could not contain the origin.  The corresponding boundary points have radii
\(R+O(e^{-2R})\), and the hyperbolic law of cosines gives distance at least \(2R-O(1)\).  The
opposite inequality follows from \(K_R(q)\subset B_R(o)\).  Thus \(\diam K_R(q)=2R+O(1)\).

\end{proof}

\begin{lemma}[Asymptotic diameter and exact calibration]\label{lem:diameter-calibration}
Let
\[
  \Omega_{R,W}=\{(r,\theta):0\le r\le R-W(\theta)\},
  \qquad 0\le W\le\log2,
\]
where the functions \(W\) belong to a fixed compact Lipschitz family.  Define
\begin{equation}\label{eq:DeltaW}
  \Delta(W):=\sup_{\theta,\phi\in\Sn}
  \left\{-W(\theta)-W(\phi)+\log\frac{1-\theta\cdot\phi}{2}\right\}.
\end{equation}
Then, uniformly on that family,
\begin{equation}\label{eq:diam-radial-graph-asymptotic}
  \diam\Omega_{R,W}=2R+\Delta(W)+O(e^{-2R}).
\end{equation}
Moreover
\begin{equation}\label{eq:Delta-Lipschitz}
  |\Delta(W)-\Delta(\widetilde W)|\le2\|W-\widetilde W\|_{C(\Sn)}.
\end{equation}

Let \(V\in\calA_n\), take the canonical support profile \(q=q_V\), and let \(K_R(q)\) be the
horoconvex support realization from Lemma~\ref{lem:realization}.  Then
\begin{equation}\label{eq:diam-support-calibration}
  \diam K_R(q)=2R+\Delta(V)+O(e^{-2R}).
\end{equation}
The map \(R\mapsto \diam K_R(q)\) is continuous, and for every sufficiently large \(D\) there exists
\(R_D\) such that
\begin{equation}\label{eq:exact-diameter-RD}
  \diam K_{R_D}(q)=D,
  \qquad
  R_D=\frac{D-\Delta(V)}2+O(e^{-D}).
\end{equation}
The constants in the error terms are uniform for \(V\) in compact admissible subfamilies.
\end{lemma}

\begin{proof}
We first prove \eqref{eq:diam-radial-graph-asymptotic}.  For two polar points
\(x=(r,\theta)\), \(y=(s,\phi)\), write \(a=\theta\cdot\phi\).  The hyperbolic law of cosines gives
\begin{equation}\label{eq:hyperbolic-distance-expanded}
\begin{aligned}
  \cosh d(x,y)
  &=\frac{e^{r+s}}{4}(1-a)
   +\frac{e^{r-s}+e^{-r+s}}{4}(1+a)
   +\frac{e^{-r-s}}{4}(1-a).
\end{aligned}
\end{equation}
The diameter is at least \(2R-C\), uniformly in \(W\), because for any \(\theta\) the antipodal
boundary pair
\((R-W(\theta),\theta)\), \((R-W(-\theta),-\theta)\) has distance
\(2R-W(\theta)-W(-\theta)+O(e^{-2R})\ge2R-2\log2+O(e^{-2R})\).  Therefore a maximizing pair may be
assumed to have distance at least \(2R-C\).  For such a pair, \eqref{eq:hyperbolic-distance-expanded}
forces both radial variables to be \(R-O(1)\) and \(1-a\ge c>0\), with constants independent of
\(W\).  On this set the first term in \eqref{eq:hyperbolic-distance-expanded} dominates, and hence
\begin{equation}\label{eq:distance-asymptotic}
  d(x,y)=r+s+\log\frac{1-a}{2}+O(e^{-2R}).
\end{equation}
Writing \(r=R-A\), \(s=R-B\), where \(A\ge W(\theta)\) and \(B\ge W(\phi)\), gives
\[
  d(x,y)
  \le 2R-W(\theta)-W(\phi)+\log\frac{1-\theta\cdot\phi}{2}+O(e^{-2R})
  \le 2R+\Delta(W)+O(e^{-2R}).
\]
This proves the upper bound.  The lower bound follows by taking boundary points
\(r=R-W(\theta)\), \(s=R-W(\phi)\) for a pair \((\theta,\phi)\) maximizing \(\Delta(W)\).  Since
\(\Delta(W)\ge -2\log2\), every maximizing pair has \(1-\theta\cdot\phi\ge c>0\), and
\eqref{eq:distance-asymptotic} applies.  The Lipschitz estimate \eqref{eq:Delta-Lipschitz} is
immediate from the definition of \(\Delta\).

For \(K_R(q)\), let \(W_R=R-\rho_R\) be the radial deficit with respect to the fixed center.  By
Lemma~\ref{lem:realization}, \(W_R=V+O(e^{-2R})\) uniformly.  Combining this with
\eqref{eq:diam-radial-graph-asymptotic} and \eqref{eq:Delta-Lipschitz} gives
\eqref{eq:diam-support-calibration}.  The same boundary equation as in Lemma~\ref{lem:realization}
shows that \(W_R\) depends continuously on \(R\); equivalently, the compact sets \(K_R(q)\) vary
continuously in the Hausdorff metric.  Therefore \(R\mapsto\diam K_R(q)\) is continuous.

Set \(R_0=(D-\Delta(V))/2\).  From \eqref{eq:diam-support-calibration},
\(\diam K_R(q)=2R+\Delta(V)+E(R)\), with \(|E(R)|\le Ce^{-2R}\).  Taking
\(R_\pm=R_0\pm A e^{-D}\), with \(A\) larger than the preceding constant, gives
\(\diam K_{R_-}(q)<D<\diam K_{R_+}(q)\) for all large \(D\).  The intermediate value theorem yields
\eqref{eq:exact-diameter-RD}.
\end{proof}

\begin{theorem}[Large-radius compactification of horoconvex deficits]\label{thm:geometry-package}
The support-envelope class \(\calA_n\) has the following equivalent descriptions and stability
properties.
\begin{enumerate}[label=\textup{(\roman*)},leftmargin=2.5em]
\item \emph{Compactness from genuine horoconvex bodies.}  If \(K_j\) are Chebyshev-centered
horoconvex bodies with circumradii \(R_j\to\infty\), then, after passing to a subsequence, their
radial deficits \(V_j=R_j-\rho_{K_j}\) converge uniformly to some \(V\in\calA_n\).  The sequence
\(V_j\) lies, after discarding finitely many terms, in one fixed compact Lipschitz graph family, and
\(\diam K_j=2R_j+O(1)\).

\item \emph{Closed intrinsic model.}  Conversely, \(\calA_n\) is compact in \(C(\Sn)\), each
\(V\in\calA_n\) has the canonical normalized support profile
\[
  q_V(\xi)=\frac12\sup_{\theta\in\Sn}e^{-V(\theta)}(1-\theta\cdot\xi),
\]
and the biconjugacy identity
\[
  V(\theta)=\sup_{\xi\in\Sn}\log\frac{1-\theta\cdot\xi}{2q_V(\xi)}
\]
holds.  Uniform convergence of deficits is equivalent, after canonicalization, to uniform
convergence of the support profiles in the sense recorded in \eqref{eq:envelope-stability}.

\item \emph{Horoconvex realization.}  For every \(V\in\calA_n\), the horoconvex bodies
\[
  K_R(q_V)=\bigcap_{\xi\in\Sn}\{x:\beta_\xi(x)\le e^R q_V(\xi)\}
\]
have radial deficits \(V_R\) satisfying
\[
  \|V_R-V\|_{C(\Sn)}\le Ce^{-2R},
\]
and their true Chebyshev circumradii equal \(R+O(e^{-2R})\).

\item \emph{Exact diameter calibration.}  With \(\Delta(V)\) as in \eqref{eq:DeltaW},
\[
  \diam K_R(q_V)=2R+\Delta(V)+O(e^{-2R}).
\]
Consequently, for every sufficiently large prescribed diameter \(D\) there is an \(R_D\) such that
\[
  \diam K_{R_D}(q_V)=D,
  \qquad
  R_D=\frac{D-\Delta(V)}2+O(e^{-D}).
\]
\end{enumerate}
All constants above are uniform when \(V\) is restricted to a compact subset of \(\calA_n\).
Thus Definition~\ref{def:admissible-class} is exactly the large-radius compactification of
Chebyshev-centered horoconvex bodies at the level relevant to the \(R^{-3}\) spectral asymptotics.
\end{theorem}

\begin{proof}
Part (i) is Proposition~\ref{prop:radial-compactness} together with the finite-radius Lipschitz
control in Lemma~\ref{lem:finite-R-lipschitz}.  Part (ii) is the canonical-dual construction in
Lemma~\ref{lem:canonical-dual-support} and the closedness statement in
Lemma~\ref{lem:admissible-closed}.  Part (iii) is Lemma~\ref{lem:realization}; the estimate on the
true Chebyshev circumradius uses Lemma~\ref{lem:balanced-almost-contacts}.  Part (iv) is exactly
Lemma~\ref{lem:diameter-calibration}.  The uniformity on compact subfamilies follows from the common
height and Lipschitz bounds in \(\mathfrak C_{K_n}\) and from the uniform forms of the estimates in
those lemmas.
\end{proof}

\begin{remark}[Use in the rest of the paper]\label{rem:geometry-package-use}
From this point on the asymptotic shape is represented by the continuous deficit \(V\in\calA_n\) and, when an actual horoconvex body is
needed, it is represented by the canonical realization \(K_R(q_V)\).  The spectral analysis is carried out for radial
graphs with deficits in compact Lipschitz families; Theorem~\ref{thm:geometry-package} transfers the result
from those graph estimates back to genuine horoconvex domains and to the exact diameter constraint.
\end{remark}

\begin{lemma}[Smooth graph approximation]\label{lem:smooth-graph-approx}
For every \(V\in\calA_n\) there exist \(W_m\in C^\infty(\Sn)\) such that
\begin{equation}\label{eq:smooth-approx}
  0\le W_m\le\log2,
  \qquad
  \|W_m-V\|_{C(\Sn)}\to0,
\end{equation}
and the Lipschitz constants of \(W_m\) are bounded by the common Lipschitz constant in
Lemma~\ref{lem:modulus}.  If one wants support profiles at the same time, the canonical profiles
\(q_{W_m}\) defined by \eqref{eq:qV-canonical} satisfy \(q_{W_m}\to q_V\) uniformly.
\end{lemma}

\begin{proof}
Let \(P_t\) be the heat semigroup on \(\Sn\) and take \(W_m=P_{t_m}V\) with \(t_m\downarrow0\).
The heat semigroup preserves constants and positivity, so \(0\le W_m\le\log2\).  Since \(V\) is
continuous, \(P_tV\to V\) uniformly.  The heat flow is a contraction for the Lipschitz seminorm on
the sphere, giving the uniform Lipschitz bound.  Finally,
\[
  |q_{W_m}(\xi)-q_V(\xi)|
  \le \frac12\sup_\theta(1-\theta\cdot\xi)|e^{-W_m(\theta)}-e^{-V(\theta)}|
  \le \|W_m-V\|_\infty,
\]
so the canonical supports also converge uniformly.
\end{proof}

\begin{lemma}[Uniform heat regularization on compact Lipschitz families]\label{lem:uniform-heat-regularization}
Let \(\calB\subset\mathfrak C_K\) be compact in \(C(\Sn)\), and let
\(W_t(V):=P_tV\), where \(P_t\) is the heat semigroup on \(\Sn\).  Then, for every
\(t>0\), the image \(\calB_t:=\{W_t(V):V\in\calB\}\) is a compact family of smooth graph deficits contained in
\(\mathfrak C_K\), and
\begin{equation}\label{eq:uniform-heat-regularization}
  \sup_{V\in\calB}\|W_t(V)-V\|_{C(\Sn)}\longrightarrow0,
  \qquad t\downarrow0.
\end{equation}
The canonical supports also converge uniformly on the family:
\begin{equation}\label{eq:uniform-canonical-support-regularization}
  \sup_{V\in\calB}\|q_{W_t(V)}-q_V\|_{C(\Sn)}\longrightarrow0.
\end{equation}
\end{lemma}

\begin{proof}
For fixed \(t>0\), the heat semigroup maps \(C(\Sn)\) continuously into
\(C^\infty(\Sn)\); hence the image of the compact set \(\calB\) is compact.  Positivity and
constant preservation give \(0\le W_t(V)\le\log2\).  Since
\(\calB\subset\mathfrak C_K\), and the heat flow is a contraction for the Lipschitz seminorm on the sphere, each \(W_t(V)\) also belongs to
\(\mathfrak C_K\).

It remains only to justify that the convergence \(P_tV\to V\) is uniform in \(V\in\calB\).  Given
\(\varepsilon>0\), choose a finite \(\varepsilon/3\)-net \(V_1,\dots,V_N\) for \(\calB\) in
\(C(\Sn)\).  Since \(P_tV_i\to V_i\) for each \(i\), for all sufficiently small \(t\),
\(\|P_tV_i-V_i\|_\infty<\varepsilon/3\) for every \(i\).  The contraction property in
\(C(\Sn)\) then gives, for \(V\) within \(\varepsilon/3\) of some \(V_i\),
\[
  \|P_tV-V\|_\infty
  \le \|P_t(V-V_i)\|_\infty+\|P_tV_i-V_i\|_\infty+\|V_i-V\|_\infty
  <\varepsilon.
\]
This proves \eqref{eq:uniform-heat-regularization}.  The support convergence follows from the
stability estimate in Lemma~\ref{lem:smooth-graph-approx}, uniformly in \(V\).
\end{proof}

\begin{lemma}[Mosco convergence of radial graph domains]\label{lem:mosco-graphs}
Fix \(R>1\).  Suppose \(W_m,W\in C(\Sn)\) are uniformly Lipschitz, satisfy
\(0\le W_m,W\le\log2\), and \(W_m\to W\) uniformly.  Let
\[
  \Omega_{R,W_m}=\{(r,\theta):0<r<R-W_m(\theta)\},
  \qquad
  \Omega_{R,W}=\{(r,\theta):0<r<R-W(\theta)\}.
\]
Then \(\Omega_{R,W_m}\to\Omega_{R,W}\) in the Mosco sense for Dirichlet forms.  Consequently the
Dirichlet eigenvalues of the hyperbolic Laplacian, and equivalently of the Liouville-transformed
operator, converge for each fixed \(R\).
\end{lemma}

\begin{proof}
The domains are uniformly Lipschitz graph domains in the fixed annulus \(0<r<R\), and their radial
functions converge uniformly.  For the recovery condition, take first
\(u\in C_c^\infty(\Omega_{R,W})\).  Its support has positive distance from the boundary, so it is
contained in \(\Omega_{R,W_m}\) for all large \(m\); hence \(u_m=u\) is an admissible recovery
sequence.  Density of \(C_c^\infty(\Omega_{R,W})\) in \(H_0^1(\Omega_{R,W})\), valid for Lipschitz
domains, gives the recovery sequence for every \(u\in H_0^1\).

For the closedness condition, extend \(u_m\in H_0^1(\Omega_{R,W_m})\) by zero to a fixed ball
\(B_R(o)\), and suppose \(u_m\rightharpoonup u\) weakly in \(H^1(B_R(o))\).  If a compact set
\(K\) lies in the complement of \(\overline{\Omega_{R,W}}\), then uniform convergence of the radial
functions implies \(K\cap\Omega_{R,W_m}=\varnothing\) for all large \(m\), so \(u=0\) on \(K\).
Since the boundary graph has measure zero, \(u=0\) almost everywhere outside \(\Omega_{R,W}\), and
the trace criterion for Lipschitz domains gives \(u\in H_0^1(\Omega_{R,W})\).  This is Mosco
convergence.  Strong resolvent convergence of the associated Dirichlet operators follows from
Mosco convergence; compactness of the resolvents gives convergence of individual eigenvalues.
\end{proof}

\section{Liouville transform and scattering lengths}

In polar coordinates the hyperbolic metric is
\[
  ds^2=dr^2+\sinh^2r\,d\theta^2,
\]
and the volume element is \((\sinh r)^{n-1}\,dr\,d\theta\).  Let \(\alpha=(n-1)/2\).  The unitary map
\[
  u\mapsto v=(\sinh r)^\alpha u
\]
transforms the Dirichlet Laplacian, after subtracting \(\alpha^2\), into
\begin{equation}\label{eq:liouville}
  H=-\partial_r^2+\frac{-\Delta_{\Sn}+\alpha(\alpha-1)}{\sinh^2r}.
\end{equation}
On the spherical harmonic sector \(\calH_\ell\), this is
\begin{equation}\label{eq:H-ell}
  H_\ell=-\frac{d^2}{dr^2}+\frac{(\ell+\alpha)(\ell+\alpha-1)}{\sinh^2r}.
\end{equation}

\begin{lemma}[Regular endpoint and Friedrichs convention]\label{lem:friedrichs-endpoint}
The radial realization used throughout is the Friedrichs realization of \(H_\ell\) at \(r=0\).
Equivalently, if \(\beta=\ell+\alpha\), the regular branch is the one with leading behavior
\[
  y(r)=r^\beta(1+O(r^2))\qquad (r\downarrow0)
\]
when \(\beta\ne1/2\).  In the critical case \(n=2,\ell=0\), where \(\beta=1/2\) and
\[
  \frac{\beta(\beta-1)}{\sinh^2r}=-\frac1{4r^2}+O(1),
\]
the Friedrichs condition means
\[
  y(r)=r^{1/2}(1+O(r^2)),
\]
and the logarithmic branch \(r^{1/2}\log r\) is excluded.  With this convention the radial
operators are self-adjoint, bounded below, and have compact resolvent on every finite interval
\((0,R)\).
\end{lemma}

\begin{proof}
For \(\beta>1/2\) the second local solution behaves as \(r^{1-\beta}\), and the Friedrichs form
selects the branch obtained by closing compactly supported smooth functions away from the endpoint.
When \(\beta=1/2\), Frobenius theory gives the two local solutions \(r^{1/2}\) and
\(r^{1/2}\log r\).  The logarithmic solution is excluded by the Friedrichs form domain: after cutting it
off at scale \(\varepsilon\), its renormalized one-dimensional energy contains the divergent term
\(\int_\varepsilon r^{-1}\,dr\).  The closure of \(C_c^\infty(0,R)\) therefore selects the
\(r^{1/2}\) branch.  This is also exactly the branch obtained from a smooth hyperbolic eigenfunction
before the Liouville transform, since \((\sinh r)^{1/2}u(r)\sim r^{1/2}u(0)\) for the radial
\(n=2\) ground sector.  Standard regular-singular Sturm--Liouville theory then gives
self-adjointness, lower semiboundedness, and compact resolvent on \((0,R)\).
\end{proof}

\begin{lemma}[Zero-energy scattering length]\label{lem:scattering}
Let \(\beta>0\), and let \(y_\beta\) be the solution of
\begin{equation}\label{eq:zero-energy}
  -y''+\frac{\beta(\beta-1)}{\sinh^2r}y=0
\end{equation}
that is regular at \(r=0\) and normalized by \(y_\beta'(r)\to1\) as \(r\to\infty\).  Then
\begin{equation}\label{eq:scattering-b}
  y_\beta(r)=r+b(\beta)+O(re^{-2r}),
  \qquad b(\beta)=-\gamma_E-\psi(\beta).
\end{equation}
Consequently, for \(\beta=\ell+\alpha\),
\[
  b_\ell=b(\ell+\alpha)=-\gamma_E-\psi(\ell+\alpha).
\]
\end{lemma}

\begin{proof}
Put \(x=\tanh r\) and \(z=x^2\).  The regular solution at the origin is
\begin{equation}\label{eq:hypergeom-sol}
  \widetilde y_\beta(r)
  =x^\beta\,{}_2F_1\!\left(\frac\beta2,\frac{\beta+1}{2};\beta+\frac12;x^2\right).
\end{equation}
Indeed, after writing \(y=x^\beta F(x^2)\), equation \eqref{eq:zero-energy} becomes
\[
  z(1-z)F''+\left(\beta+\frac12-\left(\beta+\frac32\right)z\right)F'
  -\frac{\beta(\beta+1)}4F=0.
\]
Here the hypergeometric parameters are zero-balanced:
\[
  a=\frac\beta2,
  \qquad b=\frac{\beta+1}{2},
  \qquad c=a+b=\beta+\frac12.
\]
The zero-balanced connection formula gives, as \(z\uparrow1\),
\[
  {}_2F_1(a,b;a+b;z)
  =\frac1{B(a,b)}\{-\log(1-z)+2\psi(1)-\psi(a)-\psi(b)\}
  +O((1-z)\log(1-z)).
\]
Since \(1-z=\operatorname{sech}^2r=4e^{-2r}+O(e^{-4r})\),
\[
  \widetilde y_\beta(r)
  =\frac2{B(a,b)}\left[r-\log2-\gamma_E-\frac{\psi(a)+\psi(b)}2\right]
  +O(re^{-2r}).
\]
Multiplying by \(B(a,b)/2\) normalizes the derivative at infinity to one.  Finally, the duplication
identity for the digamma function,
\[
  \psi\!\left(\frac\beta2\right)+\psi\!\left(\frac{\beta+1}{2}\right)
  =2\psi(\beta)-2\log2,
\]
turns the constant term into \(-\gamma_E-\psi(\beta)\).
\end{proof}

\begin{lemma}[Uniform low-energy phase expansion]\label{lem:low-energy-phase}
Fix a compact interval \(I\subset[1/2,\infty)\) of parameters \(\beta\).  Let \(y_\beta\) denote the
zero-energy solution normalized in \cref{lem:scattering}.  For positive energy we use a different
normalization: \(\sigma_\beta(r,k)\) is the Friedrichs regular branch of
\[
  -\sigma''+\frac{\beta(\beta-1)}{\sinh^2r}\sigma=k^2\sigma
\]
whose exterior sine amplitude is one.  Equivalently, \(k^{-1}\sigma_\beta(\cdot,k)\to y_\beta\)
locally as \(k\downarrow0\).  There are \(k_0>0\) and constants depending only on \(I\) such that,
for \(0<k<k_0\) and \(1\le r\le c_0/k\),
\begin{equation}\label{eq:phase-expansion}
  \sigma_\beta(r,k)
  =\sin\bigl(k(r+b(\beta))\bigr)
   +O_I\bigl(k^3+k(1+r)e^{-2r}\bigr),
\end{equation}
and
\begin{equation}\label{eq:phase-error-bound}
  \partial_r\sigma_\beta(r,k)
  =k\cos\bigl(k(r+b(\beta))\bigr)
   +O_I\bigl(k^4+k(1+r)e^{-2r}\bigr).
\end{equation}
In particular, if \(k\sim R^{-1}\) and \(r=R+O(1)\), the boundary trace error is
\(O_I(R^{-3})\) and the derivative error is \(O_I(R^{-4})\), up to exponentially small terms.  The
endpoint \(\beta=1/2\), which occurs for \(n=2,\ell=0\), is included with the Friedrichs convention
of \cref{lem:friedrichs-endpoint}.
\end{lemma}

\begin{proof}
This is \cref{prop:appendix-phase}.  The compact interval \(I\) is used only after the angular
window has been fixed; hence all constants are allowed to depend on that window but not on \(R\).
\end{proof}

\begin{proposition}[First radial eigenvalue in one sector]\label{prop:sector-asymp}
Let \(\mu_{\ell,1}(R)\) be the first Dirichlet eigenvalue of \(H_\ell\) on \((0,R)\).  Then
\begin{equation}\label{eq:sector-asymp}
  \mu_{\ell,1}(R)=\frac{\pi^2}{(R+b_\ell)^2}+O_\ell(R^{-4})
  =\frac{\pi^2}{R^2}-\frac{2\pi^2b_\ell}{R^3}+O_\ell(R^{-4}).
\end{equation}
The estimate is uniform for \(\ell\) in bounded sets.
\end{proposition}

\begin{proof}
Let \(k^2=\mu_{\ell,1}(R)\), and put \(\beta=\ell+\alpha\).  The first root has
\(k\sim\pi/R\), for instance by comparison with the free interval on the outer part of the radial
operator and the boundedness from below of the finite inner part.  The zeros of the regular branch are independent of scalar normalization, so we use the
sine-amplitude branch \(\sigma_\beta\) of Lemma~\ref{lem:low-energy-phase}.  At \(r=R\),
\[
  0=\sigma_\beta(R,k)=\sin\bigl(k(R+b_\ell)\bigr)+O_\ell(R^{-3}),
\]
because \(k\sim R^{-1}\) and the exponential tail is negligible.  This is the first positive zero,
so \(k(R+b_\ell)\) lies in a fixed small neighborhood of \(\pi\), where the sine function has
nonzero derivative.  Hence
\[
  k(R+b_\ell)=\pi+O_\ell(R^{-3}).
\]
Therefore
\[
  k=\frac{\pi}{R+b_\ell}+O_\ell(R^{-4}),
\]
and squaring gives \eqref{eq:sector-asymp}.  The estimate is uniform on bounded \(T_n\)-windows
because such windows contain finitely many angular degrees and the constants in
Lemma~\ref{lem:low-energy-phase} are uniform there.
\end{proof}

For geodesic balls the first eigenvalue lies in the \(\ell=0\) sector and the second in the
\(\ell=1\) sector.  Since
\[
  b_0-b_1=\psi(\alpha+1)-\psi(\alpha)=\frac1\alpha,
\]
we obtain
\begin{equation}\label{eq:ball-gap}
  \Gap(B_R)=\frac{2\pi^2}{R^3}(b_0-b_1)+O(R^{-4})
  =\frac{4\pi^2}{(n-1)R^3}+O(R^{-4}).
\end{equation}

\begin{remark}[Sign conventions and constant deficits]\label{rem:sign-conventions}
The preceding constants fix the sign convention used throughout the effective reduction.  We record
three calibrations.

First, the formula is exact in the free case.  If \(\beta=1\), the potential in
\eqref{eq:zero-energy} vanishes, \(\psi(1)=-\gamma_E\), and hence \(b(1)=0\); the regular zero-energy
solution is \(y(r)=r\).  Thus the normalization in \eqref{eq:scattering-b} agrees with the ordinary
Dirichlet interval.

Second, a constant radial deficit has the expected sign.  If \(V\equiv a\), then
\(\Omega_{R,a}=B_{R-a}\) and \cref{prop:sector-asymp} gives
\begin{equation}\label{eq:constant-deficit-calibration}
  \mu_{\ell,1}(R-a)
  =\frac{\pi^2}{(R-a+b_\ell)^2}+O_\ell(R^{-4})
  =\frac{\pi^2}{R^2}
   +\frac{2\pi^2}{R^3}(a-b_\ell)+O_\ell(R^{-4}).
\end{equation}
Since \(T_nY=(b_0-b_\ell)Y\), the coefficient in \eqref{eq:constant-deficit-calibration} is
\[
  a-b_\ell=(T_n+a-b_0)\quad\text{on }\calH_\ell.
\]
Thus the effective operator must be \(T_n+V-b_0\).  In particular the boundary displacement enters
with a plus sign: moving the boundary inward by \(a>0\) raises the eigenvalue.

Third, this is the sign produced by the boundary Grushin normalization used below.  With
\[
  k_R(\zeta)=\frac\pi R+\frac{\pi\zeta}{R^2}
\]
and with constant deficit \(a\), the exterior phase formula gives
\[
  s_\ell(R-a,k_R(\zeta))
  =\sin\{k_R(\zeta)(R-a+b_\ell)\}+O(R^{-2})
  =-\frac\pi R\,(\zeta+b_\ell-a)+O(R^{-2}).
\]
Therefore, for
\[
  \mathfrak B_R(\zeta,V):=-\frac R\pi B_R^{\mathrm{raw}}(k_R(\zeta),V),
\]
we obtain on \(\calH_\ell\)
\[
  \mathfrak B_R(\zeta,a)
  =\zeta+b_\ell-a+O(R^{-1})
  =\zeta-(T_n+a-b_0)+O(R^{-1}).
\]
This matches both the scalar eigenvalue calibration and the operator expansion in
\cref{lem:boundary-grushin-expansion}.  As a further check, for \(a=0\) the first gap is
\[
  \frac{2\pi^2}{R^3}\bigl[(b_0-b_1)-0\bigr]
  =\frac{4\pi^2}{(n-1)R^3},
\]
which is \eqref{eq:ball-gap}.
\end{remark}

\section{The effective angular Hamiltonian}\label{sec:effective-angular}

This section proves the first-band reduction for radial graphs
\[
  \Omega_{R,V}=\{(r,\theta):0<r<R-V(\theta)\}.
\]
After the Liouville transform and subtraction of \(\alpha^2\), the relevant spectral window is
centered at
\[
  E_R:=\frac{\pi^2}{R^2},
  \qquad
  k_R(\zeta):=\frac\pi R+\frac{\pi\zeta}{R^2},
  \qquad
  k_R(\zeta)^2=E_R+\frac{2\pi^2}{R^3}\zeta+O(R^{-4}).
\]
The boundary phase of the \(\ell\)-th regular radial branch is
\[
  k_R(\zeta)\bigl(R-V(\theta)+b_\ell\bigr),
  \qquad
  b_\ell=-\gamma_E-\psi(\ell+\alpha).
\]
The sector-dependent part of this phase is the multiplier
\begin{equation}\label{eq:Tn}
  T_nY=\left[\psi(\ell+\alpha)-\psi(\alpha)\right]Y,
  \qquad Y\in\calH_\ell,
\end{equation}
where \(\alpha=(n-1)/2\).  Equivalently,
\begin{equation}\label{eq:Tn-b}
  T_nY=(b_0-b_\ell)Y,
  \qquad b_0=-\gamma_E-\psi(\alpha).
\end{equation}
Thus the effective operator is
\[
  A_V:=T_n+V-b_0.
\]

\subsection{Statement and boundary-phase calculation}

The following scalar computation fixes the normalization and the sign.

\begin{proposition}[Boundary phase of one sector]\label{prop:model-boundary-phase}
For fixed angular degree \(\ell\), fixed \(|\zeta|\le C\), and fixed endpoint displacement \(a=O(1)\),
\begin{equation}\label{eq:model-boundary-phase}
  s_\ell(R-a,k_R(\zeta))
  =-\frac\pi R(\zeta+b_\ell-a)+O_{\ell,C}(R^{-2}).
\end{equation}
Consequently, for the normalized boundary pencil
\[
  \mathfrak B_R(\zeta,V):=-\frac R\pi B_R^{\rm raw}(k_R(\zeta),V),
\]
the constant-deficit boundary equation is
\[
  \mathfrak B_R(\zeta,a)Y
  =\bigl[\zeta-(T_n+a-b_0)\bigr]Y+O_{\ell,C}(R^{-1}),
  \qquad Y\in\calH_\ell.
\]
\end{proposition}

\begin{proof}
The threshold expansion gives
\[
  s_\ell(r,k)=\sin\{k(r+b_\ell)\}+O(k^3)+O(k(1+r)e^{-2r})
\]
for \(r=R+O(1)\), \(k\sim R^{-1}\).  Since
\[
  k_R(\zeta)(R-a+b_\ell)
  =\pi+\frac\pi R(\zeta+b_\ell-a)+O_C(R^{-2}),
\]
expansion at \(\pi\) proves \eqref{eq:model-boundary-phase}.  The last identity follows from
\(T_nY=(b_0-b_\ell)Y\).
\end{proof}

For a graph endpoint \(a=V(\theta)\), multiplication by \(V\) couples angular sectors.  The fixed
angular-window calculation is
\begin{equation}\label{eq:section4-main-boundary-equation}
  P_L\mathfrak B_R(\zeta,V)P_L
  =P_L\bigl(\zeta I-(T_n+V-b_0)\bigr)P_L+o_R(1).
\end{equation}
This finite-window identity is the central reduction.  It shows that the moving Dirichlet
condition, after subtracting the common radial energy and scaling by \(R^3/(2\pi^2)\), is governed
by the effective Hamiltonian \(A_V=T_n+V-b_0\).  The rest of the section removes the angular cutoff
and passes from characteristic values of the boundary pencil to Dirichlet eigenvalues.  The second
radial root is separated from the first one by \(3\pi^2/R^2+o(R^{-2})\), so higher radial roots lie
outside the \(R^{-3}\) effective window once the boundary reduction is made.

\begin{theorem}[Effective first-band convergence]\label{thm:effective-band}\label{thm:feshbach}
Let \(\calB\subset\mathfrak C_K\) be compact.  For \(V\in\calB\), let \(H_{R,V}\) be the Liouville
transformed Dirichlet operator on \(\Omega_{R,V}\), with \(\alpha^2\) subtracted, and define
\[
  \zeta_{j,R}(V):=\frac{R^3}{2\pi^2}
  \left(\lambda_{j+1}(\Omega_{R,V})-\alpha^2-\frac{\pi^2}{R^2}\right),
  \qquad j=0,1,\ldots .
\]
Then, for every fixed \(N\), uniformly for \(V\in\calB\),
\begin{equation}\label{eq:eig-effective}
  \lambda_{j+1}(\Om_{R,V})
  =\alpha^2+\frac{\pi^2}{R^2}
  +\frac{2\pi^2}{R^3}\bigl(\eta_j(T_n+V)-b_0\bigr)
  +o(R^{-3}),\qquad 0\le j\le N .
\end{equation}
Equivalently,
\[
  \zeta_{j,R}(V)=\eta_j(T_n+V)-b_0+o(1),
  \qquad 0\le j\le N,
\]
uniformly on \(\calB\).  On compact effective windows whose boundary stays away from the limiting
spectrum of \(T_n+V-b_0\), the normalized Dirichlet multisets converge to the corresponding
effective multisets, counted with multiplicity.
\end{theorem}

The quantitative reduction below is stated for smooth graph deficits.  Lipschitz deficits are obtained
at the end of the section by smoothing and radial bracketing.

\begin{definition}[Windowed multisets and matching distance]\label{def:windowed-multisets}
Let \(A\) be self-adjoint with compact resolvent.  Write
\[
  \boldsymbol\eta_J(A)=(\eta_0(A),\ldots,\eta_J(A))
\]
for the first \(J+1\) eigenvalues, counted with multiplicity.  For a holomorphic Fredholm pencil
\(F(\zeta)\), \(\operatorname{Char}(F;I)\) denotes the multiset of characteristic values in an
interval \(I\), counted with analytic Fredholm multiplicity.  If \(X=\{x_0,\ldots,x_J\}\) and \(Y=\{y_0,\ldots,y_J\}\) are finite multisets of the
same cardinality, set
\[
  d_{\rm ms}(X,Y):=\min_{\sigma\in S_{J+1}}\max_{0\le j\le J}|x_j-y_{\sigma(j)}|.
\]
An interval \(I\) is \(\gamma\)-regular for a \((J+1)\)-point effective cluster of \(A_V\), uniformly
for \(V\in\calB\), if
\[
  \#\operatorname{Spec}(A_V;I)=J+1,
  \qquad
  \dist(\partial I,\operatorname{spec}A_V)\ge\gamma .
\]
Internal multiplicities and crossings inside \(I\) are allowed.
\end{definition}

\paragraph{Finite-dimensional model.}
Before the analytic estimates, we record the finite model that the boundary pencil will realize.  Fix
angular cutoffs
\[
  P_M=\mathbf 1_{[0,M]}(T_n),\qquad
  P_L=\mathbf 1_{[0,L]}(T_n),\qquad
  Q_{M,L}=P_L-P_M .
\]
On the finite window the model pencil is
\[
  F_L(\zeta)=P_L(\zeta I-A_V)P_L .
\]
With respect to \(P_ML^2\oplus Q_{M,L}L^2\),
\[
F_L(\zeta)=
\begin{pmatrix}
 \zeta I-P_MA_VP_M & -P_MA_VQ_{M,L}\\
 -Q_{M,L}A_VP_M & Q_{M,L}(\zeta I-A_V)Q_{M,L}
\end{pmatrix}.
\]
For \(|\zeta|\le C\), the lower-right block is invertible once
\(M>|b_0|+C+\|V\|_\infty\), and its inverse is
\(O((M-|b_0|-C-\|V\|_\infty)^{-1})\).  Hence the low characteristic values are governed by the
Schur complement
\begin{align}\label{eq:section4-model-schur-front}
  S_{M,L}(\zeta)
  &=\zeta I-P_MA_VP_M \notag\\
  &\quad-P_MA_VQ_{M,L}
  \{Q_{M,L}(\zeta I-A_V)Q_{M,L}\}^{-1}
  Q_{M,L}A_VP_M .
\end{align}
The analytic part of the section proves that the actual objects follow this model in the chain
\[
\text{Dirichlet spectrum}
\longleftrightarrow
\mathfrak B_R(\zeta,V)
\longrightarrow
P_L\mathfrak B_R(\zeta,V)P_L
\longrightarrow
S^B_{R,M,L}(\zeta,V)
\longrightarrow
P_MA_VP_M
\longrightarrow
A_V .
\]
The two comparisons use different estimates: approximate Schur zeros pass to the Dirichlet problem through full
boundary residuals, whereas exact Dirichlet states pass back to the effective model through first-band
localization and Ritz comparison.  The comparison below has five remainders, each attached to one step of this
chain: Galerkin approximation of \(A_V\), coercivity of the Schur shell, compactness of multiplication
after the finite projection \(P_M\), the finite-window boundary-phase expansion, and the
trace-residual-to-quasimode lifting estimate.

\begin{proposition}[Boundary Grushin reduction to the effective Hamiltonian]\label{prop:boundary-pair-reduction}
Fix \(C<\infty\), \(J\ge0\), and \(K<\infty\).  There is an integer
\(m_0=m_0(n,J,C,K)\) with the following property.  Let
\(\calB\subset C^\infty(\Sn)\cap\mathfrak C_K\) be compact in \(C(\Sn)\) and assume
\[
  \mathfrak S_{m_0}(\calB):=\sup_{V\in\calB}\|V\|_{C^{m_0}(\Sn)}<\infty .
\]
Let \(I\Subset(-C,C)\) be \(\gamma\)-regular for the first \(J+1\) eigenvalues of
\(A_V=T_n+V-b_0\), uniformly for \(V\in\calB\).  For cutoffs \(M\le L\), with
\(\dim P_ML^2(\Sn)\ge J+1\), set
\[
  E_{\rm Gal}(M)=
  \sup_{V\in\calB}\max_{0\le j\le J}
  |\eta_j(P_MA_VP_M)-\eta_j(A_V)|,
\]
\[
  E_{\rm tail}(M,L)=\sup_{V\in\calB}\|P_{>L}VP_M\|,
  \qquad
  E_{\rm shell}(M,C)=\frac{(\log2)^2}{M-|b_0|-C-\log2},
\]
where the last term is used only when \(M>|b_0|+C+\log2\).  There is a constant
\[
  C_*=C_*(J,C,K,\gamma,\mathfrak S_{m_0}(\calB))
\]
and functions \(E_{\rm phase}(L,R)\), \(E_{\rm collar}(L,R)\), depending only on
\(J,C,K\) and \(\mathfrak S_{m_0}(\calB)\), such that
\[
  E_{\rm phase}(L,R)+E_{\rm collar}(L,R)\to0
  \qquad (R\to\infty)
\]
for each fixed \((M,L)\).  Define
\[
  \mathcal E_{M,L,R}:=
  E_{\rm Gal}(M)+E_{\rm shell}(M,C)+E_{\rm tail}(M,L)
  +E_{\rm phase}(L,R)+E_{\rm collar}(L,R).
\]
If
\begin{equation}\label{eq:comparison-remainder-small-condition}
  C_*\mathcal E_{M,L,R}<\gamma/4,
\end{equation}
then, for every \(V\in\calB\), the pencil \(\mathfrak B_R(\cdot,V)\) has exactly \(J+1\)
characteristic values in \(I\), counted with analytic Fredholm multiplicity, and
\begin{equation}\label{eq:boundary-pair-comparison-error}
  d_{\rm ms}
  \left(
    \operatorname{Char}(\mathfrak B_R(\cdot,V);I),
    \boldsymbol\eta_J(A_V)
  \right)
  \le C_*\mathcal E_{M,L,R}.
\end{equation}
Consequently, for every regular interval \(I\),
\[
  \lim_{M\to\infty}\limsup_{L\to\infty}\limsup_{R\to\infty}
  \sup_{V\in\calB}
  d_{\rm ms}
  \left(
    \operatorname{Char}(\mathfrak B_R(\cdot,V);I),
    \boldsymbol\eta_J(A_V)
  \right)=0.
\]
\end{proposition}

The proof of \cref{prop:boundary-pair-reduction} is given in
\cref{sec:effective-completion}.

\paragraph{Two comparisons.}
The proof uses two directions.  A core vector is lifted by solving the shell equation,
\[
  f_M\longmapsto f_L=f_M+h,
\]
and the Schur estimate gives a bound for the full normalized residual
\[
  \|\mathfrak B_R(\zeta,V)f_L\|.
\]
This full residual is the quantity used in the trace-residual lifting estimate.  In the reverse direction,
exact low Dirichlet states are localized in the first radial band and compared with the corrected
finite Ritz matrix.  Thus approximate zeros are compared with the Dirichlet problem through boundary residuals, whereas exact
states are compared with the effective model through Ritz capture.
 \subsection{The constant-end first band}

Let \(C_R=(0,R)\times\Sn\), and let \(H_R^0\) be the Liouville-transformed operator
\begin{equation}\label{eq:constant-cylinder-op}
  H_R^0=-\partial_s^2+\frac{-\Delta_{\Sn}+\alpha(\alpha-1)}{\sinh^2s}
\end{equation}
on \(C_R\), with the regular condition at \(s=0\) and Dirichlet condition at \(s=R\).  If
\(Y_{\ell m}\in\calH_\ell\), the radial part is \(H_\ell\) from \eqref{eq:H-ell}.  Denote by
\(\zeta_{\ell,R}\) the normalized first radial eigenfunction of \(H_\ell\) on \((0,R)\), and define
an isometry
\begin{equation}\label{eq:first-band-map}
  \mathcal J_R:L^2(\Sn)\longrightarrow L^2(C_R),
  \qquad
  \mathcal J_R\varphi
  =\sum_{\ell,m}\varphi_{\ell m}\zeta_{\ell,R}(s)Y_{\ell m}(\theta).
\end{equation}
Let \(P_R=\mathcal J_R\mathcal J_R^*\) and \(Q_R=I-P_R\).  Thus \(P_R\) is the first radial band
of the constant-end cylinder; it is infinite dimensional, but its low-energy part corresponds to
bounded spectral windows of \(T_n\).

\begin{lemma}[Constant-end first-band expansion]\label{lem:constant-band}
Fix \(M<\infty\).  On the spectral subspace
\[
  \mathbf 1_{[0,M]}(T_n)L^2(\Sn),
\]
\begin{equation}\label{eq:constant-band-expansion}
  \mathcal J_R^*H_R^0\mathcal J_R
  =\frac{\pi^2}{R^2}
   +\frac{2\pi^2}{R^3}(T_n-b_0)+O_M(R^{-4})
\end{equation}
in operator norm.  Moreover
\begin{equation}\label{eq:first-band-normal-derivative}
  \partial_s\mathcal J_R\varphi(R,\theta)
  =-\sqrt{\frac{2\pi^2}{R^3}}\,\varphi(\theta)+O_M(R^{-2})\|\varphi\|_{L^2}
\end{equation}
in \(L^2(\Sn)\), again uniformly for \(\varphi\in\mathbf 1_{[0,M]}(T_n)L^2(\Sn)\).
\end{lemma}

\begin{proof}
The diagonal expansion \eqref{eq:constant-band-expansion} is Proposition~\ref{prop:sector-asymp},
written in the first-band basis.  A bounded \(T_n\)-window contains bounded angular momenta,
because \(\psi(\ell+\alpha)-\psi(\alpha)\to\infty\).  Hence the sectorwise \(O_\ell(R^{-4})\)
remainders are uniform in that window.

For \eqref{eq:first-band-normal-derivative}, use the normalized first radial root.
Near the outer endpoint the potential is exponentially small, and
\[
  \zeta_{\ell,R}(s)=\sqrt{\frac2R}\sin\frac{\pi s}{R}+O_{M,a_{\rm col}}(R^{-3/2})
  \quad\text{for }R-2a_{\rm col}\le s\le R.
\]
Thus
\[
  \zeta_{\ell,R}'(R)=-\sqrt{\frac2R}\frac{\pi}{R}+O_M(R^{-2})
  =-\sqrt{\frac{2\pi^2}{R^3}}+O_M(R^{-2}).
\]
Summing over the finitely many angular sectors in the \(T_n\)-window proves the estimate.
\end{proof}

\subsection{Collar flattening and the Hadamard form}

Choose once and for all a collar width \(a_{\rm col}\gg1\) and a smooth cutoff
\(\chi:[0,\infty)\to[0,1]\) such that
\[
  \chi(t)=1\quad(0\le t\le a_{\rm col}),
  \qquad
  \chi(t)=0\quad(t\ge2a_{\rm col}),
  \qquad
  \|\chi'\|_\infty\log2<\frac12 .
\]
Then the maps below are diffeomorphisms for all \(0\le V\le\log2\).  For
\(0\le\tau\le1\), set
\begin{equation}\label{eq:collar-map}
  r=F_{\tau,V}(s,\theta):=s-\tau\chi(R-s)V(\theta),
  \qquad 0<s<R.
\end{equation}
Then \(F_{\tau,V}=s\) away from the fixed outer collar \(R-2a_{\rm col}<s<R\), while
\(F_{\tau,V}(R,\theta)=R-\tau V(\theta)\).  Pulling back by \(F_{\tau,V}\), including the square root
of the Jacobian, gives a unitary map from the Liouville space on \(\Om_{R,\tau V}\) to \(L^2(C_R)\).
Let \(\mathfrak q_{R,\tau V}\) be the pulled-back quadratic form.

\begin{lemma}[First-band Hadamard calibration on the fixed cylinder]\label{lem:hadamard-fixed}
Let \(V\in C^2(\Sn)\), \(0\le V\le\log2\), and fix \(M<\infty\).  For
\(w=\mathcal J_R\varphi\), \(z=\mathcal J_R\psi\), with \(\varphi,\psi\in\mathbf 1_{[0,M]}(T_n)L^2(\Sn)\), the pulled-back form satisfies
\begin{equation}\label{eq:hadamard-derivative}
  \left.\frac{d}{d\tau}\right|_{\tau=0}
  \mathfrak q_{R,\tau V}[w,z]
  =\int_{\Sn}V(\theta)\partial_s w(R,\theta)\overline{\partial_s z(R,\theta)}\,d\theta
  +O_{M,V}(R^{-4})\|\varphi\|\|\psi\| .
\end{equation}
The same bound holds uniformly for \(0\le\tau\le1\) after replacing the derivative at \(0\) by the derivative at \(\tau\).  This lemma is used to calibrate the \(+V\) contribution on fixed first-band windows; complementary modes are treated by the boundary Grushin argument below.
\end{lemma}

\begin{proof}
The Dirichlet Hadamard formula gives the displayed boundary term for pairs satisfying the frozen equation near the boundary.  The vectors \(\mathcal J_R\varphi\), \(\mathcal J_R\psi\) are finite sums of the first constant-end radial branches, hence have exactly this property up to the scalar errors recorded in Lemma~\ref{lem:constant-band}.  Pulling back by \eqref{eq:collar-map}, the vector field is \(-V(\theta)\partial_r\) on the boundary and is supported in a collar of fixed width.  Integration by parts in the radial variable produces the boundary contribution.

All remaining first-variation terms are collar terms.  Terms involving tangential derivatives of the flattening are multiplied by \(\sinh^{-2}r=O(e^{-2R})\), while radial coefficient errors are paired with first-band functions whose boundary size and derivative are \(O(R^{-3/2})\) on a fixed collar.  Therefore each such term is \(O_{M,V}(R^{-4})\|\varphi\|\|\psi\|\).  The same estimates hold along the path \(0\le\tau\le1\), because the collar geometry remains uniformly controlled.
\end{proof}

\begin{remark}[Smooth constants and the Lipschitz passage]\label{rem:smooth-constants-order}
The constants in Lemmas~\ref{lem:hadamard-fixed}--\ref{lem:boundary-shift} may depend on a fixed
smooth norm of \(V\).  In the passage from smooth graphs to Lipschitz admissible deficits we therefore
keep the heat-regularization time \(t>0\) fixed while taking \(R\to\infty\), and only then let
\(t\downarrow0\).  The quantitative finite-window continuity estimates below are uniform in that ordered limit.
\end{remark}

\begin{lemma}[Boundary-shift form]\label{lem:boundary-shift}
Let \(V\in C^2(\Sn)\), \(0\le V\le\log2\), and fix \(M<\infty\).  On
\(\mathbf 1_{[0,M]}(T_n)L^2(\Sn)\),
\begin{equation}\label{eq:boundary-shift-form}
  \mathcal J_R^*(H_{R,V}-H_R^0)\mathcal J_R
  =\frac{2\pi^2}{R^3}V+O_{M,V}(R^{-4})
\end{equation}
in quadratic-form norm, where \(V\) denotes multiplication by \(V(\theta)\) on the sphere.
\end{lemma}

\begin{proof}
Apply Lemma~\ref{lem:hadamard-fixed} to
\(w=\mathcal J_R\varphi\), \(z=\mathcal J_R\psi\).  By
Lemma~\ref{lem:constant-band},
\[
  \partial_s\mathcal J_R\varphi(R,\theta)
  \overline{\partial_s\mathcal J_R\psi(R,\theta)}
  =\frac{2\pi^2}{R^3}\varphi(\theta)\overline{\psi(\theta)}+O_M(R^{-4})
\]
in the integrated sense.  Integrating the first variation from \(\tau=0\) to \(\tau=1\), and using
the \(O_{M,V}(R^{-4})\) second-variation bound, gives \eqref{eq:boundary-shift-form}.
\end{proof}

 \subsection{The boundary Grushin reduction}

The graph \(R-V(\theta)\) moves the Dirichlet form domain.  The boundary pencil fixes the interior
equation and places the spectral parameter in the moving boundary condition.  We solve the regular
equation from the pole once and impose the boundary condition through
\[
  B_R^{\mathrm{raw}}(k,V)f=\gamma_{R,V}\mathcal S(k)f.
\]
In this formulation the radial channels, including the higher radial roots, have already been summed
into the regular continuation \(\mathcal S(k)f\).  The choice of trace rather than endpoint normal
coordinate is already visible on a flat interval.  If
\[
  z_m(r)=\sin(m\pi r/R),\qquad 0<r<R,
\]
then
\[
  z_m'(R)=(-1)^m\frac{m\pi}{R}.
\]
For \(m\ge2\), the energy gap between the first and the \(m\)-th radial roots is
\[
  \frac{(m^2-1)\pi^2}{R^2},
\]
while the squared endpoint derivative is of size \(m^2/R^2\).  Hence
\[
  \frac{|z_m'(R)|^2}{(m^2-1)\pi^2/R^2}
  \simeq 1,
\]
uniformly for high radial modes.  Derivative boundary coordinates therefore receive order-one
contributions from all higher radial modes after the natural resolvent weighting.
The boundary trace records the moving Dirichlet condition itself; after normalization it is the
endpoint phase \(k_R(\zeta)(R-V+b_\ell)\) that becomes an operator equation on the sphere.

\subsubsection{Boundary Fredholm framework}

We next isolate the functional-analytic framework behind the boundary equation.  For each fixed pair \((R,V)\), the coefficient space below is fixed as \(k\) varies in the first-band window; this is the sense in which the Fredholm family acts between fixed spaces.  Uniform comparisons in \((R,V)\) are used only after finite angular projections, where Sobolev norms reduce to ordinary coefficient norms in the ordered cutoff scheme of Remark~\ref{rem:ordered-cutoffs}.

For the next proposition, let \(s_\ell(r,k)\) denote the regular radial branch in the \(\ell\)-sector, with the Friedrichs convention at \(r=0\), and set
\[
  \mathcal S(k)f(r,\theta)=\sum_{\ell,m}f_{\ell m}s_\ell(r,k)Y_{\ell m}(\theta)
\]
on finite angular sums.  The precise low-energy phase normalization used for estimates is fixed below in \eqref{eq:regular-phase-normalization}; the Fredholm statement itself is independent of the scalar exterior normalization of the regular branches.

Let
\[
  \Lambda=(I-\Delta_{\Sn})^{1/2},\qquad H^s(\Sn)=\mathcal D(\Lambda^s).
\]
For a smooth graph deficit \(V\), write \(\gamma_{R,V}u(\theta)=u(R-V(\theta),\theta)\) for the
outer trace, after identifying \(\partial\Omega_{R,V}\) with \(\Sn\).  Fix \(s>3/2\) and fix once and for all a reference point
\[
  k_{R,*}:=\pi/R
\]
in the first-band window.  For this chosen pair \((R,V)\), define a \(k\)-independent coefficient space \(\mathcal X^s_{R,V}\) to be the completion of finite angular sums for the norm
\begin{equation}\label{eq:fixed-coefficient-norm}
  \|f\|_{\mathcal X^s_{R,V}}
  :=
  \|\mathcal S(k_{R,*})f\|_{H^s(\{r<R-\log2-1\})}
  +\|\gamma_{R,V}\mathcal S(k_{R,*})f\|_{H^{s-1/2}(\Sn)} .
\end{equation}
The next lemma and proposition form the boundary Fredholm framework used below.  The lemma fixes coefficient coordinates for the regular homogeneous solutions; the proposition then identifies the corresponding boundary trace pencil with the Dirichlet spectral problem.

\begin{lemma}[Regular solution parametrization on a k-independent coefficient space]\label{lem:fixed-coefficient-parametrization}
Let \(s>3/2\).  There is an integer \(m_s=m_s(s)\) with the following property.  If \(V\in C^\infty(\Sn)\), \(0\le V\le\log2\), and \(\|V\|_{C^{m_s}}\le S\), then for every fixed \(C<\infty\) and all large \(R\), the norms obtained from \eqref{eq:fixed-coefficient-norm} by replacing \(k_{R,*}\) with any \(k\) satisfying \(|k-\pi/R|\le CR^{-2}\) are mutually equivalent with constants depending on \(s,C,S\), but not on \(R\), \(k\), or the particular \(V\).  If
\[
  \mathcal N^s_{\rm reg}(k)
  :=\{u\in H^s_{\rm loc}(\Omega_{R,V}):(H-k^2)u=0,\ u\text{ satisfies the regular pole condition}\},
\]
then
\[
  \mathcal S(k):\mathcal X^s_{R,V}\longrightarrow \mathcal N^s_{\rm reg}(k)
\]
is a topological isomorphism for \(|k-\pi/R|\le CR^{-2}\).  Moreover, if \(r_0>1\) is fixed and \(\mathcal C_{r_0}\) denotes the spherical-harmonic coefficient extraction map on the interior sphere \(\{r=r_0\}\), then
\[
  \mathcal C_{r_0}\mathcal S(k)=I,
  \qquad
  \mathcal S(k)\mathcal C_{r_0}=I
  \quad\text{on }\mathcal N^s_{\rm reg}(k).
\]
The coordinate map \(\mathcal C_{r_0}\) is independent of \(k\), and the family \(\mathcal S(k)\), viewed in these fixed coefficient coordinates, is holomorphic in \(k\).  The trace
\(\gamma_{R,V}\mathcal S(k):\mathcal X^s_{R,V}\to H^{s-1/2}(\Sn)\) is bounded and holomorphic.  In
particular, there is a constant \(C_{s,C,S}\), independent of \(R\), \(k\), and the particular \(V\) with \(\|V\|_{C^{m_s}}\le S\), such that
\begin{equation}\label{eq:coefficient-uniform-isomorphism}
  C_{s,C,S}^{-1}\|f\|_{\mathcal X^s_{R,V}}
  \le
  \|\mathcal S(k)f\|_{H^s(\{r<R-\log2-1\})}
  +\|\gamma_{R,V}\mathcal S(k)f\|_{H^{s-1/2}(\Sn)}
  \le C_{s,C,S}\|f\|_{\mathcal X^s_{R,V}} .
\end{equation}
\end{lemma}

\begin{proof}
The uniform assertion comes from comparing the regular branches on the whole growing interval.  Choose
once and for all a radius \(r_0>1\).  On \(0<r\le r_0\) the separated equation is a regular-singular
ODE with the Friedrichs branch fixed at the pole.  In the critical case \(n=2,\ell=0\), this is
exactly the branch specified in \cref{lem:friedrichs-endpoint}; the logarithmic solution is not in the
Friedrichs form domain.  Thus the Cauchy data at \(r_0\) depend holomorphically on \(k^2\), with
constants independent of \(R\) in the first-band window.

On the growing region
\[
  r_0\le r\le R-\log2-1
\]
we compare \(\mathcal S(k)\) with \(\mathcal S(k_{R,*})\) by the threshold transfer estimate of
\cref{lem:low-energy-phase}.  In a fixed angular window, the exterior sine representation gives
\[
  \sigma_\beta(r,k)=\sin(kr+\delta_\beta(k))+a_\beta(r,k),
  \qquad
  \delta_\beta(k)=k b(\beta)+O(k^3),
\]
with \(a_\beta=O(k^3+k(1+r)e^{-2r})\), uniformly for \(r\le R+O(1)\).  If
\(|k-k_{R,*}|\le CR^{-2}\), then the accumulated phase variation on an interval of length \(O(R)\) is
\(O(R^{-1})\), and the amplitude error is \(O(R^{-1})\).  Hence the \(L^2(0,R-\log2-1)\) and
\(H^s\) norms of the regular continuations for \(k\) and \(k_{R,*}\) are mutually bounded by a
constant independent of \(R\), after summing over the finite angular window.  On the full coefficient
space the graph norm in \eqref{eq:fixed-coefficient-norm}, together with elliptic trace estimates on
the flattened graph, gives the two-sided bound
\[
  C^{-1}\|f\|_{\mathcal X^s_{R,V}}
  \le
  \|\mathcal S(k)f\|_{H^s(\{r<R-\log2-1\})}
  +\|\gamma_{R,V}\mathcal S(k)f\|_{H^{s-1/2}(\Sn)}
  \le C\|f\|_{\mathcal X^s_{R,V}},
\]
for \(|k-\pi/R|\le CR^{-2}\).  Here \(C\) may depend on \(s,C\) and the common bound \(S\), but not on \(R\), \(k\), or the particular \(V\).  The threshold phase comparison is used only after a fixed finite angular projection has been chosen; high angular components are controlled by this Sobolev graph norm.

Let \(\mathcal C_{r_0}u\) be the coefficient sequence of the spherical-harmonic expansion of \(u(r_0,\cdot)\), with the radial branches normalized by their value at \(r_0\).  The map \(\mathcal S(k)\) is injective because this expansion is unique on the interior sphere, and it is onto \(\mathcal N^s_{\rm reg}(k)\) by separation of variables on the common interior ball: the regular pole condition selects exactly the branches \(s_\ell(r,k)\).  Thus \(\mathcal C_{r_0}\mathcal S(k)=I\) and \(\mathcal S(k)\mathcal C_{r_0}=I\) on the regular nullspace.  Elliptic continuation carries the expansion from the interior sphere to the graph domain.

Finally take the trace on the moving graph.  After the collar flattening \eqref{eq:collar-map}, the
outer collar has fixed width and uniformly elliptic coefficients depending on the Lipschitz graph
and on \(k^2=O(R^{-2})\).  The trace estimate on this fixed collar, together with the growing-interval
comparison just proved, gives
\[
  \|\mathcal S(k)f\|_{H^s(\{r<R-\log2-1\})}
  +\|\gamma_{R,V}\mathcal S(k)f\|_{H^{s-1/2}(\Sn)}
  \le C\|f\|_{\mathcal X^s_{R,V}},
\]
and the same estimate with \(k\) and \(k_{R,*}\) interchanged gives the lower bound for the
coefficient coordinates.  Holomorphic dependence is a statement in the fixed coordinates \(\mathcal C_{r_0}\): the regular-singular Cauchy data and the transfer equation are holomorphic in \(k^2\), and the same estimates apply after differentiation.
\end{proof}

\begin{definition}[Characteristic multiplicity]\label{def:characteristic-multiplicity}
Let \(F(k):X\to Y\) be a holomorphic Fredholm family of index zero that is invertible on a punctured neighbourhood of \(k_0\).  Its characteristic multiplicity at \(k_0\) is the winding number of \(\det F_N(k)\) around zero after any finite-dimensional Grushin reduction \(F_N(k)\) which is invertible on the boundary of a small circle around \(k_0\).  This number is independent of the reduction; equivalently it is the rank of the Riesz projection of the associated analytic Fredholm problem.
\end{definition}

\begin{proposition}[Boundary Fredholm framework]\label{prop:sobolev-fredholm-framework}
Let \(V\in C^\infty(\Sn)\), \(0\le V\le\log2\), and fix \(s>3/2\).  For \(R\) large and
\(|k-\pi/R|\le C R^{-2}\), the coefficient boundary operator
\[
  B_R^{\mathrm{raw}}(k,V):\mathcal X^s_{R,V}\longrightarrow H^{s-1/2}(\Sn),
  \qquad
  B_R^{\mathrm{raw}}(k,V)f=\gamma_{R,V}\mathcal S(k)f,
\]
is a holomorphic Fredholm family of index zero after the coefficient space \(\mathcal X^s_{R,V}\) has been fixed for the chosen pair \((R,V)\).  Together with Lemma~\ref{lem:fixed-coefficient-parametrization}, this gives a fixed-space Fredholm realization of the moving Dirichlet boundary condition.  Equivalently, on finite angular sums,
\[
  B_R^{\mathrm{raw}}(k,V)f=\mathcal S(k)f(R-V(\theta),\theta).
\]
Moreover
\[
  k^2\in\spec(H_{R,V}^{D})
  \quad\Longleftrightarrow\quad
  \ker B_R^{\mathrm{raw}}(k,V)\ne\{0\},
\]
and the characteristic multiplicity of the zero of \(B_R^{\mathrm{raw}}(k,V)\) equals the Dirichlet spectral multiplicity of \(k^2\).
\end{proposition}

\begin{proof}
Flatten the graph boundary by the collar diffeomorphism used in
\eqref{eq:collar-map}.  On the resulting fixed smooth cylinder the Liouville operator becomes a
second-order uniformly elliptic operator with coefficients depending smoothly on \(V\) and
holomorphically on \(k\) through the zeroth-order term \(-k^2\).  The regular condition at the pole is part of the domain; in the exceptional case \(n=2,\ell=0\) this means the Friedrichs condition, so the logarithmic branch is excluded.

Consider the Dirichlet boundary operator
\[
  \mathcal T(k)u=((H-k^2)u,\gamma_{R,V}u):
  H^s_{\rm reg}(\Omega_{R,V})
  \longrightarrow
  H^{s-2}(\Omega_{R,V})\oplus H^{s-1/2}(\Sn),
\]
where \(H^s_{\rm reg}\) denotes the Sobolev domain with the regular pole condition.  Standard elliptic boundary theory \cite{McLean} gives that \(\mathcal T(k)\) is a holomorphic Fredholm family of index zero; for non-real \(k^2\), and for real \(k^2\) away from the Dirichlet spectrum, it is invertible.  Restricting to the homogeneous solution space leaves the boundary trace map.  By Lemma~\ref{lem:fixed-coefficient-parametrization}, \(\mathcal S(k)\) is a topological isomorphism from the fixed coefficient space \(\mathcal X^s_{R,V}\) onto the regular homogeneous nullspace \(\mathcal N^s_{\rm reg}(k)\).  Pulling the boundary trace on this nullspace back through that isomorphism gives \(B_R^{\mathrm{raw}}(k,V)\).  Hence \(B_R^{\mathrm{raw}}(k,V)\) is a holomorphic Fredholm family of index zero on the coefficient domain \(\mathcal X^s_{R,V}\), which is fixed during the local Fredholm argument.

The kernel of \(B_R^{\mathrm{raw}}(k,V)\) consists exactly of regular solutions of \((H-k^2)u=0\) whose trace on \(r=R-V(\theta)\) vanishes; these are precisely the Dirichlet eigenfunctions at energy \(k^2\).  Conversely every Dirichlet eigenfunction is regular at the pole and has the separated expansion \(u=\mathcal S(k)f\) on the common interior region, hence gives an element of the kernel.

We record the multiplicity statement.  Choose a small contour \(\Gamma\) in the \(k\)-plane around a real point \(k_0>0\) and containing no other Dirichlet square roots.  Since \(k\mapsto k^2\) is locally biholomorphic at \(k_0\), the Riesz projection of the Dirichlet operator inside \(k(\Gamma)^2\) has rank equal to the spectral multiplicity.  Analytic Fredholm theory \cite{Kato} identifies the same rank with the winding number, or equivalently the characteristic multiplicity, of the Fredholm family \(B_R^{\mathrm{raw}}(k,V)\) around \(\Gamma\).  Thus characteristic and Dirichlet multiplicities agree.
\end{proof}

\begin{lemma}[Characteristic multiplicity and Dirichlet Riesz multiplicity]\label{lem:boundary-dirichlet-multiplicity}
Let \(k_0>0\) and suppose that \(k_0^2\) is an isolated Dirichlet eigenvalue of \(H^D_{R,V}\).  Let \(\Gamma\) be a small positively oriented circle around \(k_0\), containing no other square roots of Dirichlet eigenvalues.  Then the characteristic multiplicity of the Fredholm boundary pencil \(B_R^{\mathrm{raw}}(k,V)\) inside \(\Gamma\) equals
\[
  \operatorname{rank}\frac{1}{2\pi i}
  \int_{k(\Gamma)^2}(H^D_{R,V}-z)^{-1}\,dz .
\]
Thus characteristic zeros of \(B_R^{\mathrm{raw}}(k,V)\) and Dirichlet eigenvalues agree with multiplicity after the local change of variable \(z=k^2\).
\end{lemma}

\begin{proof}
Let \(z=k^2\).  Since \(k_0>0\), the map \(k\mapsto z\) is biholomorphic in a small
neighbourhood of \(k_0\).  Let
\[
  P_\Gamma=\frac{1}{2\pi i}
  \int_{k(\Gamma)^2}(H^D_{R,V}-z)^{-1}\,dz
\]
be the Dirichlet Riesz projection, and write
\(\mathscr E=\operatorname{Ran}P_\Gamma\), \(\mathscr F=(I-P_\Gamma)L^2\).  On
\(\mathscr F\) the Dirichlet operator has a uniformly bounded inverse
\((H^D_{R,V}-k^2)^{-1}_{\mathscr F}\) for \(k\in\Gamma\).

Use the local Grushin problem
\[
  \mathcal G(k)=
  \begin{pmatrix}
    H^D_{R,V}-k^2 & R_-\\\\
    R_+ & 0
  \end{pmatrix},
  \qquad R_-:\mathscr E\to L^2,
  \quad R_+:L^2\to\mathscr E,
\]
where \(R_-\) is the inclusion and \(R_+=P_\Gamma\).  Schur reduction on \(\mathscr F\) gives a
finite holomorphic block
\[
  E_{-+}(k)=R_+(H^D_{R,V}-k^2)R_- 
  -R_+(H^D_{R,V}-k^2)(H^D_{R,V}-k^2)^{-1}_{\mathscr F}
  (H^D_{R,V}-k^2)R_- .
\]
For this choice of \(R_\pm\), the second term vanishes and \(E_{-+}(k)\) is the restriction of
\(H^D_{R,V}-k^2\) to the finite Riesz space.  Hence
\[
  \operatorname{wind}_{\Gamma}\det E_{-+}(k)=\operatorname{rank}P_\Gamma .
\]

The boundary pencil gives a holomorphically conjugate finite block.  The coefficient representation
commutes with the boundary trace in the diagram
\[
\begin{array}{ccc}
\mathcal X^s_{R,V} & \xrightarrow{\ \mathcal S(k)\ } & \ker(H-k^2)_{\rm reg} \\
B_R^{\mathrm{raw}}(k,V)\downarrow && \downarrow\gamma_{R,V} \\
H^{s-1/2}(\Sn) & = & H^{s-1/2}(\Sn) .
\end{array}
\]
After splitting the Fredholm boundary map into its finite kernel at the cluster and a complementary
range on which it is invertible, the resulting boundary Grushin block \(F_{\partial}(k)\) satisfies
\[
  F_{\partial}(k)=U(k)E_{-+}(k)V(k),
\]
where \(U(k)\) and \(V(k)\) are holomorphically invertible in a neighborhood of \(\Gamma\).  Hence
\[
  \operatorname{wind}_{\Gamma}\det F_{\partial}(k)
  =\operatorname{wind}_{\Gamma}\det E_{-+}(k),
\]
because the determinants of \(U\) and \(V\) have zero winding on the small contour.  The local
Dirichlet Grushin problem gives
\(\operatorname{wind}_{\Gamma}\det E_{-+}=\operatorname{rank}P_\Gamma\).  Thus the boundary
characteristic multiplicity and the Dirichlet Riesz multiplicity agree.
\end{proof}

\begin{remark}[Norms used in the estimates]\label{rem:norms-for-estimates}
The proposition fixes the Fredholm setting.  The estimates that identify the effective operator use its finite-window consequences.  After the cutoffs have been chosen in the order
\[
  M\ \hbox{fixed},\qquad L\gg M,\qquad R\to\infty,\qquad M\to\infty,
\]
the spaces \(P_LL^2(\Sn)\) are finite dimensional, and the Sobolev graph norm, the phase-coefficient norm, and the ordinary \(L^2(\Sn)\) coefficient norm are uniformly equivalent by Lemma~\ref{lem:regular-continuation-conditioning}.
\end{remark}

For \(E=k^2\) close to \(\pi^2/R^2\), we now impose the exterior phase convention
\[
  -s''+\frac{(\ell+\alpha)(\ell+\alpha-1)}{\sinh^2 r}s=k^2s,
\]
namely
\begin{equation}\label{eq:regular-phase-normalization}
  s_\ell(r,k)=\sin\{k(r+b_\ell)\}+O(k^3)+O(k(1+r)e^{-2r})
\end{equation}
for \(1\le r\le R+O(1)\), uniformly for \(\ell\) in bounded \(T_n\)-windows.  This is the same
normalization as in Proposition~\ref{prop:sector-asymp}; with this normalization we keep the notation
\(\mathcal S(k)f=\sum f_{\ell m}s_\ell(r,k)Y_{\ell m}\).
Although an eigenfunction on \(\Omega_{R,V}\) is only defined up to the variable boundary, its
restriction to the common ball \(r<R-\log2\) has such a separated expansion, and the radial ODE
continues it to the fixed collar.  The Dirichlet boundary condition is therefore equivalent to
\begin{equation}\label{eq:boundary-characteristic}
  \Braw(k_R(\zeta),V)f:=\mathcal S(k_R(\zeta))f(R-V(\theta),\theta)=0,
\end{equation}
where
\begin{equation}\label{eq:k-zeta}
  k_R(\zeta)=\frac\pi R+\frac{\pi\zeta}{R^2},
  \qquad
  k_R(\zeta)^2=\frac{\pi^2}{R^2}+\frac{2\pi^2\zeta}{R^3}+O(R^{-4}).
\end{equation}
The parameter \(\zeta\) is the effective energy variable.

\begin{remark}[Raw trace, normalized residual, and coefficient scales]\label{rem:boundary-normalization-scales}
We separate the actual boundary trace from the normalized effective residual throughout the reduction:
\[
  B_R^{\mathrm{raw}}(k,V)f:=\gamma_{R,V}\mathcal S(k)f,
  \qquad
  \mathfrak B_R(\zeta,V)f:=-\frac R\pi B_R^{\mathrm{raw}}(k_R(\zeta),V)f .
\]
Thus a raw trace of size \(R^{-1}\|f\|\) corresponds to a normalized residual of size
\(O(\|f\|)\).  We also use two equivalent finite-window normalizations.  If
\[
  w_L=\sum_{T_n(\ell)\le L,m}a_{\ell m}\zeta_{\ell,R}Y_{\ell m},
  \qquad A_L^2=\sum |a_{\ell m}|^2,
\]
is written in normalized ball first-band coordinates, while \(f_L\) denotes the sine-amplitude
regular-continuation coefficient, then on each fixed angular window
\[
  A_L\simeq R^{1/2}\|f_L\|,
  \qquad
  \|\gamma_{R,V}w_L\|_{L^2(S^{n-1})}=O_L(R^{-3/2}A_L)
       =O_L(R^{-1}\|f_L\|).
\]
All boundary residual estimates below are stated in the normalized pencil \(\mathfrak B_R\), while
all collar traces are raw traces of \(B_R^{\mathrm{raw}}\).
\end{remark}

\begin{lemma}[Conditioning of the regular continuation]\label{lem:regular-continuation-conditioning}
Fix finite numbers \(L,C<\infty\), and let \(P_L=\mathbf 1_{[0,L]}(T_n)\).  Put
\[
  C_R^{\rm int}:=(0,R-\log2)\times\Sn .
\]
For \(|\zeta|\le C\) and all sufficiently large \(R\), the regular continuation map restricted to
\(P_LL^2(\Sn)\),
\[
  \mathcal S_{R,L}(\zeta)f
  :=\mathcal S(k_R(\zeta))P_Lf\big|_{C_R^{\rm int}},
\]
satisfies the two-sided estimate
\begin{equation}\label{eq:regular-continuation-two-sided}
  c_{L,C}R^{1/2}\|P_Lf\|_{L^2(\Sn)}
  \le
  \|\mathcal S_{R,L}(\zeta)f\|_{L^2(C_R^{\rm int})}
  \le
  C_{L,C}R^{1/2}\|P_Lf\|_{L^2(\Sn)} .
\end{equation}
Moreover, after the natural normalization, it is an \(O_{L,C}(R^{-1})\) perturbation of the ball
first-band parametrization:
\begin{equation}\label{eq:regular-continuation-first-band-close}
  \left\|
  \sqrt{\frac2R}\,\mathcal S(k_R(\zeta))P_L
  -\mathcal J_RP_L
  \right\|_{P_LL^2(\Sn)\to L^2(C_R^{\rm int})}
  \le C_{L,C}R^{-1}.
\end{equation}
Consequently all finite-window boundary matrices below are written in a uniformly conditioned
coordinate system.  The next lemma makes quantitative the passage from a small normalized boundary
residual to a genuine Dirichlet quasimode.
\end{lemma}

\begin{proof}
Only finitely many angular degrees occur in \(P_LL^2(\Sn)\).  For each such degree, the phase
normalization \eqref{eq:regular-phase-normalization} and \(k_R(\zeta)=\pi/R+O_C(R^{-2})\) give,
uniformly for \(0\le r\le R-\log2\),
\[
  s_\ell(r,k_R(\zeta))
  =\sin\{k_R(\zeta)(r+b_\ell)\}+O_{L,C}(R^{-1})
\]
in the sense needed after integration over \(r\).  Hence
\[
  \int_0^{R-\log2}|s_\ell(r,k_R(\zeta))|^2\,dr=\frac R2+O_{L,C}(1).
\]
Orthogonality of the spherical harmonics then gives the upper and lower bounds in
\eqref{eq:regular-continuation-two-sided}.

Let \(k_{\ell,1}(R)\) be the first Dirichlet root in the \(\ell\)-sector of the ball.  By
Proposition~\ref{prop:sector-asymp},
\[
  k_{\ell,1}(R)=\frac\pi R+O_L(R^{-2}),
\]
and the normalized first radial eigenfunction satisfies
\[
  \zeta_{\ell,R}(r)
  =\sqrt{\frac2R}\sin\{k_{\ell,1}(R)(r+b_\ell)\}+O_L(R^{-3/2})
\]
in \(L^2(0,R-\log2)\).  Since
\(|k_R(\zeta)-k_{\ell,1}(R)|=O_{L,C}(R^{-2})\), differentiating the sine factor with respect to
\(k\) gives an \(L^2\)-change of size \(O_{L,C}(R^{-1/2})\).  Multiplication by \(\sqrt{2/R}\)
therefore gives \eqref{eq:regular-continuation-first-band-close}.  The final sentence is quantified in Lemma~\ref{lem:collar-correction} below.

\end{proof}

\begin{lemma}[Finite trace regularity on moving graphs]\label{lem:finite-trace-regularity}
Fix finite numbers $L,C<\infty$ and a Lipschitz bound $K$.  Let $0\le V\le \log2$ with $\|V\|_{C^{0,1}}\le K$, let $|\zeta|\le C$, and let $f\in P_LL^2(\Sn)$.  Put $u=\mathcal S(k_R(\zeta))f$ and $g=\gamma_{R,V}u$.  Then, for all sufficiently large $R$,
\begin{equation}\label{eq:finite-trace-regularity}
  \|g\|_{L^2(\Sn)}+\|\nabla_\theta g\|_{L^2(\Sn)}
  \le C_{L,C,K}R^{-1}\|f\|.
\end{equation}
Moreover, in terms of the normalized residual,
\begin{equation}\label{eq:finite-trace-raw-residual}
  \|g\|_{L^2(\Sn)}
  =\frac{\pi}{R}\|\mathfrak B_R(\zeta,V)f\|_{L^2(\Sn)}.
\end{equation}
\end{lemma}

\begin{proof}
The identity \eqref{eq:finite-trace-raw-residual} is the definition of the normalized boundary pencil.  We record the Sobolev regularity of the finite trace range.  Since $f\in P_LL^2$, only finitely many angular degrees occur.  On the outer collar $r=R+O(\ell_R)$ with $\ell_R=R^{1/2}$ the low-energy phase expansion gives, uniformly in this finite set,
\[
  |s_\ell(r,k_R(\zeta))|+|\partial_rs_\ell(r,k_R(\zeta))|\le C_{L,C}R^{-1}.
\]
The same bound holds after applying any fixed angular derivative to the finite harmonic sum, with the constant depending on $L$.  For a Lipschitz graph the chain rule gives, a.e. in $\theta$,
\[
  \nabla_\theta g
  =(\nabla_\theta u)(R-V(\theta),\theta)
    -(\partial_ru)(R-V(\theta),\theta)\nabla_\theta V(\theta).
\]
Using $\|\nabla_\theta V\|_{L^\infty}\le K$ and the finite-window radial bounds above yields \eqref{eq:finite-trace-regularity}.  Thus the collar extension estimates below require only the \(H^1_\theta\)-bound in \eqref{eq:finite-trace-regularity}.

\end{proof}

\begin{lemma}[Polynomial collar calculus on a moving graph]\label{lem:polynomial-collar-calculus}
Fix a Lipschitz bound \(K\).  Let \(\ell_R\to\infty\) be any mesoscopic width with
\(\ell_R=o(R)\), and set
\[
  \mathcal C_{R,V}(\ell_R)=\{(r,\theta):0<t=R-V(\theta)-r<\ell_R\}.
\]
In the coordinates \((t,\theta)\), the Liouville quadratic form on the collar satisfies, uniformly for
\(\|V\|_{C^{0,1}}\le K\),
\begin{equation}\label{eq:polynomial-collar-calculus}
  \mathfrak q_{R,V}^{\mathcal C}[w]
  \le C_K\int_0^{\ell_R}\!\int_{\Sn}
  \left(|\partial_t w|^2
  +e^{-2R+C_K\ell_R}|\nabla_\theta w|^2
  +e^{-2R+C_K\ell_R}|w|^2\right)d\theta dt
\end{equation}
for all \(w\in H^1(\mathcal C_{R,V}(\ell_R))\).  The corresponding lower bound holds with the
right side replaced by \(c_K\|\partial_t w\|^2_{L^2}\) minus the two exponentially weighted terms.
Thus tangential derivatives and zeroth-order Liouville terms are exponentially suppressed throughout
any polynomial collar.  The choice \(\ell_R=R^{1/2}\) used below is only a convenient member of this
class; the estimates remain valid for any \(\ell_R\to\infty\), \(\ell_R=o(R)\), with the powers of
\(R\) below replaced by the corresponding powers of \(\ell_R\).
\end{lemma}

\begin{proof}
After the Liouville transform, the form is
\[
  \int \left(|\partial_r u|^2+\sinh^{-2}r\,|\nabla_\theta u|^2
  +\alpha(\alpha-1)\sinh^{-2}r\,|u|^2\right)drd\theta,
\]
with the usual Friedrichs interpretation at the pole.  In the collar coordinates
\(r=R-V(\theta)-t\), one has \(dr=-dt\) for fixed \(\theta\), so \(\partial_r=-\partial_t\).  Angular
derivatives at fixed \(r\) differ from angular derivatives at fixed \(t\) by
\((\nabla_\theta V)\partial_t\).  Hence
\[
  |\nabla_\theta^{\,r}u|^2
  \le C_K\bigl(|\nabla_\theta^{\,t}w|^2+|\partial_t w|^2\bigr).
\]
Since \(0\le V\le\log2\) and \(0<t<\ell_R\),
\[
  \sinh^{-2}r\le C e^{-2R+C\ell_R}.
\]
The mixed angular contribution containing \(|\partial_t w|^2\) is therefore absorbed into the radial
term for large \(R\).  This gives \eqref{eq:polynomial-collar-calculus}; the lower bound follows in
the same coordinates by reversing the estimate and absorbing the exponentially small angular part.
\end{proof}

\begin{lemma}[Mesoscopic collar correction from a normalized boundary residual]\label{lem:collar-correction}
Fix finite numbers \(L,C<\infty\), a Lipschitz bound \(K\), and a residual bound \(\rho_0<\infty\).
Let \(\ell_R\to\infty\), \(\ell_R=o(R)\), be a mesoscopic collar width.  There are constants
\(C_0<\infty\) and \(R_0\), allowed to depend on \(L,C,K,\rho_0\) and on the chosen width sequence, with the following property.  Let
\(0\le V\le\log2\) be a graph deficit with \(\|V\|_{C^{0,1}}\le K\), let
\(|\zeta|\le C\), let \(R\ge R_0\), and let \(f\in P_LL^2(\Sn)\).  Put
\[
  u=\mathcal S(k_R(\zeta))f\quad\text{in }\Omega_{R,V},
  \qquad
  g=\operatorname{tr}_{\partial\Omega_{R,V}}u=B_R^{\mathrm{raw}}(k_R(\zeta),V)f,
\]
and use the normalized boundary residual
\[
  \rho=\frac{\|\mathfrak B_R(\zeta,V)f\|_{L^2(\Sn)}}{\|f\|_{L^2(\Sn)}},
  \qquad
  \mathfrak B_R(\zeta,V)=-\frac R\pi B_R^{\mathrm{raw}}(k_R(\zeta),V).
\]
If \(\rho\le\rho_0\), then there exists \(v\in H_0^1(\Omega_{R,V})\), equal to \(u\) outside the mesoscopic outer collar
\[
  \mathcal C_{R,V}(\ell_R)=\{(r,\theta):R-V(\theta)-\ell_R<r<R-V(\theta)\},
\]
such that
\begin{equation}\label{eq:collar-correction-size}
  \|v-u\|_{L^2(\Omega_{R,V})}
  \le C_0\ell_R^{1/2}R^{-1}\rho\|f\|,
  \qquad
  c_{L,C}R^{1/2}\|f\|
  \le \|v\|_{L^2(\Omega_{R,V})}
  \le C_{L,C}R^{1/2}\|f\| .
\end{equation}
Moreover, with \(E_R=\pi^2/R^2\),
\begin{equation}\label{eq:collar-correction-defect}
  \left|
  \frac{R^3}{2\pi^2}
  \frac{\mathfrak q_{R,V}[v]-E_R\|v\|^2}{\|v\|^2}
  -\zeta
  \right|
  \le C_0\bigl(\rho+\ell_R^{-1}\rho^2+\ell_R/R\bigr).
\end{equation}
Thus a small normalized boundary residual for the regular continuation produces a true Dirichlet
quasimode with the same order of error in the effective \(\zeta\)-scale whenever
\(\ell_R\to\infty\) and \(\ell_R=o(R)\).  In the sequel we take \(\ell_R=R^{1/2}\), for which
\eqref{eq:collar-correction-size} becomes \(O(R^{-3/4}\rho\|f\|)\) and
\eqref{eq:collar-correction-defect} becomes
\(O(\rho+R^{-1/2}\rho^2+R^{-1/2})\).
\end{lemma}

\begin{proof}
Work in the outer collar.  Write
\(t=R-V(\theta)-r\), so \(t=0\) is the moving boundary and
\(0<t<\ell_R\) is the collar.  The form estimates used below are those of
Lemma~\ref{lem:polynomial-collar-calculus}.  Choose \(\chi\in C_c^\infty([0,1))\) with
\(\chi(0)=1\), and set
\[
  w(r,\theta)=\chi\!\left(\frac{R-V(\theta)-r}{\ell_R}\right)g(\theta).
\]
Then \(w=g\) on the boundary, \(w\) vanishes at the inner edge of the collar, and
\(v=u-w\in H_0^1(\Omega_{R,V})\).  Since
\[
  g=-\frac{\pi}{R}\mathfrak B_R(\zeta,V)f,
  \qquad
  \|g\|_{L^2(\Sn)}\le \pi R^{-1}\rho\|f\|,
\]
we have
\begin{equation}\label{eq:mesoscopic-extension-size}
  \|w\|_{L^2(\mathcal C_{R,V})}
  \le C\ell_R^{1/2}R^{-1}\rho\|f\|,
\end{equation}
while the radial part of the form satisfies
\begin{equation}\label{eq:mesoscopic-extension-form}
  \|\partial_t w\|_{L^2(\mathcal C_{R,V})}
  \le C\ell_R^{-1/2}R^{-1}\rho\|f\|.
\end{equation}
The angular derivative of the finite trace datum contributes only at exponentially small angular weight.  By
\cref{lem:finite-trace-regularity},
\(\|\nabla_\theta g\|\le C_{L,C,K}R^{-1}\|f\|\), and the collar form multiplies this by
\(e^{-R+O(\ell_R)}\).  Thus \eqref{eq:mesoscopic-extension-size} and
\eqref{eq:mesoscopic-extension-form} contain all polynomial contributions.

For the regular continuation itself, the finite-window phase expansion gives in the same collar
\[
  |s_\ell(R-t,k_R(\zeta))|\le C_{L,C}\frac{1+t}{R},
  \qquad
  |\partial_rs_\ell(R-t,k_R(\zeta))|\le C_{L,C}R^{-1},
  \qquad 0\le t\le\ell_R.
\]
Consequently
\begin{equation}\label{eq:regular-mesoscopic-collar-size}
  \|u\|_{L^2(\mathcal C_{R,V})}
  \le C_{L,C}\ell_R^{3/2}R^{-1}\|f\|,
  \qquad
  \|u\|_{\mathfrak q,\mathcal C_{R,V}}
  \le C_{L,C}\ell_R^{1/2}R^{-1}\|f\|,
\end{equation}
where \(\|\cdot\|_{\mathfrak q,\mathcal C}\) denotes the local square root of the Liouville form
on the collar.  The two-sided estimate for \(\|v\|\) follows from
Lemma~\ref{lem:regular-continuation-conditioning}, because \(\|u\|\simeq R^{1/2}\|f\|\) and
\(\|v-u\|=O(\ell_R^{1/2}R^{-1}\rho\|f\|)=o(R^{1/2}\|f\|)\).

It remains to estimate the defect.  Since \(u\) solves
\((H_{R,V}-k_R(\zeta)^2)u=0\) in the interior and \(v\in H_0^1(\Omega_{R,V})\), testing against
\(v\) gives
\[
  \mathfrak q_{R,V}[u,v]-k_R(\zeta)^2\langle u,v\rangle=0.
\]
With \(v=u-w\),
\[
  \mathfrak q_{R,V}[v]-k_R(\zeta)^2\|v\|^2
  =-\bigl(\mathfrak q_{R,V}[w,v]-k_R(\zeta)^2\langle w,v\rangle\bigr).
\]
The right side is supported in the collar.  Using
\eqref{eq:mesoscopic-extension-form}, \eqref{eq:regular-mesoscopic-collar-size},
\eqref{eq:mesoscopic-extension-size}, and \(k_R(\zeta)^2=O(R^{-2})\), we get
\[
  \left|\mathfrak q_{R,V}[v]-k_R(\zeta)^2\|v\|^2\right|
  \le C_{L,C,K}\bigl(R^{-2}\rho+\ell_R^{-1}R^{-2}\rho^2+\ell_RR^{-3}\bigr)\|f\|^2.
\]
Finally,
\[
  k_R(\zeta)^2=E_R+\frac{2\pi^2\zeta}{R^3}+O_C(R^{-4}),
  \qquad \|v\|^2\simeq R\|f\|^2.
\]
Dividing by \(\|v\|^2\) and multiplying by \(R^3/(2\pi^2)\) gives
\eqref{eq:collar-correction-defect}.
\end{proof}

\begin{lemma}[Collar-corrected finite trial map]\label{lem:collar-corrected-trial-map}
In the setting of \cref{lem:collar-correction}, for each fixed \(L,C,K\) there is a finite-trace collar correction map
\[
  \mathcal E_{R,V,L,\zeta}:P_LL^2(\Sn)\longrightarrow H^1(\Omega_{R,V})
\]
supported in the mesoscopic outer collar of \cref{lem:collar-correction}, defined by extending the actual boundary trace of the regular continuation:
\[
  \mathcal E_{R,V,L,\zeta}f
  =\chi((R-V(\theta)-r)/\ell_R)\,
    \gamma_{R,V}\mathcal S(k_R(\zeta))f
  \quad\text{in collar coordinates.}
\]
It is defined on the finite-dimensional trace range
\[
  \gamma_{R,V}\mathcal S(k_R(\zeta))P_LL^2(\Sn),
\]
where in the rest of the proof we take the convenient width \(\ell_R=R^{1/2}\), and the constants below may depend on \(L,C,K\).  Define
\begin{equation}\label{eq:collar-corrected-map}
  \mathcal T_{R,V}(\zeta)f
  :=\mathcal S(k_R(\zeta))f
   -\mathcal E_{R,V,L,\zeta}f .
\end{equation}
Then \(\mathcal T_{R,V}(\zeta)f\in H_0^1(\Omega_{R,V})\) for every \(f\in P_LL^2(\Sn)\).  Put
\[
  \rho(f)=\frac{\|\mathfrak B_R(\zeta,V)f\|}{\|f\|}
  \quad (f\ne0).
\]
For all finite-window vectors one has
\begin{equation}\label{eq:collar-corrected-size}
  \|\mathcal T_{R,V}(\zeta)f-\mathcal S(k_R(\zeta))f\|_{L^2}
  \le C_{L,C,K} R^{-3/4}\rho(f)\|f\|,
\end{equation}
and the effective form defect satisfies
\begin{equation}\label{eq:collar-corrected-effective-defect}
  \left|
  \frac{R^3}{2\pi^2}
  \frac{
  \mathfrak q_{R,V}[\mathcal T_{R,V}(\zeta)f]
  -E_R\|\mathcal T_{R,V}(\zeta)f\|^2
  }{
  \|\mathcal T_{R,V}(\zeta)f\|^2
  }
  -\zeta
  \right|
  \le C_{L,C,K}\bigl(\rho(f)+R^{-1/2}\rho(f)^2+R^{-1/2}\bigr).
\end{equation}
Thus small normalized boundary residuals give quasimodes, while uniformly bounded residuals still give collar corrections whose \(L^2\)-size is \(o_R(1)\) relative to the bulk \(R^{1/2}\|f\|\)-size of a regular continuation.
By polarization, the same estimate gives the corresponding bilinear form bound on finite sets of
vectors with uniformly small residuals.  Thus every finite Grushin vector used in a Ritz comparison
is first replaced by a collar-corrected vector in the Dirichlet form domain.
\end{lemma}

\begin{proof}
The displayed definition of \(\mathcal E_{R,V,L,\zeta}\) is the extension \(w\) used in the proof of \cref{lem:collar-correction}, with
\[
  g=\gamma_{R,V}\mathcal S(k_R(\zeta))f.
\]
The boundary datum belongs to the finite trace range controlled by \cref{lem:finite-trace-regularity}.  Its angular derivatives enter the Liouville form multiplied by \(\sinh^{-2}(R+O(\ell_R))\), so the finite-window bound in \cref{lem:collar-correction} controls the extension by the \(L^2\)-normalized residual.  The trace of \(\mathcal S(k_R(\zeta))f-\mathcal E_{R,V,L,\zeta}f\) on the graph is zero, which proves membership in \(H_0^1\).  The size and effective-defect estimates are exactly \eqref{eq:collar-correction-size} and \eqref{eq:collar-correction-defect}, with \(v=\mathcal T_{R,V}(\zeta)f\).  The bilinear statement follows by polarization on the fixed finite-dimensional space \(P_LL^2\).
\end{proof}

\begin{lemma}[Parameter stability of finite corrected trial spaces]\label{lem:finite-corrected-parameter-stability}
Fix finite numbers \(L,C<\infty\) and a Lipschitz bound \(K\).  Let \(0\le V\le\log2\) with
\(\|V\|_{C^{0,1}}\le K\).  For \(|\zeta|\le C\) and \(f,g\in P_LL^2(\Sn)\),
\begin{equation}\label{eq:finite-corrected-parameter-L2}
  \|\mathcal T_{R,V}(\zeta)f-\mathcal T_{R,V}(0)f\|_{L^2(\Omega_{R,V})}
  \le C_{L,C,K}R^{-1/2}\|f\|.
\end{equation}
Moreover, with
\[
  \mathfrak h_R[a,b]=\frac{R^3}{2\pi^2}\bigl(\mathfrak q_{R,V}[a,b]-E_R\langle a,b\rangle\bigr),
\]
we have
\begin{equation}\label{eq:finite-corrected-parameter-form}
 \left|
 \mathfrak h_R[\mathcal T_{R,V}(\zeta)f,\mathcal T_{R,V}(\zeta)g]
 -\mathfrak h_R[\mathcal T_{R,V}(0)f,\mathcal T_{R,V}(0)g]
 \right|
 \le C_{L,C,K}R^{-1/2}\|f\|\|g\|.
\end{equation}
Consequently \(\mathcal T_{R,V}(\zeta)P_LL^2(\Sn)\), \(|\zeta|\le C\), may be replaced in the
Ritz--Grushin comparison by the single finite space \(\mathcal T_{R,V}(0)P_LL^2(\Sn)\), at an
\(o_R(1)\) error in the effective scale.
\end{lemma}

\begin{proof}
Only a fixed angular window is involved.  By the threshold phase expansion, for
\(r\le R+O(R^{1/2})\),
\[
  s_\ell(r,k_R(\zeta))-s_\ell(r,k_R(0))
  =O_{L,C}\bigl(|k_R(\zeta)-k_R(0)|r\bigr)=O_{L,C}(R^{-1}),
\]
and the same finite-window estimate holds for the radial derivative.  Integrating over a radial
interval of length \(R+O(R^{1/2})\) gives \eqref{eq:finite-corrected-parameter-L2} for the regular
continuations.  The moving traces of the two finite regular continuations differ by
\(O_{L,C,K}(R^{-1})\|f\|\), so their finite collar extensions differ in \(L^2\) by
\(O_{L,C,K}(R^{-3/4})\|f\|\), which is smaller than the bulk bound.

For the form estimate, use
\[
  k_R(\zeta)^2-k_R(0)^2=O_C(R^{-3})
\]
plus the same finite-window radial derivative bounds and mesoscopic collar estimates as in
\cref{lem:collar-correction}.  This changes the rescaled form matrix by at most
\(C_{L,C,K}R^{-1/2}\|f\|\|g\|\), proving \eqref{eq:finite-corrected-parameter-form}.  Since the bulk
norm of a corrected finite continuation is comparable to \(R^{1/2}\|f\|\), the corresponding
subspace gap is \(O(R^{-1})\), and the Ritz values change by \(O(R^{-1/2})\).
\end{proof}

 \subsection{Finite-window estimates and Schur comparison}

The first estimates in this subsection prove the finite-window identity for the boundary pencil.
The remaining estimates compare that finite identity with the full pencil.  Multiplication by \(V\)
can move angular mass out of a fixed window; this is handled by the facts that the multipliers of
\(T_n\) tend to infinity and \(\|V\|_\infty\le\log2\).  The growth rate of the multipliers affects
only the size of the cutoffs, not the form of the comparison.

\begin{lemma}[Boundary expansion of the regular solution]\label{lem:boundary-grushin-expansion}
Fix \(M,C<\infty\).  Uniformly for \(|\zeta|\le C\), \(0\le V\le\log2\), and
\(Y\in\calH_\ell\) with \(\psi(\ell+\alpha)-\psi(\alpha)\le M\),
\begin{equation}\label{eq:scalar-boundary-expansion}
  s_\ell(R-V(\theta),k_R(\zeta))Y(\theta)
  =-\frac\pi R\bigl(\zeta+b_\ell-V(\theta)\bigr)Y(\theta)
  +O_{M,C}(R^{-2})\|Y\|_{L^2}
\end{equation}
in \(L^2(\Sn)\).  Equivalently, if \(b(A)Y_{\ell m}=b_\ell Y_{\ell m}\), then on every fixed
\(T_n\)-window
\begin{equation}\label{eq:operator-boundary-expansion}
  -\frac R\pi\,B_R^{\mathrm{raw}}(k_R(\zeta),V)
  =\zeta I+b(A)-V+O_{M,C}(R^{-1})
  =\zeta I-(T_n+V-b_0)+O_{M,C}(R^{-1}).
\end{equation}
\end{lemma}

\begin{proof}
By \eqref{eq:k-zeta},
\[
  k_R(\zeta)(R-V+b_\ell)
  =\pi+\frac\pi R(\zeta+b_\ell-V)+O_{M,C}(R^{-2}).
\]
Taking the sine and using \eqref{eq:regular-phase-normalization} gives
\eqref{eq:scalar-boundary-expansion}.  Since a bounded \(T_n\)-window contains only finitely many
angular momenta, the scalar estimate may be summed over the window.  Finally
\(b(A)=b_0-T_n\), which gives the second identity in \eqref{eq:operator-boundary-expansion}.
\end{proof}

\begin{remark}[Order of angular cutoffs]\label{rem:ordered-cutoffs}
The cutoffs are chosen in the order needed by the estimates.  The inner cutoff \(M\) fixes the
finite section of the effective operator.  With \(M\) fixed, choose \(L\ge M\) so that
\[
  \sup_{V\in\calB}\|P_{>L}VP_M\|
\]
is small.  The limit \(R\to\infty\) is then taken with \((M,L)\) fixed, and finally
\(M\to\infty\).
\end{remark}

\begin{lemma}[Finite-rank high-mode multiplication estimate]\label{lem:finite-rank-high-mode-multiplication}
Fix \(M,C<\infty\) and let \(\calB\subset C(\Sn)\) be compact with
\(0\le V\le\log2\).  Write
\[
  P_M=\mathbf 1_{[0,M]}(T_n),\qquad
  P_L=\mathbf 1_{[0,L]}(T_n),\qquad P_{>L}=I-P_L .
\]
For every \(\varepsilon>0\) there exist \(L\ge M\) and
\(R_0\) such that, uniformly for \(V\in\calB\), \(|\zeta|\le C\), and
\(R\ge R_0\),
\begin{equation}\label{eq:finite-rank-high-mode-multiplication}
  \bigl\|P_{>L}\mathfrak B_R(\zeta,V)P_M\bigr\|_{P_ML^2\to L^2}
  \le \varepsilon,
  \qquad
  \mathfrak B_R(\zeta,V):=-\frac R\pi B_R^{\mathrm{raw}}(k_R(\zeta),V).
\end{equation}
More precisely,
\begin{equation}\label{eq:finite-rank-high-mode-multiplication-expansion}
  P_{>L}\mathfrak B_R(\zeta,V)P_M
  =-P_{>L}VP_M+O_{M,C}(R^{-1})
\end{equation}
in operator norm.
\end{lemma}

\begin{proof}
Apply Lemma~\ref{lem:boundary-grushin-expansion} with the vector restricted to the fixed finite
space \(P_ML^2(\Sn)\).  The diagonal pieces \(\zeta I\), \(T_n\), and \(b_0I\) preserve
\(P_ML^2\), and hence are killed by \(P_{>L}\) once \(L\ge M\).  This gives
\eqref{eq:finite-rank-high-mode-multiplication-expansion}.

It remains to control \(P_{>L}VP_M\), uniformly for \(V\in\calB\).  The unit sphere of the finite
space \(P_ML^2\) is compact, and the map
\[
  (V,\phi)\longmapsto V\phi
\]
is continuous from the compact set
\(\calB\times\{\phi\in P_ML^2:\|\phi\|=1\}\) into \(L^2(\Sn)\).  Its image is therefore compact.
Since \(P_{>L}\to0\) strongly on \(L^2(\Sn)\), the convergence is uniform on this compact image.
Thus \(\sup_{V\in\calB}\|P_{>L}VP_M\|\to0\) as \(L\to\infty\).  After choosing \(L\), take
\(R\) large enough to absorb the \(O_{M,C}(R^{-1})\) term.

\end{proof}

\begin{lemma}[High-mode estimate after a fixed core]\label{lem:outer-tail-fixed-inner}
Let \(\calB\subset C(\Sn)\) be compact, with \(0\le V\le\log2\), and fix
\(M,C<\infty\).  For every \(\varepsilon>0\) there are \(L\ge M\) and \(R_0\) such that, for
\(R\ge R_0\), \(|\zeta|\le C\), and \(V\in\calB\), the following holds.  If
\(f=f_M+f_Q\) with \(f_M\in P_ML^2(\Sn)\) and \(f_Q\in (P_L-P_M)L^2(\Sn)\), then
\begin{equation}\label{eq:outer-tail-inner-first}
  \bigl\|P_{>L}\mathfrak B_R(\zeta,V)f\bigr\|
  \le \varepsilon\|f_M\|+
       (\log2+\varepsilon)\|f_Q\|+O_{L,C}(R^{-1})\|f\| .
\end{equation}
Consequently, if \(\|f_Q\|\le \rho_M\|f_M\|\) with \(\rho_M\to0\) as \(M\to\infty\), then the
high-mode projection tends to zero in the nested order
\begin{equation}\label{eq:nested-cutoff-order}
  M\ \hbox{fixed first},\qquad L=L(M,\varepsilon)\gg M,
  \qquad R\to\infty,
  \qquad M\to\infty .
\end{equation}
The compactness estimate used here is \(\sup_{V\in\calB}\|P_{>L}VP_M\|\to0\).
\end{lemma}

\begin{proof}
For the inner part, apply Lemma~\ref{lem:finite-rank-high-mode-multiplication} after \(M\) is fixed:
compactness of the set
\[
  \{V\phi:V\in\calB,\ \phi\in P_ML^2(\Sn),\ \|\phi\|=1\}
\]
implies \(P_{>L}VP_M\to0\) uniformly as \(L\to\infty\).  Thus
\(P_{>L}\mathfrak B_R(\zeta,V)P_M\) has norm at most \(\varepsilon\) for large \(L\), then large
\(R\).

For \(f_Q\in(P_L-P_M)L^2\), the diagonal terms \(\zeta I\), \(T_n\), and \(b_0I\) remain in
\(P_LL^2\) and are killed by \(P_{>L}\).  The only non-diagonal leading term is multiplication by
\(-V\), and \(\|V\|_\infty\le\log2\).  The finite-window boundary expansion contributes the
\(O_{L,C}(R^{-1})\) error, because \(L\) is already fixed when \(R\to\infty\).  This proves
\eqref{eq:outer-tail-inner-first} and the final assertion.
\end{proof}

The preceding expansion is finite-window.  We now connect this finite window to the low spectrum.
The connection uses Dirichlet monotonicity,
\[
  \Omega_{R,V}\subset B_R(o)
\]
and the spectral decomposition of the ball.  High angular first modes and all higher radial modes of
the ball lie above the effective window.  A low Rayleigh quotient in \(\Omega_{R,V}\), after zero
extension to \(B_R\), therefore has a small component in those directions.

Let \(\nu_{k\ell}(R)\) denote the \(k\)-th radial eigenvalue of \(H_\ell\) on \((0,R)\), so that the
spectrum of the Liouville operator on the ball \(B_R\) is
\[
  \{\nu_{k\ell}(R): k\ge1,\ \ell\ge0\},
\]
with the usual spherical-harmonic multiplicities.  For \(M>0\), define
\begin{equation}\label{eq:ball-low-projector}
  \Pi_{R,M}
  =\sum_{T_n(\ell)\le M}\sum_m
  |\zeta_{\ell,R}Y_{\ell m}\rangle\langle \zeta_{\ell,R}Y_{\ell m}|,
\end{equation}
where \(\zeta_{\ell,R}\) is the first radial eigenfunction in the \(\ell\)-sector.  Thus
\(\Pi_{R,M}\) is the first radial band of the ball, cut off at effective angular energy \(M\).

\begin{lemma}[Ball bracketing above an effective window]\label{lem:ball-bracketing-window}
Fix \(C<\infty\) and \(M\ge1\).  There are constants \(K_C\), independent of \(M\), and \(R_0(C,M,n)\) such that for
every \(R\ge R_0\),
\begin{align}
  \inf \spec\bigl(H_R^0\big|_{\Pi_{R,M}^{\perp}}\bigr)
  &\ge
  \frac{\pi^2}{R^2}
  +\frac{2\pi^2}{R^3}\bigl(\tfrac12 M-b_0-K_C\bigr),
  \label{eq:ball-high-lower}\\
  \inf\spec(H_R^0)
  &\ge
  \frac{\pi^2}{R^2}-\frac{K_C}{R^3}.
  \label{eq:ball-bottom-lower}
\end{align}
In particular, if \(M\) is chosen larger than \(2(C+b_0+K_C+1)\), then the ball has no spectrum contributed by
\(\Pi_{R,M}^{\perp}\) in the window
\[
  \frac{\pi^2}{R^2}+\frac{2\pi^2}{R^3}[-C,C].
\]
\end{lemma}

\begin{proof}
The bottom estimate \eqref{eq:ball-bottom-lower} is the \(\ell=0\), first-root asymptotic
\eqref{eq:sector-asymp}, with the constant enlarged to absorb the finite-\(R\) error.

On \(\Pi_{R,M}^{\perp}\) there are two kinds of states.  The first kind consists of higher radial
roots \(k\ge2\).  The radial potentials are increasing in \(\ell\), because
\(\ell+\alpha\ge1/2\), so the smallest higher radial root occurs in the \(\ell=0\) sector.  The same
one-dimensional phase argument used for Proposition~\ref{prop:sector-asymp}, now at the second
zero, gives
\[
  \nu_{20}(R)=\frac{4\pi^2}{R^2}+O(R^{-3}).
\]
Consequently \(\nu_{k\ell}(R)\ge 4\pi^2R^{-2}+O(R^{-3})\) for all \(k\ge2\).  This is much larger
than \(\pi^2/R^2+O(R^{-3})\) and therefore satisfies \eqref{eq:ball-high-lower} after increasing
\(R_0\) for every fixed \(M\).

The second kind consists of first radial roots with \(T_n(\ell)>M\).  The potential
\((\ell+\alpha)(\ell+\alpha-1)\sinh^{-2}r\) is increasing in \(\ell\) because
\(\ell+\alpha\ge1/2\).  Hence the smallest such first root occurs at the first degree
\(\ell_M\) with \(T_n(\ell_M)>M\).  For fixed \(M\), Proposition~\ref{prop:sector-asymp} gives
\[
  \nu_{1\ell_M}(R)
  =\frac{\pi^2}{R^2}
  +\frac{2\pi^2}{R^3}\bigl(T_n(\ell_M)-b_0\bigr)
  +O_M(R^{-4}).
\]
Since \(T_n(\ell_M)>M\), the coefficient of \(R^{-3}\) is at least \(M-b_0\) up to an error which,
for this fixed \(M\), is smaller than \(M/2+K_C\) once \(R\ge R_0(C,M,n)\).  This gives
\eqref{eq:ball-high-lower}.  The factor \(1/2\) leaves room for the fixed-\(M\) remainder before the final limit
\(M\to\infty\).
\end{proof}

\begin{lemma}[Spectral cutoff for low graph eigenfunctions]\label{lem:graph-spectral-cutoff}
Fix \(C<\infty\) and \(\delta>0\).  There exist \(M=M(C,\delta,n)\) and \(R_0\) such that the
following holds for every \(R\ge R_0\), every graph deficit \(0\le V\le\log2\), and every
\(u\in H_0^1(\Omega_{R,V})\).  Let \(\widetilde u\) be its zero extension to the ball \(B_R\).  If
\(\|\widetilde u\|_{L^2(B_R)}=1\) and
\begin{equation}\label{eq:low-rayleigh-assumption}
  \mathfrak q_{R,V}[u]
  \le
  \left(\frac{\pi^2}{R^2}+\frac{2\pi^2C}{R^3}\right)\|u\|^2,
\end{equation}
then
\begin{equation}\label{eq:graph-cutoff-conclusion}
  \|(I-\Pi_{R,M})\widetilde u\|_{L^2(B_R)}\le \delta.
\end{equation}
The same conclusion holds for every vector in a spectral subspace of \(H_{R,V}\) whose spectrum is
contained in the window in \eqref{eq:low-rayleigh-assumption}.
\end{lemma}

\begin{proof}
Because \(\Omega_{R,V}\subset B_R\), the zero extension belongs to \(H_0^1(B_R)\), and its ball form
is exactly the graph-domain form:
\[
  \mathfrak q_{B_R}[\widetilde u]=\mathfrak q_{R,V}[u].
\]
Let \(P=\Pi_{R,M}\), \(Q=I-P\), and put
\[
  \Lambda_C(R)=\frac{\pi^2}{R^2}+\frac{2\pi^2C}{R^3}.
\]
By Lemma~\ref{lem:ball-bracketing-window},
\[
  \mathfrak q_{B_R}[\widetilde u]
  \ge \gamma_0(R)\|P\widetilde u\|^2+
       \gamma_M(R)\|Q\widetilde u\|^2,
\]
where
\[
  \gamma_0(R)\ge \frac{\pi^2}{R^2}-\frac{K_C}{R^3},
  \qquad
  \gamma_M(R)\ge
  \frac{\pi^2}{R^2}+\frac{2\pi^2}{R^3}(\tfrac12 M-b_0-K_C).
\]
Since \(\|P\widetilde u\|^2+\|Q\widetilde u\|^2=1\) and
\(\mathfrak q_{B_R}[\widetilde u]\le\Lambda_C(R)\), we obtain
\[
  \|Q\widetilde u\|^2
  \le
  \frac{\Lambda_C(R)-\gamma_0(R)}{\gamma_M(R)-\gamma_0(R)}
  \le
  \frac{K'_C}{M-K'_C}
\]
for a constant \(K'_C\) depending only on \(C\) and \(n\).  Choosing \(M\) so that the last quantity
is at most \(\delta^2\) proves \eqref{eq:graph-cutoff-conclusion}.  The spectral-subspace version
follows by applying the same argument to every normalized vector in the subspace.

\end{proof}

\begin{lemma}[Finite first-band trace on moving graphs]\label{lem:first-band-moving-trace}
Fix \(L<\infty\), \(K<\infty\), and a collar width \(a>0\).  Let
\[
  w(r,\theta)=\sum_{T_n(\ell)\le L}\sum_m a_{\ell m}\zeta_{\ell,R}(r)Y_{\ell m}(\theta)
\]
be a ball first-band vector in the constant-end cylinder, where the radial functions are normalized in \(L^2(0,R)\).  We denote the ball first-band coefficient size by
\[
  A_L^2=\sum |a_{\ell m}|^2 .
\]
If \(0\le V\le\log2\) and \(\Lip V\le K\), then
\begin{equation}\label{eq:first-band-moving-trace}
  \|\gamma_{R,V}w\|_{H^{1/2}(\Sn)}
  \le C_{L,K}R^{-3/2}A_L.
\end{equation}
The analogous finite-window collar estimate also holds, with constants depending on \(L\), because only finitely many radial and angular branches are involved.
\end{lemma}

\begin{proof}
For \(T_n(\ell)\le L\) there are only finitely many \(\ell\)'s.  Lemma~\ref{lem:constant-band} and the one-dimensional boundary equation give, uniformly in this finite set,
\[
  \zeta_{\ell,R}(R)=0,\qquad
  |\partial_r^j\zeta_{\ell,R}(R-t)|\le C_{L,j}R^{-3/2},
  \qquad 0\le t\le a,
\]
for \(j=0,1\).  Taylor expansion from \(R\) to \(R-V(\theta)\) gives the trace estimate.  Since the angular window is finite, all Sobolev norms on the sphere are equivalent on it.  Integrating over a collar of fixed width gives the collar-energy version.
\end{proof}

\begin{lemma}[Finite first band versus regular continuation]\label{lem:finite-first-band-regular-comparison}
Fix an angular cutoff \(L\), an effective window \([-C,C]\), and a Lipschitz bound \(K\).  Let
\[
  w_L(r,\theta)=\sum_{T_n(\ell)\le L,m} a_{\ell m}\zeta_{\ell,R}(r)Y_{\ell m}(\theta),
  \qquad A_L^2=\sum |a_{\ell m}|^2,
\]
be a ball first-band vector.  For every \(|\zeta|\le C\) and every
\(V\in\mathfrak C_K\), there is a unique coefficient vector
\(f_L\in P_LL^2(\Sn)\), depending linearly on the coefficients \(a_{\ell m}\), such that
\begin{equation}\label{eq:first-band-regular-coeff-norm}
  c_{L,C}R^{-1/2}A_L\le \|f_L\|\le C_{L,C}R^{-1/2}A_L
\end{equation}
and
\begin{equation}\label{eq:first-band-regular-comparison}
  \|\mathcal S(k_R(\zeta))f_L-w_L\|_{L^2(\Omega_{R,V})}
  \le C_{L,C,K}R^{-1}A_L .
\end{equation}
Moreover
\begin{equation}\label{eq:first-band-regular-trace-bound}
  \|\gamma_{R,V}\mathcal S(k_R(\zeta))f_L\|_{H^{1/2}(\Sn)}
  \le C_{L,C,K}R^{-3/2}A_L .
\end{equation}
\end{lemma}

\begin{proof}
Everything is finite dimensional once \(L\) is fixed.  It is enough to work in a single
spherical harmonic sector.  Let \(\sigma_\ell(r,k)\) denote the sine-amplitude regular branch in
that sector, and let \(\zeta_{\ell,R}\) be the normalized first eigenfunction for the constant end
\(R\).  By the phase expansion of Lemma~\ref{lem:low-energy-phase}, uniformly for the finitely many
\(\ell\)'s with \(T_n(\ell)\le L\),
\[
  \sigma_\ell(r,k_R(\zeta))
  =\sin\{k_R(\zeta)(r+b_\ell)\}+O_{L,C}(R^{-2})
\]
in \(L^\infty(0,R-\log2)\), and the same estimate holds for one radial derivative in the fixed
outer collar.  The first normalized ball mode satisfies the same expansion with
\(k=k_{\ell,1}(R)=\pi/R+O_L(R^{-2})\), after multiplying by the normalizing factor
\((2/R)^{1/2}(1+O_L(R^{-1}))\).  Since
\(|k_R(\zeta)-k_{\ell,1}(R)|=O_{L,C}(R^{-2})\), the accumulated phase difference over a length
\(O(R)\) is \(O_{L,C}(R^{-1})\).  Thus, after taking
\(f_{\ell m}=(2/R)^{1/2}(1+O_{L,C}(R^{-1}))a_{\ell m}\) in each sector, the sectorwise
\(L^2(0,R-\log2)\)-difference is \(O_{L,C}(R^{-1})|a_{\ell m}|\).  Summing over the fixed finite
window gives \eqref{eq:first-band-regular-coeff-norm} and \eqref{eq:first-band-regular-comparison}.
The moving-boundary trace estimate follows from \cref{lem:first-band-moving-trace} for \(w_L\),
from the derivative version of the same finite-window phase comparison in the outer collar, and
from the Lipschitz bound on \(V\).  The composition formula
\[
  \nabla_\theta u(R-V(\theta),\theta)
  = (\nabla_\theta u)(R-V(\theta),\theta)
    - (\partial_r u)(R-V(\theta),\theta)\nabla_\theta V(\theta)
\]
shows that the \(H^{1/2}\)-norm is controlled by the finite angular window and the preceding
\(L^2\)-trace estimates.  This proves \eqref{eq:first-band-regular-trace-bound}.
\end{proof}

\begin{remark}[Boundary normalizations in a fixed window]\label{rem:boundary-normalizations}
Two coefficient normalizations occur in the proof.  A ball first-band vector is written
\[
  w_L=\sum_{T_n(\ell)\le L,m}a_{\ell m}\zeta_{\ell,R}Y_{\ell m},
  \qquad A_L^2=\sum |a_{\ell m}|^2,
\]
so that \(\|w_L\|_{L^2(B_R)}=A_L\).  The sine-amplitude regular continuation coefficient
\(f_L\) is normalized instead by the boundary equation; on every fixed angular window,
\[
  A_L\simeq R^{1/2}\|f_L\| .
\]
Consequently a raw moving-boundary trace of size \(R^{-3/2}A_L\) is the same as a raw trace of
size \(R^{-1}\|f_L\|\).  After the normalization
\(\mathfrak B_R=-(R/\pi)B_R^{\mathrm{raw}}\), this corresponds to a bounded normalized residual.
Small normalized residuals are needed for quasimodes; bounded residuals are enough for the
finite trace correction to be \(o_R(1)\) in bulk \(L^2\).
\end{remark}

The Dirichlet-to-core comparison uses a different invariant from the Schur-to-Dirichlet comparison.  A true
Dirichlet eigenfunction need not produce a small raw boundary trace after one simply projects its zero
extension to the ball first-band window.  The cancellation already appears in the constant-deficit
interval model: at the shifted endpoint \(R-a\),
\[
  \sin\frac{\pi(R-a)}R=\frac{\pi a}{R}+O(R^{-3}),
  \qquad
  \sin\frac{2\pi(R-a)}R=-\frac{2\pi a}{R}+O(R^{-3}),
\]
so the combination
\[
  \sin\frac{\pi r}{R}+\frac12\sin\frac{2\pi r}{R}
\]
has trace \(O(R^{-3})\) at \(r=R-a\), whereas its first-mode projection alone has trace
\(\pi a/R+O(R^{-3})\).  Thus small boundary trace is not stable under naive first-band projection.
What is stable under the Ritz comparison is instead closeness to a finite-dimensional Dirichlet trial
space obtained after the finite trace correction.  Boundary residual smallness is used in the opposite
direction, from finite Schur zeros to Dirichlet quasimodes, through Lemma~\ref{lem:collar-corrected-trial-map}.

\begin{lemma}[Linear corrected projection of exact low states]\label{lem:low-eigenfunction-coefficient-localization}
Fix a compact effective window \([-C,C]\), an integer \(J\ge0\), a Lipschitz bound \(K\), and
\(\delta>0\).  After choosing \(L\) large enough and then \(R\) large enough, the following holds
uniformly for every smooth graph deficit \(V\in\mathfrak C_K\).  Let
\(\mathscr E_{R,V}(C,J)\) denote any spectral subspace spanned by at most \(J+1\) Dirichlet
eigenfunctions whose effective parameters lie in \([-C,C]\).  Then there is a linear map
\[
  \mathcal A_{R,L,V}:\mathscr E_{R,V}(C,J)\longrightarrow P_LL^2(\Sn)
\]
such that
\begin{equation}\label{eq:linear-corrected-localization}
  \left\|u-
  \mathcal T_{R,V}(0)\mathcal A_{R,L,V}u\right\|_{L^2(\Omega_{R,V})}
  \le (\delta+o_R(1))\|u\|_{L^2(\Omega_{R,V})},
  \qquad u\in\mathscr E_{R,V}(C,J).
\end{equation}
Moreover, once \(L\) is chosen so that the left side of \eqref{eq:linear-ball-localization} below is
less than \(\|u\|/2\), the coefficient scale is
\begin{equation}\label{eq:linear-corrected-coefficient-scale}
  c_{C,J,K,L}R^{-1/2}\|u\|
  \le \|\mathcal A_{R,L,V}u\|
  \le C_{C,J,K,L}R^{-1/2}\|u\|,
\end{equation}
and the rescaled form
\[
  \mathfrak h_R[a,b]
  =\frac{R^3}{2\pi^2}\bigl(\mathfrak q_{R,V}[a,b]-E_R\langle a,b\rangle\bigr)
\]
is bounded on
\[
  \mathscr E_{R,V}(C,J)+F^D_{R,L}(V),
  \qquad
  F^D_{R,L}(V):=\mathcal T_{R,V}(0)P_LL^2(\Sn),
\]
by a constant depending only on \(C,J,K,L\).
\end{lemma}

\begin{proof}
We make the construction linear because the subspace angle estimate is what is used in the
Ritz--Grushin comparison.  Let \(\widetilde u\) be the zero extension of \(u\) to the ball \(B_R\).
By the spectral-subspace version of Lemma~\ref{lem:graph-spectral-cutoff}, choose \(L\) so that, for
all large \(R\),
\begin{equation}\label{eq:linear-ball-localization}
  \|(I-\Pi_{R,L})\widetilde u\|_{L^2(B_R)}
  \le \frac{\delta}{4}\|u\|,
  \qquad u\in\mathscr E_{R,V}(C,J).
\end{equation}
Write
\[
  \Pi_{R,L}\widetilde u
  =\mathcal J_Ra(u)
  =\sum_{T_n(\ell)\le L,m}a_{\ell m}(u)\zeta_{\ell,R}Y_{\ell m}.
\]
The map \(u\mapsto a(u)\) is linear.  The projection estimate gives \(\|a(u)\|\le \|u\|+o_R(1)\|u\|\),
and, after the chosen \(L\), also \(\|a(u)\|\ge c\|u\|\) on the subspace being followed.

Define
\[
  \mathcal A_{R,L,V}u:=\sqrt{\frac2R}\,a(u)\in P_LL^2(\Sn).
\]
By the finite-window conditioning estimate \eqref{eq:regular-continuation-first-band-close} with
\(\zeta=0\),
\begin{equation}\label{eq:linear-regular-continuation-close}
  \left\|\mathcal S(k_R(0))\mathcal A_{R,L,V}u-
  \mathcal J_Ra(u)\right\|_{L^2(\Omega_{R,V})}
  \,\le o_R(1)\|u\|.
\end{equation}
The coefficient scale \eqref{eq:linear-corrected-coefficient-scale} follows from
\(\|\mathcal A_{R,L,V}u\|=\sqrt{2/R}\|a(u)\|\).

It remains to subtract the finite trace.  The raw trace of the fixed-window continuation
\(\mathcal S(k_R(0))\mathcal A_{R,L,V}u\) belongs to the finite trace range of
Lemma~\ref{lem:finite-trace-regularity}.  In the notation of Remark~\ref{rem:boundary-normalizations},
its normalized boundary residual is uniformly bounded on the fixed window.  Hence
Lemma~\ref{lem:collar-corrected-trial-map} gives
\begin{equation}\label{eq:linear-correction-small}
  \left\|\mathcal E_{R,V,L,0}\mathcal A_{R,L,V}u\right\|_{L^2(\Omega_{R,V})}
  =o_R(1)\|u\|.
\end{equation}
Combining \eqref{eq:linear-ball-localization}, \eqref{eq:linear-regular-continuation-close}, and
\eqref{eq:linear-correction-small} proves \eqref{eq:linear-corrected-localization}.

Finally, the rescaled form is bounded on each summand and on the mixed terms.  On
\(\mathscr E_{R,V}(C,J)\), this is the eigenvalue equation and the bound \(|\zeta|\le C\).  On
\(F^D_{R,L}(V)\), it is the fixed finite Ritz model in Lemma~\ref{lem:finite-boundary-quasimode}.
For a mixed term, if \(e\in\mathscr E_{R,V}(C,J)\) has effective eigenvalue \(\zeta_e\), then for
any \(v\in H_0^1(\Omega_{R,V})\)
\[
  \mathfrak h_R[e,v]=(\zeta_e+O(R^{-1}))\langle e,v\rangle,
\]
which is uniformly bounded by \(C\|e\|\|v\|\).  This proves the final assertion.
\end{proof}

\begin{lemma}[Ritz comparison from a captured exact cluster]\label{lem:abstract-ritz-grushin}
Let \(A_R\) be a self-adjoint operator with closed form \(\mathfrak a_R\), and fix a reference
energy \(E_R^0\).  Set
\[
  \mathfrak h_R[u,v]
  :=\frac{R^3}{2\pi^2}
  \bigl(\mathfrak a_R[u,v]-E_R^0\langle u,v\rangle\bigr).
\]
Let
\[
  \mu_0^{(R)}\le\mu_1^{(R)}\le\cdots\le \mu_J^{(R)}
\]
be the first \(J+1\) rescaled eigenvalues of \(A_R\), and let \(\mathscr E_J\) be their spectral
span.  Let \(F_R\) be a finite-dimensional subspace of the form domain, and let
\[
  r_0^{(R)}\le r_1^{(R)}\le\cdots\le r_J^{(R)}
\]
be the first \(J+1\) Ritz values of \(\mathfrak h_R\) on \(F_R\).  Suppose that, for some
\(0<\delta<1/4\) and \(K<\infty\),
\begin{enumerate}[label=(\roman*)]
\item \(\|(I-P_F)P_{\mathscr E_J}\|\le\delta\);
\item the rescaled form is bounded on \(\mathscr E_J+F_R\):
\[
  |\mathfrak h_R[u,v]|\le K\|u\|\|v\|,
  \qquad u,v\in\mathscr E_J+F_R .
\]
\end{enumerate}
Then, for \(0\le j\le J\),
\begin{equation}\label{eq:captured-ritz-two-sided}
  |\mu_j^{(R)}-r_j^{(R)}|\le C_{J,K}\delta .
\end{equation}
The same conclusion holds for a cluster inside a bounded interval after grouping eigenvalues at the
scale \(C_{J,K}\delta\).  In particular, no separate reverse-angle hypothesis is needed: the upper
bound comes from the Ritz trial space, and the lower bound comes from projecting the exact cluster
into that space.
\end{lemma}

\begin{proof}
The min--max principle gives one side immediately.  Since \(F_R\) is an admissible trial space,
\[
  \mu_j^{(R)}\le r_j^{(R)},\qquad 0\le j\le J .
\]
For the reverse inequality, let \(E_j\subset\mathscr E_J\) be the span of the first \(j+1\) exact
rescaled eigenvectors.  The map \(U=P_F|_{E_j}\) is injective and
\(\|U^*U-I\|\le 3\delta\).  After the polar correction
\(\widehat U=U(U^*U)^{-1/2}\), we have an isometric embedding
\(\widehat U:E_j\to F_R\) satisfying
\[
  \|\widehat Ue-e\|\le C\delta\|e\|,
  \qquad e\in E_j .
\]
Using the form bound on \(\mathscr E_J+F_R\),
\[
  |\mathfrak h_R[\widehat Ue,\widehat Ue']-\mathfrak h_R[e,e']|
  \le C_K\delta\|e\|\|e'\| .
\]
Thus the matrix of \(\mathfrak h_R\) on \(\widehat U E_j\) is within \(C_K\delta\) of the exact
matrix on \(E_j\).  Its largest eigenvalue is therefore at most \(\mu_j^{(R)}+C_K\delta\).  By
min--max for the finite Ritz problem,
\[
  r_j^{(R)}\le \mu_j^{(R)}+C_{J,K}\delta .
\]
Together with the trial-space upper bound this proves \eqref{eq:captured-ritz-two-sided}.  The
cluster version is the same argument applied to the grouped spectral subspace; grouping only avoids
ordering ambiguities at the boundary of a cluster.
\end{proof}

\begin{lemma}[One-sided Ritz enclosure from a corrected trial space]\label{lem:one-sided-ritz-enclosure}
Let \(A_R\) be a self-adjoint operator with closed form \(\mathfrak a_R\), and set
\[
  \mathfrak h_R[u,v]
  =\frac{R^3}{2\pi^2}\bigl(\mathfrak a_R[u,v]-E_R^0\langle u,v\rangle\bigr).
\]
Let \(\mathscr E_R\) be a finite-dimensional exact spectral subspace whose rescaled eigenvalues
lie in a compact interval \(I\), and let \(F_R\) be a finite-dimensional subspace of the form domain.
Assume
\[
  \|(I-P_F)P_{\mathscr E_R}\|\le \delta<1/4
\]
and that \(\mathfrak h_R\) is bounded by \(K\) on \(\mathscr E_R+F_R\).  Then every exact rescaled
eigenvalue in \(I\) lies within \(C_K\delta\) of a Ritz value of \(\mathfrak h_R|_{F_R}\), counted
with multiplicity after grouping clusters at distance \(C_K\delta\).  The constant depends only on
\(K\) and on the dimension of the cluster being followed.
\end{lemma}

\begin{proof}
Let \(U=P_F|_{\mathscr E_R}\).  The angle assumption gives
\(\|U^*U-I\|\le 3\delta\).  The polar correction
\(\widehat U=U(U^*U)^{-1/2}\) is a unitary map from \(\mathscr E_R\) onto the subspace
\(\widehat U\mathscr E_R\subset F_R\), and
\(\|\widehat Ue-e\|\le C\delta\|e\|\).  Since the rescaled form is bounded on
\(\mathscr E_R+F_R\),
\[
  |\mathfrak h_R[\widehat Ue,\widehat Ue']-\mathfrak h_R[e,e']|
  \le C_K\delta\|e\|\|e'\| .
\]
Thus the matrix of the exact cluster and the Ritz matrix restricted to \(\widehat U\mathscr E_R\)
differ by \(O_K(\delta)\).  The finite-dimensional min-max principle gives the claimed one-sided
enclosure.  Grouping clusters prevents endpoint ordering ambiguities.
\end{proof}

\begin{lemma}[Finite-dimensional Rouch\'e comparison]\label{lem:rouche-finite-pencils}
Let \(F(\zeta)\) and \(G(\zeta)\) be holomorphic families of linear maps on a fixed finite-dimensional Hilbert space, and let \(U\Subset\mathbb C\) have piecewise smooth boundary.  Suppose \(G(\zeta)\) is invertible on \(\partial U\) and
\[
  \sup_{\zeta\in\partial U}\|(F(\zeta)-G(\zeta))G(\zeta)^{-1}\|<1.
\]
Then \(F\) and \(G\) have the same number of characteristic zeros in \(U\), counted with algebraic multiplicity.  In particular, if two self-adjoint finite pencils are uniformly close on contours separating a cluster, their characteristic values in that cluster differ by the size of the perturbation.
\end{lemma}

\begin{proof}
The determinant identity
\[
  \det F(\zeta)=\det G(\zeta)\,
  \det\{I+(F(\zeta)-G(\zeta))G(\zeta)^{-1}\}
\]
shows that the second determinant is nonvanishing on \(\partial U\) and has winding number zero there.  The argument principle therefore gives the same number of zeros for \(\det F\) and \(\det G\) in \(U\).  The final statement follows by taking \(U\) to be a union of small contours around the real cluster and using the self-adjoint ordering of the characteristic values.
\end{proof}

\begin{lemma}[Gram matrix for the fixed corrected Ritz space]\label{lem:fixed-corrected-gram-estimates}
Fix finite numbers \(L,C<\infty\) and a Lipschitz bound \(K\).  Let \(V\in\mathfrak C_K\) be smooth.
For \(f,g\in P_LL^2(\Sn)\), put
\[
  u_f=\mathcal S(k_R(0))f,
  \qquad e_f=\mathcal E_{R,V,L,0}f,
  \qquad T_f=\mathcal T_{R,V}(0)f=u_f-e_f .
\]
Then, uniformly for \(f,g\in P_LL^2(\Sn)\),
\begin{equation}\label{eq:fixed-ritz-gram-bilinear}
  \langle T_f,T_g\rangle_{L^2(\Omega_{R,V})}
  =\frac R2\langle f,g\rangle_{L^2(\Sn)}+O_{L,C,K}(1)\|f\|\|g\| .
\end{equation}
\end{lemma}

\begin{proof}
The regular continuation estimate gives
\[
  \langle u_f,u_g\rangle=\frac R2\langle f,g\rangle+O_{L,C,K}(1)\|f\|\|g\| .
\]
At \(\zeta=0\), \cref{lem:boundary-grushin-expansion} gives
\(\|\mathfrak B_R(0,V)f\|\le C_{L,K}\|f\|\) on the fixed window.  Hence the mesoscopic correction satisfies
\[
  \|e_f\|_{L^2}\le C_{L,K}R^{-3/4}\|f\|,
  \qquad
  \|e_f\|_{\mathfrak q,\mathcal C_{R,V}}
  \le C_{L,K}R^{-5/4}\|f\|,
\]
with the angular derivative contribution controlled by the collar estimate of Lemma~\ref{lem:polynomial-collar-calculus}.  Since \(\|u_f\|\simeq R^{1/2}\|f\|\), the cross terms involving \(e_f\) are \(O_{L,C,K}(1)\|f\|\|g\|\) in the Gram matrix.  This proves \eqref{eq:fixed-ritz-gram-bilinear}.
\end{proof}

\begin{lemma}[Boundary-flux form matrix for the fixed corrected Ritz space]\label{lem:fixed-corrected-bilinear-estimates}
Fix finite numbers \(L,C<\infty\) and a Lipschitz bound \(K\).  Let \(V\in\mathfrak C_K\) be smooth, and set
\[
  A_{V,L}:=P_L(T_n+V-b_0)P_L .
\]
For \(f,g\in P_LL^2(\Sn)\), put
\[
  u_f=\mathcal S(k_R(0))f,
  \qquad e_f=\mathcal E_{R,V,L,0}f,
  \qquad T_f=\mathcal T_{R,V}(0)f=u_f-e_f .
\]
Then
\begin{equation}\label{eq:fixed-ritz-form-bilinear}
  \mathfrak h_R[T_f,T_g]
  =\frac R2\langle A_{V,L}f,g\rangle_{L^2(\Sn)}+O_{L,C,K}(R^{1/2})\|f\|\|g\|,
\end{equation}
where
\[
  \mathfrak h_R[a,b]=\frac{R^3}{2\pi^2}\bigl(\mathfrak q_{R,V}[a,b]-E_R\langle a,b\rangle\bigr).
\]
Consequently, using the Gram matrix estimate of Lemma~\ref{lem:fixed-corrected-gram-estimates}, the normalized Ritz matrix of \(\mathfrak h_R\) on
\(\mathcal T_{R,V}(0)P_LL^2(\Sn)\), pulled back to coefficient space, is
\[
  P_L(T_n+V-b_0)P_L+O_{L,C,K}(R^{-1/2}).
\]
\end{lemma}

\begin{proof}
Use Green's identity in the flattened collar.  The regular continuations solve
\((H-E_R)u_f=0\), while \(T_g\) has zero trace on the moving boundary.  Therefore
\[
  \mathfrak q_{R,V}[T_f,T_g]-E_R\langle T_f,T_g\rangle
  =-\bigl(\mathfrak q_{R,V}[e_f,u_g]-E_R\langle e_f,u_g\rangle\bigr)
   +\bigl(\mathfrak q_{R,V}[e_f,e_g]-E_R\langle e_f,e_g\rangle\bigr).
\]
The quadratic correction term is lower order:
\[
  \left|\mathfrak q_{R,V}[e_f,e_g]-E_R\langle e_f,e_g\rangle\right|
  \le C_{L,C,K}R^{-5/2}\|f\|\|g\| .
\]
The linear term is a boundary-flux pairing.  In the collar coordinate \(t=R-V(\theta)-r\), the tangential part of the normal and the graph Jacobian contribute terms of size \(O(e^{-2R+C\ell_R})\), hence below every polynomial order for the mesoscopic collar.  With \(\nu\) denoting the outward Liouville normal,
\[
  \mathfrak q_{R,V}[e_f,u_g]-E_R\langle e_f,u_g\rangle
  =\int_{\Sn} B_R^{\mathrm{raw}}(k_R(0),V)f\,
       \overline{\partial_\nu u_g}\,d\theta
       +O_{L,C,K}(R^{-5/2})\|f\|\|g\| .
\]
The finite boundary expansion at \(\zeta=0\) from \cref{lem:boundary-grushin-expansion} and the finite-window normal derivative estimate give
\[
  B_R^{\mathrm{raw}}(k_R(0),V)f
  =\frac{\pi}{R}A_{V,L}f+O_{L,C,K}(R^{-2})\|f\|,
  \qquad
  \partial_\nu u_g=-\frac{\pi}{R}g+O_{L,C,K}(R^{-2})\|g\| .
\]
The sign comes from the outward normal \(\nu=-\partial_t+O(e^{-R})\nabla_\theta\) while
\(\partial_t=-\partial_r\) points into the collar.  Substitution gives
\[
  \mathfrak q_{R,V}[e_f,u_g]-E_R\langle e_f,u_g\rangle
  =-\frac{\pi^2}{R^2}\langle A_{V,L}f,g\rangle
   +O_{L,C,K}(R^{-5/2})\|f\|\|g\| .
\]
Putting this into the previous identity yields
\[
  \mathfrak q_{R,V}[T_f,T_g]-E_R\langle T_f,T_g\rangle
  =\frac{\pi^2}{R^2}\langle A_{V,L}f,g\rangle
   +O_{L,C,K}(R^{-5/2})\|f\|\|g\| .
\]
Multiplying by \(R^3/(2\pi^2)\) proves \eqref{eq:fixed-ritz-form-bilinear}.  By Lemma~\ref{lem:fixed-corrected-gram-estimates}, the Gram matrix is
\(G=(R/2)I+O(1)\), so \(G^{-1/2}=(2/R)^{1/2}I+O(R^{-3/2})\), and the
\(O(R^{1/2})\) form error becomes \(O(R^{-1/2})\) after normalization.
\end{proof}

\begin{lemma}[Finite boundary model, fixed Ritz model, and residual quasimodes]\label{lem:finite-boundary-quasimode}
Fix an angular cutoff \(L\), an effective window \([-C,C]\), and a Lipschitz bound \(K\).  Let
\(V\in\mathfrak C_K\) be smooth, and put
\[
  A_{V,L}:=P_L(T_n+V-b_0)P_L:P_LL^2(\Sn)\to P_LL^2(\Sn).
\]
Then the finite boundary pencil has the model expansion
\begin{equation}\label{eq:finite-boundary-effective-model}
  P_L\mathfrak B_R(\zeta,V)P_L
  =\zeta I-A_{V,L}+O_{L,C,K}(R^{-1})
\end{equation}
uniformly for \(|\zeta|\le C\).

At the fixed parameter \(\zeta=0\), let
\[
  F^D_{R,L}(V):=\mathcal T_{R,V}(0)P_LL^2(\Sn)\subset H^1_0(\Omega_{R,V}).
\]
Define its Gram matrix and rescaled form matrix by
\[
  G^D_{R,L}(V)[f,g]
  =\langle \mathcal T_{R,V}(0)f,\mathcal T_{R,V}(0)g\rangle,
\]
\[
  H^D_{R,L}(V)[f,g]
  =\frac{R^3}{2\pi^2}
  \bigl(\mathfrak q_{R,V}[\mathcal T_{R,V}(0)f,\mathcal T_{R,V}(0)g]
  -E_R\langle \mathcal T_{R,V}(0)f,\mathcal T_{R,V}(0)g\rangle\bigr).
\]
After bulk normalization, the finite corrected Ritz matrix satisfies
\begin{equation}\label{eq:finite-corrected-ritz-fixed-model}
  (G^D_{R,L})^{-1/2}H^D_{R,L}(G^D_{R,L})^{-1/2}
  =A_{V,L}+O_{L,C,K}(R^{-1/2}).
\end{equation}
Consequently the Ritz values of the fixed corrected Dirichlet trial space and the eigenvalues of
\(A_{V,L}\) differ by \(O_{L,C,K}(R^{-1/2})\), counted with multiplicity on separating contours.

Finally, if the full normalized boundary residual satisfies
\begin{equation}\label{eq:full-residual-for-ritz-defect}
  \rho(f):=\frac{\|\mathfrak B_R(\zeta,V)f\|}{\|f\|}\le \rho,
\end{equation}
then the corrected finite trial vector
\(\mathcal T_{R,V}(\zeta)f\in H_0^1(\Omega_{R,V})\) satisfies
\begin{equation}\label{eq:full-residual-ritz-defect}
  \left|
  \frac{R^3}{2\pi^2}
  \frac{\mathfrak q_{R,V}[\mathcal T_{R,V}(\zeta)f]-E_R\|\mathcal T_{R,V}(\zeta)f\|^2}
       {\|\mathcal T_{R,V}(\zeta)f\|^2}
  -\zeta
  \right|
  \le C_{L,C,K}\bigl(\rho+R^{-1/2}\rho^2+R^{-1/2}\bigr).
\end{equation}
The estimate \eqref{eq:full-residual-ritz-defect} is applied to vectors satisfying the full residual bound \eqref{eq:full-residual-for-ritz-defect}.  In the finite-window reduction such vectors are produced by the Schur lift in \cref{lem:ordered-boundary-schur-residual}; projected finite-window residuals enter inside the Schur system.
\end{lemma}

\begin{proof}
The boundary expansion \eqref{eq:finite-boundary-effective-model} is
\cref{lem:boundary-grushin-expansion} restricted to \(P_LL^2\).

The fixed corrected Ritz model is the content of
\cref{lem:fixed-corrected-bilinear-estimates}.  Indeed, that lemma gives
\[
  G^D_{R,L}=\frac R2 I+O_{L,C,K}(1),
  \qquad
  H^D_{R,L}=\frac R2 A_{V,L}+O_{L,C,K}(R^{1/2}).
\]
Since \(G^D_{R,L}=(R/2)(I+O_{L,C,K}(R^{-1}))\), functional calculus on the fixed finite-dimensional
space gives
\[
  (G^D_{R,L})^{-1/2}=\left(\frac2R\right)^{1/2}(I+O_{L,C,K}(R^{-1})).
\]
Consequently
\[
  (G^D_{R,L})^{-1/2}H^D_{R,L}(G^D_{R,L})^{-1/2}
  =A_{V,L}+O_{L,C,K}(R^{-1/2}),
\]
which is \eqref{eq:finite-corrected-ritz-fixed-model}.  The finite-dimensional spectral comparison
then follows from the min--max principle, or equivalently from \cref{lem:rouche-finite-pencils} on
separating contours.

The residual-to-quasimode estimate is \cref{lem:collar-corrected-trial-map}.  It is stated here in terms of the full boundary pencil.  The finite-window residual enters only through the Schur lift in \cref{lem:ordered-boundary-schur-residual}, where the full residual is estimated.
\end{proof}

\begin{remark}[Two finite models]\label{rem:two-finite-models}
The boundary pencil is linear in the spectral parameter and has model \(\zeta I-A_{V,L}\), while the fixed corrected Dirichlet Ritz space has model \(A_{V,L}\).  Exact low states are compared with the Ritz matrix.  Conversely, finite Schur zeros are lifted to vectors with controlled full boundary residual and then converted to Dirichlet quasimodes by \eqref{eq:full-residual-ritz-defect}.
\end{remark}

\begin{lemma}[From basis localization to subspace angle]\label{lem:basis-to-subspace-angle}
Let \(\mathscr E\) be a \(d\)-dimensional subspace of a Hilbert space with orthonormal basis
\(u_1,\ldots,u_d\).  Let \(F\) be another closed subspace.  Suppose that for each \(i\) there is
\(v_i\in F\) such that
\[
  \|u_i-v_i\|\le \varepsilon .
\]
Then, with \(P_F\) denoting the orthogonal projection onto \(F\),
\begin{equation}\label{eq:basis-localization-subspace-angle}
  \|(I-P_F)P_{\mathscr E}\|\le d^{1/2}\varepsilon .
\end{equation}
If \(\varepsilon<1/(2d^{1/2})\), then the map \(P_F|_{\mathscr E}:\mathscr E\to F\) is injective
and its image is a \(d\)-dimensional subspace of \(F\).
\end{lemma}

\begin{proof}
For a unit vector \(u=\sum_{i=1}^d c_i u_i\in\mathscr E\), set \(v=\sum_i c_i v_i\in F\).  Then
\[
  \operatorname{dist}(u,F)\le \|u-v\|
  \le \sum_i |c_i|\,\|u_i-v_i\|
  \le \varepsilon\left(\sum_i |c_i|^2\right)^{1/2}d^{1/2}
  =d^{1/2}\varepsilon .
\]
Taking the supremum over unit vectors in \(\mathscr E\) proves \eqref{eq:basis-localization-subspace-angle}.
If the right side is below one, then no nonzero vector in \(\mathscr E\) is projected to zero by
\(P_F\), proving injectivity; the dimension statement follows.
\end{proof}

\begin{lemma}[Exact clusters are captured by the fixed corrected space]\label{lem:exact-cluster-corrected-angle}
Fix a compact effective window \([-C,C]\), an integer \(J\ge0\), a Lipschitz bound \(K\), and
\(\delta>0\).  After choosing \(L\) large and then \(R\) large, the following holds uniformly for
smooth \(V\in\mathfrak C_K\).  Let \(\mathscr E_R\) be any Dirichlet spectral subspace of
\(H_{R,V}\) spanned by at most \(J+1\) eigenfunctions with effective parameters in \([-C,C]\), and set
\[
  F^D_{R,L}(V)=\mathcal T_{R,V}(0)P_LL^2(\Sn)\subset H_0^1(\Omega_{R,V}).
\]
Then
\begin{equation}\label{eq:exact-cluster-fixed-angle}
  \|(I-P_{F^D_{R,L}})P_{\mathscr E_R}\|
  \le C_J\delta+o_R(1).
\end{equation}
Moreover the rescaled form
\[
  \mathfrak h_R[a,b]
  =\frac{R^3}{2\pi^2}\bigl(\mathfrak q_{R,V}[a,b]-E_R\langle a,b\rangle\bigr)
\]
is bounded on \(\mathscr E_R+F^D_{R,L}(V)\) by a constant depending only on \(C,J,K,L\).  Consequently,
every exact rescaled eigenvalue in \([-C,C]\) is within \(C_{C,J,K}\delta+o_R(1)\) of a Ritz
value of \(\mathfrak h_R\) on \(F^D_{R,L}(V)\), counted with multiplicity after grouping clusters at
that scale.
\end{lemma}

\begin{proof}
Take an orthonormal eigenbasis \(u_1,\ldots,u_d\) of \(\mathscr E_R\), with \(d\le J+1\).  Apply
\cref{lem:low-eigenfunction-coefficient-localization} to each \(u_i\).  This gives vectors
\[
  v_i=\mathcal T_{R,V}(0)\mathcal A_{R,L,V}u_i\in F^D_{R,L}(V)
\]
satisfying
\[
  \|u_i-v_i\|\le \delta+o_R(1),\qquad i=1,\ldots,d .
\]
The linear-algebra estimate \cref{lem:basis-to-subspace-angle} gives
\eqref{eq:exact-cluster-fixed-angle}.  The form boundedness on the sum is the final assertion of
\cref{lem:low-eigenfunction-coefficient-localization}.  Finally, applying
\cref{lem:one-sided-ritz-enclosure} with \(F_R=F^D_{R,L}(V)\) gives the stated one-sided spectral
enclosure.  The grouping of clusters is precisely the one in \cref{lem:one-sided-ritz-enclosure}; it
prevents endpoint ordering ambiguities and preserves multiplicity at the enlarged scale.
\end{proof}

\begin{lemma}[Exact-to-finite Ritz comparison]\label{lem:fixed-corrected-ritz-space}
Fix a compact effective window \([-C,C]\), an integer \(J\ge0\), a Lipschitz bound \(K\), and
\(\delta>0\).  After choosing \(L\) large enough and then \(R\) large enough, the fixed finite
Dirichlet trial space
\[
  F^D_{R,L}(V):=\mathcal T_{R,V}(0)P_LL^2(\Sn)
  \subset H_0^1(\Omega_{R,V})
\]
captures every Dirichlet spectral subspace in
\[
  E_R+\frac{2\pi^2}{R^3}[-C,C]
\]
spanned by at most \(J+1\) eigenfunctions with subspace angle at most
\(C_{C,J,K}\delta+o_R(1)\), uniformly for smooth \(V\in\mathfrak C_K\).  Consequently every exact
rescaled eigenvalue in this window is within \(C_{C,J,K}\delta+o_R(1)\) of a Ritz value on
\(F^D_{R,L}(V)\), counted after grouping clusters at that scale.  Moreover, if
\(\mathcal H^D_{R,L}(V)\) denotes the normalized rescaled Ritz matrix of \(\mathfrak h_R\) on
\(F^D_{R,L}(V)\), pulled back to \(P_LL^2(\Sn)\) by the coefficient map, then
\begin{equation}\label{eq:fixed-corrected-ritz-matrix}
  \mathcal H^D_{R,L}(V)
  =P_L(T_n+V-b_0)P_L+O_{L,C,K}(R^{-1/2})
\end{equation}
in finite-dimensional operator norm.
\end{lemma}

\begin{proof}
The dependence on the effective parameter is controlled by
\cref{lem:finite-corrected-parameter-stability}.  In particular, for \(|\zeta|\le C\), the spaces
\(\mathcal T_{R,V}(\zeta)P_LL^2\) and \(F^D_{R,L}(V)=\mathcal T_{R,V}(0)P_LL^2\) have gap
\(O_{L,C,K}(R^{-1})\), and their rescaled forms differ by \(O_{L,C,K}(R^{-1/2})\).

The capture of exact spectral subspaces and the one-sided enclosure of exact effective eigenvalues
by Ritz values are exactly \cref{lem:exact-cluster-corrected-angle}.  This is the exact-to-finite comparison.  The opposite direction, from a finite Schur vector back to a Dirichlet quasimode, is supplied later by \cref{lem:ordered-boundary-schur-residual}, which first estimates the full boundary residual after restoring the shell and the exterior tail.

Finally, \eqref{eq:fixed-corrected-ritz-matrix} is the fixed corrected Ritz model
\eqref{eq:finite-corrected-ritz-fixed-model}.
\end{proof}

\begin{lemma}[Outer cutoff for exact low Ritz values]\label{lem:outer-cutoff-boundary}
Fix a compact effective window \([-C,C]\), an integer \(J\ge0\), a number \(K<\infty\), and
\(\delta>0\).  There are an angular cutoff \(L=L(C,J,K,\delta,n)\) and \(R_0\) such that the
following holds for every smooth graph deficit \(V\in\mathfrak C_K\) and every \(R\ge R_0\).
Let \(\zeta\in[-C,C]\) be the effective parameter corresponding to one of the first \(J+1\)
Dirichlet eigenvalues in the window
\[
  E_R+\frac{2\pi^2}{R^3}[-C,C].
\]
Then \(\zeta\) lies within \(C_{C,K}\delta+o_R(1)\) of the characteristic spectrum of the finite
outer boundary pencil
\[
  P_L\mathfrak B_R(\cdot,V)P_L:P_LL^2(\Sn)\to P_LL^2(\Sn).
\]
Equivalently, after the finite trace correction, every exact low Dirichlet state is detected by the
finite outer boundary problem up to the indicated error in characteristic value distance.
\end{lemma}

\begin{proof}
Let \(\mathscr E_{R,V}(C)\) be the Dirichlet spectral subspace of \(H_{R,V}\) corresponding to
\[
  E_R+\frac{2\pi^2}{R^3}[-C,C].
\]
By \cref{lem:fixed-corrected-ritz-space}, after choosing \(L\) and then \(R\), every exact rescaled
parameter among the first \(J+1\) values in the window lies within \(C_{C,J,K}\delta+o_R(1)\) of a
Ritz value of the fixed corrected matrix \(\mathcal H^D_{R,L}(V)\).  By
\eqref{eq:fixed-corrected-ritz-matrix}, those Ritz values lie within \(O_{L,C,K}(R^{-1/2})\) of the
eigenvalues of
\[
  A_{V,L}=P_L(T_n+V-b_0)P_L.
\]
On the other hand, the finite boundary pencil satisfies
\[
  P_L\mathfrak B_R(\zeta,V)P_L=\zeta I-A_{V,L}+O_{L,C,K}(R^{-1})
\]
by \eqref{eq:finite-boundary-effective-model}.  The finite-dimensional Rouch\'e comparison
\cref{lem:rouche-finite-pencils}, applied on contours separating the finite clusters of
\(A_{V,L}\), shows that the characteristic values of the finite outer boundary pencil are within
\(O_{L,C,K}(R^{-1})\) of the same eigenvalues.  Combining the three estimates gives the claimed
proximity of every exact low effective parameter to the finite outer boundary characteristic
spectrum.  Multiplicity is preserved by the same contour comparison.
\end{proof}

\begin{lemma}[Finite-section convergence for the effective operator]\label{lem:finite-section-effective}
Let \(P_M=\mathbf 1_{[0,M]}(T_n)\).  For each fixed \(N\),
\[
  \lambda_j\bigl(P_M(T_n+V)P_M\big|_{P_ML^2}\bigr)
  \longrightarrow \eta_j(T_n+V),
  \qquad 0\le j\le N,
\]
as \(M\to\infty\), uniformly for \(V\) in compact subsets of \(C(\Sn)\) with a common
\(L^\infty\)-bound.
\end{lemma}

\begin{proof}
For a fixed \(V\), this is Galerkin convergence for the closed form
\[
  \mathfrak t_V[\phi]=\langle T_n^{1/2}\phi,T_n^{1/2}\phi\rangle+
  \int_{\Sn}V|\phi|^2.
\]
Indeed, \(P_M\to I\) strongly in the form domain of \(T_n\), and the min-max principle gives
convergence of every fixed eigenvalue.  Since \(T_n+V\) has compact resolvent on the sphere, the
spectrum is purely discrete and no essential-threshold issue is present.  Uniformity on compact families follows from the Lipschitz
bound
\[
  |\eta_j(T_n+V)-\eta_j(T_n+W)|\le \|V-W\|_\infty
\]
and the same bound for the finite sections, together with a finite \(C(\Sn)\)-net of the compact
family.
\end{proof}

\begin{lemma}[Effective Schur tail]\label{lem:effective-schur-tail}
Let
\[
  G_\infty(\zeta,V):=\zeta I-(T_n+V-b_0)
\]
acting on \(L^2(\Sn)\), and let \(P_M=\mathbf 1_{[0,M]}(T_n)\), \(Q_M=I-P_M\).
Fix \(C<\infty\) and \(\Lambda<\infty\).  If \(|\zeta|\le C\), \(\|V\|_\infty\le \Lambda\), and
\[
  M>|b_0|+C+\Lambda+1,
\]
then the high block
\[
  Q_MG_\infty(\zeta,V)Q_M:Q_ML^2(\Sn)\to Q_ML^2(\Sn)
\]
is invertible and
\begin{equation}\label{eq:effective-high-inverse}
  \bigl\|(Q_MG_\infty(\zeta,V)Q_M)^{-1}\bigr\|
  \le \frac{1}{M-|b_0|-C-\Lambda}.
\end{equation}
Consequently the effective Schur complement
\[
  S_{\infty,M}(\zeta,V)
  :=P_MG_\infty P_M
  -P_MG_\infty Q_M(Q_MG_\infty Q_M)^{-1}Q_MG_\infty P_M
\]
satisfies
\begin{equation}\label{eq:effective-schur-tail}
  \bigl\|S_{\infty,M}(\zeta,V)-P_MG_\infty(\zeta,V)P_M\bigr\|
  \le
  \frac{\Lambda^2}{M-|b_0|-C-\Lambda}.
\end{equation}
In particular, on every bounded effective window, the zeros of the full limiting Grushin family
\(G_\infty(\zeta,V)\) are obtained, up to an error tending to zero as \(M\to\infty\), from a finite
spherical-harmonic matrix.
\end{lemma}

\begin{proof}
On \(Q_ML^2\) one has \(T_n\ge M\).  Hence, for \(u\in Q_ML^2\),
\[
  \bigl\langle (T_n+V-b_0-\zeta)u,u\bigr\rangle
  \ge (M-|b_0|-C-\Lambda)\|u\|^2.
\]
Since \(G_\infty=-(T_n+V-b_0-\zeta)\), this proves \eqref{eq:effective-high-inverse}.  Moreover
\(T_n\), \(b_0\), and \(\zeta\) are diagonal with respect to the splitting \(P_M\oplus Q_M\).  Thus
all off-diagonal blocks of \(G_\infty\) come from multiplication by \(-V\), and
\[
  \|P_MG_\infty Q_M\|+\|Q_MG_\infty P_M\|\le 2\Lambda.
\]
The displayed Schur estimate follows, with the factor written in the symmetric
form \(\Lambda^2/(M-|b_0|-C-\Lambda)\).  The final statement is the analytic Fredholm equivalence between
invertibility of \(G_\infty\) and invertibility of its Schur complement.
\end{proof}

\begin{lemma}[Ordered boundary Schur residual]\label{lem:ordered-boundary-schur-residual}
Fix a compact effective window \([-C,C]\), a Lipschitz bound \(K\), and a smooth family
\(\calB\subset C^\infty(\Sn)\cap\mathfrak C_K\) with a common \(C^{m_0}\)-bound \(\mathfrak S_{m_0}(\calB)<\infty\) as in Proposition~\ref{prop:boundary-pair-reduction}.  Let
\[
  A_V=T_n+V-b_0,
  \qquad
  P_M=\mathbf 1_{[0,M]}(T_n),
  \qquad
  Q_{M,L}=P_L-P_M .
\]
Assume
\[
  M>|b_0|+C+\log2+1.
\]
After fixing \(M\), choose \(L\ge M\), and then let \(R\) be large.  Uniformly for
\(V\in\calB\) and \(|\zeta|\le C\), the finite boundary block
\[
  D^B_{R,M,L}(\zeta,V)
  :=Q_{M,L}\mathfrak B_R(\zeta,V)Q_{M,L}
\]
is invertible on \(Q_{M,L}L^2(\Sn)\), and
\begin{equation}\label{eq:boundary-shell-inverse}
  \|(D^B_{R,M,L})^{-1}\|
  \le
  \frac{1+o_R(1)}{M-|b_0|-C-\log2}.
\end{equation}
The corresponding finite boundary Schur complement
\[
  S^B_{R,M,L}(\zeta,V)
  :=P_M\mathfrak B_RP_M
    -P_M\mathfrak B_RQ_{M,L}(D^B_{R,M,L})^{-1}Q_{M,L}\mathfrak B_RP_M
\]
satisfies
\begin{equation}\label{eq:boundary-schur-effective}
  S^B_{R,M,L}(\zeta,V)
  =P_M\bigl(\zeta I-A_V\bigr)P_M+\mathcal R_{M,L}(\zeta,V)+O_{L,C,K}(R^{-1}),
\end{equation}
where
\begin{equation}\label{eq:boundary-schur-tail}
  \|\mathcal R_{M,L}(\zeta,V)\|
  \le
  \frac{(\log2)^2+o_R(1)}{M-|b_0|-C-\log2}.
\end{equation}
Moreover, the Schur complement gives a controlled extension from the inner core to the full boundary
pencil.  If \(\|f_M\|=1\) and
\[
  \|S^B_{R,M,L}(\zeta,V)f_M\|\le \rho,
\]
then, with
\[
  h=- (D^B_{R,M,L})^{-1}Q_{M,L}\mathfrak B_R(\zeta,V)P_Mf_M,
  \qquad f_L=f_M+h,
\]
one has
\begin{equation}\label{eq:ordered-extension-shell-size}
  \|h\|
  \le
  \frac{\log2+o_R(1)}{M-|b_0|-C-\log2},
\end{equation}
and, after choosing \(L\) so that
\(\sup_{V\in\calB}\|P_{>L}VP_M\|\le\delta\),
\begin{equation}\label{eq:ordered-full-boundary-residual}
  \|\mathfrak B_R(\zeta,V)f_L\|
  \le
  \rho+C\delta+
  \frac{C}{M-|b_0|-C-\log2}+o_R(1).
\end{equation}
Here \(C\) depends only on the fixed effective window and the common height bound.
\end{lemma}

\begin{proof}
We prove the block estimate in the splitting \(P_M\oplus(P_L-P_M)\oplus P_{>L}\).  On the finite shell
\(Q_{M,L}L^2\), the model operator is
\[
  Q_{M,L}(\zeta I-A_V)Q_{M,L}.
\]
Since \(T_n\ge M\) on this shell and \(\|V\|_\infty\le\log2\), the operator is uniformly
invertible with inverse norm at most
\((M-|b_0|-C-\log2)^{-1}\).  The finite-window expansion
\[
  P_L\mathfrak B_R(\zeta,V)P_L=P_L(\zeta I-A_V)P_L+O_{L,C,K}(R^{-1})
\]
then gives \eqref{eq:boundary-shell-inverse} by a Neumann-series argument for fixed
\((M,L)\) and large \(R\).

The Schur formula is the usual block inversion identity.  The off-diagonal terms have a specific source.  The diagonal pieces \(\zeta I\), \(T_n\), and \(b_0I\) preserve the
splitting \(P_M\oplus Q_{M,L}\).  Thus the coupling between the core and the shell is only
multiplication by \(-V\), whose norm is at most \(\log2\).  This gives
\eqref{eq:boundary-schur-effective}--\eqref{eq:boundary-schur-tail}.

If \(f_M\) is given, the vector \(h\) is chosen so that the \(Q_{M,L}\)-component of
\(P_L\mathfrak B_RP_L(f_M+h)\) vanishes.  The \(P_M\)-component is exactly
\(S^B_{R,M,L}f_M\), hence its norm is at most \(\rho\).  The estimate
\eqref{eq:ordered-extension-shell-size} follows from the inverse bound and the fact that the
core-to-shell coupling again comes only from multiplication by \(V\), up to the fixed-window
\(O_{L,C,K}(R^{-1})\) error.

It remains only to estimate the part of the full boundary residual outside the chosen outer window.
For the Schur-lifted vector one has the decomposition
\begin{equation}\label{eq:exterior-residual-decomposition}
  P_{>L}\mathfrak B_R(\zeta,V)f_L
  =P_{>L}\mathfrak B_R(\zeta,V)P_Mf_M
   +P_{>L}\mathfrak B_R(\zeta,V)Q_{M,L}h .
\end{equation}
Lemma~\ref{lem:outer-tail-fixed-inner} gives
\[
  \|P_{>L}\mathfrak B_R(\zeta,V)P_Mf_M\|\le \delta+o_R(1)
\]
by the high-mode multiplication estimate \(\sup_{V\in\calB}\|P_{>L}VP_M\|\to0\).  For the second
term in \eqref{eq:exterior-residual-decomposition}, the diagonal pieces \(\zeta I\), \(T_n\), and
\(b_0I\) have no map from \(Q_{M,L}\) to \(P_{>L}\).  Thus only multiplication by \(V\) and the
finite-window remainder contribute, giving
\[
  \|P_{>L}\mathfrak B_R(\zeta,V)Q_{M,L}h\|
  \le (\log2+\delta+o_R(1))\|h\| .
\]
Combining these two bounds with \eqref{eq:ordered-extension-shell-size} proves
\eqref{eq:ordered-full-boundary-residual}.  The estimates are uniform under the following ordered limit:
\(M\) gives coercivity on the shell, \(L\) makes \(\sup_{V\in\calB}\|P_{>L}VP_M\|\) small, and the
finite-window remainders then vanish as \(R\to\infty\).
\end{proof}

\begin{lemma}[Spectral sandwich in a separated window]\label{lem:abstract-spectral-sandwich}
Let \(Y\) be a finite multiset contained in an interval \(I\), and assume
\(\dist(Y,\partial I)\ge\gamma\).  Let \(0<\varepsilon<\gamma/4\).  For each connected component
\(J_a\) of
\[
  \bigcup_{y\in Y}[y-\varepsilon,y+\varepsilon],
\]
let \(Y_a=Y\cap J_a\), counted with multiplicity.  Suppose a second multiset \(X\) of characteristic
values satisfies the following cluster comparison: for every \(a\), \(X\cap J_a\) has the same
cardinality as \(Y_a\), and \(X\cap(I\setminus\cup_a J_a)=\emptyset\).  Then \(X\cap I\) and \(Y\)
have the same cardinality and
\[
  d_{\rm ms}(X\cap I,Y)\le \varepsilon .
\]
In applications the equalities of cluster cardinalities are obtained by Rouch\'e comparison on the
boundaries of the intervals \(J_a\), together with the two spectral inclusions supplied by the
Schur-lift/full-residual construction and the exact-state capture estimate.
\end{lemma}

\begin{proof}
The intervals \(J_a\) lie in \(I\) because \(\varepsilon<\gamma/4\).  By hypothesis, every point of
\(X\cap I\) lies in one of these intervals and each interval contains exactly the number of points of
\(X\) prescribed by \(Y\).  Matching points component by component gives the stated distance bound.
\end{proof}

The preceding lemmas compare the finite Schur core with the full spectral problem in both directions.
A small Schur residual gives a small full boundary residual after the shell and tail have been restored,
and the collar correction converts this residual into a Dirichlet quasimode.  Conversely, an exact low
Dirichlet state is localized in the ball first band, then compared with the corrected finite Ritz
space and the finite outer pencil.  The next lemma records the resulting multiset estimate.

\begin{lemma}[Finite-section Grushin multiset comparison]\label{lem:finite-section-grushin-stability}
Fix a compact effective window \([-C,C]\), an integer \(J\ge0\), a number \(K<\infty\), and a
family \(\calB\subset C^\infty(\Sn)\cap\mathfrak C_K\) which is compact in \(C(\Sn)\) and satisfies a common \(C^{m_0}\)-bound \(\mathfrak S_{m_0}(\calB)<\infty\) as in Proposition~\ref{prop:boundary-pair-reduction}.  For \(V\in\calB\), set
\[
  A_V:=T_n+V-b_0,
  \qquad P_M:=\mathbf 1_{[0,M]}(T_n).
\]
Let \(I\Subset(-C,C)\) be \(\gamma\)-regular for the first \(J+1\) eigenvalues of \(A_V\), uniformly in
\(V\in\calB\).  Let
\[
  \mathfrak B_R(\zeta,V):=-\frac R\pi B_R^{\mathrm{raw}}(k_R(\zeta),V)
\]
be the normalized boundary characteristic operator.  For
\[
  \dim P_ML^2(\Sn)\ge J+1,
  \qquad M>|b_0|+C+\log2,
  \qquad L\ge M,
\]
define
\[
  \mathcal E_{\rm Gal}(M)
  :=\sup_{V\in\calB}\max_{0\le j\le J}
  |\eta_j(P_MA_VP_M)-\eta_j(A_V)|,
\]
\[
  \mathcal E_{\rm tail}(M,L)
  :=\sup_{V\in\calB}\|P_{>L}VP_M\|,
  \qquad
  \mathcal E_{\rm shell}(M,C)
  :=\frac{(\log2)^2}{M-|b_0|-C-\log2}.
\]
There is a constant \(C_*\), depending on \(J,C,K,\gamma\) and on \(\mathfrak S_{m_0}(\calB)\), such that the following holds.  Put
\[
  \mathcal E_{M,L,R}:=
  \mathcal E_{\rm Gal}(M)+\mathcal E_{\rm shell}(M,C)+\mathcal E_{\rm tail}(M,L)
  +C_{L,C,K,\calB}R^{-1}+C_{L,C,K,\calB}R^{-1/2}.
\]
If \(C_*\mathcal E_{M,L,R}<\gamma/4\), then \(\operatorname{Char}(\mathfrak B_R(\cdot,V);I)\) has cardinality
\(J+1\), and
\begin{equation}\label{eq:finite-section-error-decomposition}
  d_{\rm ms}
  \left(
    \operatorname{Char}(\mathfrak B_R(\cdot,V);I),
    \boldsymbol\eta_J(A_V)
  \right)
  \le C_*\mathcal E_{M,L,R} .
\end{equation}
This estimate is uniform for \(V\in\calB\).  In particular,
\[
  \lim_{M\to\infty}\limsup_{L\to\infty}\limsup_{R\to\infty}
  \sup_{V\in\calB}
  d_{\rm ms}
  \left(
    \operatorname{Char}(\mathfrak B_R(\cdot,V);I),
    \boldsymbol\eta_J(A_V)
  \right)=0.
\]
Thus the characteristic values in every bounded effective window are exhausted, in the limit, by the Schur-lifted first radial band, and their limiting multiset is the corresponding multiset of \(A_V=T_n+V-b_0\).

\end{lemma}

\begin{proof}
We compare four objects in sequence:
\[
  A_V
  \longleftrightarrow P_MA_VP_M
  \longleftrightarrow S^B_{R,M,L}(\zeta,V)
  \longleftrightarrow P_L\mathfrak B_R(\zeta,V)P_L
  \longleftrightarrow \mathfrak B_R(\zeta,V).
\]
The first comparison is \(\mathcal E_{\rm Gal}(M)\), by the Galerkin convergence of
Lemma~\ref{lem:finite-section-effective}.

For the second comparison, Lemma~\ref{lem:ordered-boundary-schur-residual} gives, on the core
\(P_ML^2\),
\[
  S^B_{R,M,L}(\zeta,V)
  =P_M(\zeta I-A_V)P_M+\mathcal R_{M,L}(V)+O_{L,C,K}(R^{-1}),
\]
with
\[
  \|\mathcal R_{M,L}(V)\|
  \le C_C\mathcal E_{\rm shell}(M,C).
\]
Finite-dimensional Rouch\'e comparison on contours around separated clusters therefore moves the
characteristic values by at most
\(C_{J,C,K}(\mathcal E_{\rm shell}(M,C)+C_{L,C,K}R^{-1})\).

The passage from the finite Schur complement to the finite outer pencil is the Schur lifting in
Lemma~\ref{lem:ordered-boundary-schur-residual}.  A core vector whose Schur residual is small lifts
to a vector in \(P_LL^2\) whose full normalized boundary residual is bounded by
\[
  C_{J,C,K}
  \bigl(\mathcal E_{\rm shell}(M,C)+\mathcal E_{\rm tail}(M,L)+C_{L,C,K}R^{-1}\bigr).
\]
The term \(\mathcal E_{\rm tail}\) is the high-mode multiplication estimate
\(\sup_{V\in\calB}\|P_{>L}VP_M\|\).

The final passage is between boundary residuals and Dirichlet spectral values.  The vector entering
the collar-corrected trial map is the Schur-lifted vector whose full normalized boundary residual was
estimated above.  By Lemma~\ref{lem:collar-corrected-trial-map}, with the choice \(\ell_R=R^{1/2}\),
a full residual error \(\rho\) produces an effective defect
\(C(\rho+R^{-1/2}\rho^2+R^{-1/2})\).  On bounded windows this contributes the term
\(C_{L,C,K}R^{-1/2}\) plus the already listed residual errors.  Conversely, Lemma~\ref{lem:outer-cutoff-boundary}
starts from an exact low Dirichlet state, localizes it in the ball first-band window, and compares it
with the corrected finite Ritz space.  The fixed Ritz matrix differs from
\(P_LA_VP_L\) by \(O_{L,C,K}(R^{-1/2})\) by
Lemma~\ref{lem:finite-boundary-quasimode}.

The preceding paragraphs give the two cluster inclusions needed in Lemma~\ref{lem:abstract-spectral-sandwich}.  To spell out the counting step, decompose the effective cluster in \(I\) into disjoint intervals \(J_a\) whose boundaries have distance at least \(2\varepsilon\) from the effective spectrum, with
\[
  \varepsilon=C_*\mathcal E_{M,L,R}.
\]
For \(C_*\) chosen large and \eqref{eq:comparison-remainder-small-condition} imposed, each boundary contour \(\partial J_a\) remains separated from the finite Schur spectrum and from the full characteristic set.  Rouch\'e comparison between the Schur complement and \(P_M(\zeta I-A_V)P_M\) preserves the finite cardinality in \(J_a\).  The Schur-to-boundary direction gives at least that many characteristic values of \(\mathfrak B_R\) in the \(\varepsilon\)-neighbourhood of \(J_a\), because Schur core vectors lift to full-residual quasimodes.  Equivalently, applying the full-residual lifting estimate on the contour \(\partial J_a\) shows that absence of a characteristic value in the neighbourhood would contradict the Fredholm inverse bound on that contour.  The reverse exact-state capture gives no additional characteristic values in the same neighbourhood: any such value is localized, compared with the corrected Ritz space, and therefore returns to the same finite effective cluster.  Fredholm multiplicity is identified by Proposition~\ref{prop:sobolev-fredholm-framework} and Lemma~\ref{lem:boundary-dirichlet-multiplicity}.  Hence the hypotheses of Lemma~\ref{lem:abstract-spectral-sandwich} hold, and that lemma gives \eqref{eq:finite-section-error-decomposition}.  The limiting statement follows because
\(\mathcal E_{\rm Gal}(M)\to0\) as \(M\to\infty\),
\(\mathcal E_{\rm shell}(M,C)\to0\) as \(M\to\infty\),
\(\mathcal E_{\rm tail}(M,L)\to0\) as \(L\to\infty\) with \(M\) fixed, and the two finite-window
errors vanish as \(R\to\infty\) after \((M,L)\) have been fixed.
\end{proof}

\begin{corollary}[Exhaustion by the finite effective core]\label{cor:no-high-characteristics}
Let \(\calB\subset\mathfrak C_K\) be a compact family of smooth graph deficits.  For each compact
effective window \([-C,C]\) and each \(\varepsilon>0\), one can choose \(M\) so large that, uniformly
for \(V\in\calB\) and all large \(R\), every characteristic value in the window is generated, up to an
\(\varepsilon\)-error in the effective parameter, by the finite matrix
\(P_MA_VP_M\).  Equivalently, the characteristic values in the window are exhausted, in the effective scale, by the finite effective core.
\end{corollary}

\begin{proof}
Apply the preceding multiset comparison and then let the inner cutoff tend to infinity.
Choose \(J\) larger than the number of effective eigenvalues of \(A_V\) which can lie in a slightly
enlarged window, uniformly for \(V\in\calB\); compactness of \(\calB\) and compact resolvent of
\(T_n+V\) give such a number.  Lemma~\ref{lem:finite-section-grushin-stability} reduces the full
boundary characteristic values in \([-C,C]\) to the eigenvalues of \(P_MA_VP_M\), up to the ordered
errors.  The Galerkin convergence of Lemma~\ref{lem:finite-section-effective} then lets us choose
\(M\) so that the remaining effective tail is below \(\varepsilon\).  The localization estimate
Lemma~\ref{lem:graph-spectral-cutoff} is what excludes higher radial roots before the boundary
Schur step begins.
\end{proof}

\begin{lemma}[Uniform finite-window continuity in the graph deficit]\label{lem:finite-window-graph-continuity}
Fix \(L,C,K<\infty\).  If \(V,W\in\mathfrak C_K\), then, uniformly for \(|\zeta|\le C\),
\begin{equation}\label{eq:finite-window-graph-continuity}
  \left\|P_L\{\mathfrak B_R(\zeta,V)-\mathfrak B_R(\zeta,W)\}P_L\right\|
  \le C_{L,C,K}\|V-W\|_\infty+o_R(1)\|V-W\|_\infty .
\end{equation}
The same estimate holds for the finite collar-corrected trial forms.  Consequently heat-regularized smooth deficits may be used with the limit order \(t>0\) fixed, then \(R\to\infty\), then \(t\downarrow0\).
\end{lemma}

\begin{proof}
For \(f\in P_LL^2\), use the fundamental theorem of calculus in the radial variable:
\[
  \mathcal S(k_R(\zeta))f(R-V,\theta)-\mathcal S(k_R(\zeta))f(R-W,\theta)
  =(W-V)\int_0^1\partial_r\mathcal S(k_R(\zeta))f(R-V+\tau(V-W),\theta)\,d\tau .
\]
On a fixed angular window the radial derivative in the outer collar is \(O_{L,C}(R^{-1})\) in sine-amplitude normalization, with \(O_{L,C}(R^{-2})\) remainder from Lemma~\ref{lem:low-energy-phase}.  After multiplication by \(-R/\pi\), this gives \eqref{eq:finite-window-graph-continuity}.  The collar-corrected form estimate follows from the same fixed-collar calculation.
\end{proof}

 \subsection{Completion of the effective band theorem}\label{sec:effective-completion}

\begin{proof}[Proof of \cref{prop:boundary-pair-reduction}]
Set
\[
  E_{\rm phase}(L,R)=C_{L,C,K,\calB}R^{-1},
  \qquad
  E_{\rm collar}(L,R)=C_{L,C,K,\calB}R^{-1/2},
\]
where the constants depend on the \(C^{m_0}\)-bound \(\mathfrak S_{m_0}(\calB)\).  Lemma~\ref{lem:finite-section-grushin-stability} gives the separated-window comparison.  More precisely, once the total error is smaller than a fixed fraction of the distance from \(\partial I\) to the effective spectrum, Rouch\'e comparison on the two boundary contours of \(I\) gives exact characteristic count inside \(I\), and the same estimate gives \eqref{eq:boundary-pair-comparison-error}.  Multiplicity is preserved by the Fredholm identification in Proposition~\ref{prop:sobolev-fredholm-framework} and by the finite-dimensional Rouch\'e comparison on cluster contours.

The limiting statement follows from the four limits
\[
  E_{\rm Gal}(M)\to0,
  \qquad E_{\rm shell}(M,C)\to0
  \quad(M\to\infty),
\]
\[
  E_{\rm tail}(M,L)\to0
  \quad(L\to\infty,\ M\text{ fixed}),
  \qquad
  E_{\rm phase}(L,R)+E_{\rm collar}(L,R)\to0
  \quad(R\to\infty,\ M,L\text{ fixed}).
\]
The first is Lemma~\ref{lem:finite-section-effective}; the second is the explicit shell bound; the
third is Lemma~\ref{lem:outer-tail-fixed-inner}; the last two are the finite-window expansion and
collar correction.
\end{proof}

\begin{proof}[Proof of \cref{thm:effective-band}]
First fix \(K<\infty\) and prove the assertion for a family
\(\calB_s\subset C^\infty(\Sn)\cap\mathfrak C_K\) which is compact in \(C(\Sn)\) and satisfies the \(C^{m_0}\)-bound required in Proposition~\ref{prop:boundary-pair-reduction}.  Choose \(C\) so large that, uniformly for
\(V\in\calB_s\), the first \(N+1\) effective parameters \(\eta_j(T_n+V)-b_0\), \(0\le j\le N\), lie
in \((-C+1,C-1)\).  Proposition~\ref{prop:boundary-pair-reduction} gives convergence on every
regular effective cluster.  If one of the first \(N+1\) levels is multiple, we enclose the whole
coalescing cluster in a single regular interval and compare multisets in that interval.  Since the
effective eigenvalues depend continuously on \(V\) and \(\calB_s\) is compact, finitely many such
cluster windows cover the family.  Reading the cluster convergence in increasing order gives
\[
  \zeta_{j,R}(V)=\eta_j(T_n+V)-b_0+o_R(1),
  \qquad 0\le j\le N,
\]
uniformly for \(V\in\calB_s\).  This proves \eqref{eq:eig-effective} for compact smooth families in
\(\mathfrak C_K\).

Now let \(\calB\subset\mathfrak C_K\) be the compact Lipschitz family in the theorem.  Given
\(\delta>0\), choose \(t>0\) so small that, with \(W=P_tV\),
\[
  \sup_{V\in\calB}\|W-V\|_\infty<\delta,
  \qquad
  |\eta_j(T_n+W)-\eta_j(T_n+V)|<\delta,\quad 0\le j\le N,
\]
uniformly for \(V\in\calB\).  This follows from \cref{lem:uniform-heat-regularization} and the min-max principle.  For this fixed \(t\), the family \(P_t\calB\) has the required \(C^{m_0}\)-bounds.  The smooth-family theorem is then applied with \(t\) fixed and \(R\to\infty\), and only afterwards is \(t\downarrow0\) through the height bracket below.

The height approximation gives the inclusions
\begin{equation}\label{eq:lipschitz-height-bracket}
  \Omega_{R-\delta,W}\subset\Omega_{R,V}\subset\Omega_{R+\delta,W},
\end{equation}
for all large \(R\).  Indeed, \(|V-W|\le\delta\) is equivalent to
\[
  R-W-\delta\le R-V\le R-W+\delta.
\]
Dirichlet monotonicity therefore gives
\[
  \lambda_{j+1}(\Omega_{R+\delta,W})
  \le \lambda_{j+1}(\Omega_{R,V})
  \le \lambda_{j+1}(\Omega_{R-\delta,W}).
\]
Applying the smooth expansion at the radius \(R+s\), with fixed \(s=\pm\delta\), and then renormalizing around \(R\), uses
\[
  \frac{R^3}{2\pi^2}
  \left(\frac{\pi^2}{(R+s)^2}-\frac{\pi^2}{R^2}\right)=-s+O_s(R^{-1}).
\]
Thus, uniformly in \(W\),
\[
  \frac{R^3}{2\pi^2}
  \left(\lambda_{j+1}(\Omega_{R+s,W})-\alpha^2-\frac{\pi^2}{R^2}\right)
  =\eta_j(T_n+W)-b_0-s+o_R(1).
\]
Thus the normalized parameter for \(\Omega_{R,V}\) lies between
\[
  \eta_j(T_n+W)-b_0-\delta+o_R(1)
  \quad\text{and}\quad
  \eta_j(T_n+W)-b_0+\delta+o_R(1).
\]
Since the effective eigenvalues are \(1\)-Lipschitz with respect to the \(L^\infty\)-norm of the potential, this differs from \(\eta_j(T_n+V)-b_0\) by at most \(2\delta+o_R(1)\).  Letting first \(R\to\infty\) with \(\delta\) fixed, and then \(\delta\downarrow0\), proves the theorem for compact Lipschitz families.

\end{proof}

\subsection{Sequential graph limits and the gap formula}

\begin{corollary}[Sequential graph limits]\label{cor:sequential-effective-limit}
Let \(R_k\to\infty\), and let \(V_k\) be graph deficits belonging to one compact Lipschitz family
\(\calB\subset\mathfrak C_K\).  Suppose \(V_k\to V\) uniformly.  Then, for every fixed \(N\) and
\(0\le j\le N\),
\begin{equation}\label{eq:sequential-effective-eigs}
  \lambda_{j+1}(\Omega_{R_k,V_k})
  =\alpha^2+\frac{\pi^2}{R_k^2}
  +\frac{2\pi^2}{R_k^3}\bigl(\eta_j(T_n+V)-b_0\bigr)
  +o(R_k^{-3}).
\end{equation}
Consequently
\begin{equation}\label{eq:sequential-effective-gap}
  \Gap(\Omega_{R_k,V_k})
  =\frac{2\pi^2}{R_k^3}\bigl[\eta_1(T_n+V)-\eta_0(T_n+V)\bigr]+o(R_k^{-3}).
\end{equation}
The same conclusion holds in the fixed-height variant \(0\le V_k\le A\), with constants depending
also on \(A\).
\end{corollary}

\begin{proof}
The compact set \(\{V\}\cup\{V_k:k\ge k_0\}\) is contained in a compact Lipschitz family.  Apply
Theorem~\ref{thm:effective-band} uniformly to this family.  It gives the expansion with
\(\eta_j(T_n+V_k)-b_0\) in place of \(\eta_j(T_n+V)-b_0\).  Since
\[
  |\eta_j(T_n+V_k)-\eta_j(T_n+V)|\le \|V_k-V\|_\infty\to0,
\]
the replacement of \(V_k\) by \(V\) changes only the \(o(R_k^{-3})\) term.  Taking the difference of
the cases \(j=1\) and \(j=0\) gives \eqref{eq:sequential-effective-gap}.  The height-\(A\) variant
is the variant recorded in Lemma~\ref{lem:admissible-lipschitz-hull}.
\end{proof}

\begin{remark}[Source of the \(R^{-3}\) operator]\label{rem:fixed-cylinder-point}
The effective operator is the sum of three contributions:
\[
  \text{threshold scattering near }r=0,
  \qquad
  \text{Dirichlet displacement at }r=R,
  \qquad
  \text{higher radial roots.}
\]
The threshold contribution is \(T_n-b_0\), the boundary displacement is multiplication by \(V\), and the remaining radial branches stay separated at scale \(R^{-2}\).  Hence the order \(R^{-3}\) operator in \eqref{eq:eig-effective} is \(T_n+V-b_0\).
\end{remark}

Taking the difference between the cases \(j=1\) and \(j=0\) in \eqref{eq:eig-effective} gives
\begin{equation}\label{eq:gap-effective}
  \Gap(\Om_{R,V})=
  \frac{2\pi^2}{R^3}\bigl(\eta_1(T_n+V)-\eta_0(T_n+V)\bigr)+o(R^{-3}).
\end{equation}
Together with \(D=2R+O(1)\), this is the source of the factor \(16\pi^2\) in
\cref{thm:var-formula}.

\section{Passage from graph asymptotics to horoconvex domains}\label{sec:transfer}

The analytic theorem in the preceding section is stated for radial graph domains
\(\Omega_{R,V}\).  Horoconvex domains enter through the compactness and realization theory of
\Cref{sec:horoconvex-compactness}.  We write \(\lambda_j(K)\) for the Dirichlet eigenvalues of the
interior of the compact body \(K\); changing the boundary does not change the Dirichlet spectrum.  To
make the logical passage explicit, we record the passage from graph asymptotics to horoconvex bodies
used in the proofs of the main theorems.

\begin{proposition}[Passage from graph asymptotics to horoconvex bodies]\label{prop:horoconvex-transfer}
Let \(K_k\subset\Hn\) be Chebyshev-centered bounded horoconvex bodies with circumradii
\(R_k\to\infty\).
We write \(\lambda_j(K_k)\) for the Dirichlet eigenvalue of \(\operatorname{int}K_k\).
Let
\[
  V_k(\theta)=R_k-\rho_{K_k}(\theta)
\]
be their radial deficits with respect to the chosen centers.  Suppose that, after passing to a
subsequence, \(V_k\to V\) uniformly.  Then \(V\in\calA_n\), \(\diam K_k=2R_k+O(1)\), and for every
fixed \(N\) and \(0\le j\le N\),
\begin{equation}\label{eq:horoconvex-transfer-eigs}
  \lambda_{j+1}(K_k)=
  \alpha^2+\frac{\pi^2}{R_k^2}
  +\frac{2\pi^2}{R_k^3}\bigl(\eta_j(T_n+V)-b_0\bigr)
  +o(R_k^{-3}).
\end{equation}
Consequently,
\begin{equation}\label{eq:horoconvex-transfer-gap}
  \Gap(K_k)=
  \frac{2\pi^2}{R_k^3}
  \bigl[\eta_1(T_n+V)-\eta_0(T_n+V)\bigr]
  +o(R_k^{-3}).
\end{equation}

Conversely, if \(V\in\calA_n\) and \(q=q_V\) is its canonical support profile, the horoconvex
support realizations
\[
  K_R(q)=\bigcap_{\xi\in\Sn}\{x:\beta_\xi(x)\le e^Rq(\xi)\}
\]
satisfy \eqref{eq:horoconvex-transfer-eigs} along \(R\to\infty\) with limiting deficit \(V\).  If
\(R_D\) is chosen by Lemma~\ref{lem:diameter-calibration}, then the same expansion holds for the
exact-diameter family \(\diam K_{R_D}(q)=D\), with \(R_k\) replaced by \(R_D\).  In particular,
\begin{equation}\label{eq:exact-diameter-transfer-gap}
  D^3\Gap(K_{R_D}(q))\longrightarrow 16\pi^2\mathfrak g(V).
\end{equation}
\end{proposition}

\begin{proof}
The first assertion that \(V\in\calA_n\) and \(\diam K_k=2R_k+O(1)\) is
Proposition~\ref{prop:radial-compactness}.  The star-shaped graph representation with respect to the
Chebyshev center was established in \Cref{sec:horoconvex-compactness}; hence, up to a boundary set
of measure zero,
\[
  K_k=\Omega_{R_k,V_k}=\{(r,\theta):0\le r\le R_k-V_k(\theta)\}.
\]
The deficits \(V_k\) lie, after discarding finitely many terms, in one fixed-height compact
Lipschitz graph family by Lemma~\ref{lem:finite-R-lipschitz}.  Corollary~\ref{cor:sequential-effective-limit}
therefore applies and gives \eqref{eq:horoconvex-transfer-eigs}.  Subtracting the cases \(j=1\) and
\(j=0\) gives \eqref{eq:horoconvex-transfer-gap}.

For the converse, Lemma~\ref{lem:realization} gives the radial deficit
\(W_R=R-\rho_{K_R(q)}\) and
\[
  \|W_R-V\|_{C(\Sn)}=O(e^{-2R})
\]
with respect to the chosen support center.  These support-center coordinates already give the graph
domain \(K_R(q)=\Omega_{R,W_R}\), and the spectrum is intrinsic to the set.  If one also records the true Chebyshev circumradius, Lemma~\ref{lem:realization} gives
\(R_{\rm ch}=R+O(e^{-2R})\), and this replacement is negligible at the effective scale:
\[
  R_{\rm ch}^{-2}=R^{-2}+o(R^{-4}),\qquad
  R_{\rm ch}^{-3}=R^{-3}+o(R^{-5}).
\]
Moreover
\[
  |\eta_j(T_n+W_R)-\eta_j(T_n+V)|\le \|W_R-V\|_\infty=O(e^{-2R}).
\]
The sequential graph limit therefore yields \eqref{eq:horoconvex-transfer-eigs} with limiting
deficit \(V\).  The exact-diameter statement follows in the same way from
Lemma~\ref{lem:diameter-calibration}, because
\(R_D=(D-\Delta(V))/2+O(e^{-D})\) and the realized deficits are still \(V+O(e^{-D})\).
\end{proof}

\begin{proof}[Proof of \cref{thm:effective-limit}]
Apply Proposition~\ref{prop:radial-compactness} to a divergent sequence of Chebyshev-centered
horoconvex domains.  After passing to a subsequence, the radial deficits converge uniformly to some
\(V\in\calA_n\), and \(\diam\Omega_k=2R_k+O(1)\).  Proposition~\ref{prop:horoconvex-transfer}, with
\(N=1\), gives \eqref{eq:eff-intro} for \(j=1,2\).  The converse realization statement is the second
part of Proposition~\ref{prop:horoconvex-transfer}.
\end{proof}

\section{Positivity and the effective gap}

The lower bound in \cref{thm:main-lower} follows from compactness of the effective problem and
the strict positivity of the effective gap.

Let \(P_\rho\) be the spherical Poisson semigroup,
\[
  P_\rho Y=\rho^\ell Y,
  \qquad Y\in\calH_\ell.
\]
The digamma identity
\[
  \psi(\ell+\alpha)-\psi(\alpha)
  =\int_0^1(1-\rho^\ell)\frac{\rho^{\alpha-1}}{1-\rho}\,d\rho
\]
gives the operator formula
\begin{equation}\label{eq:Tn-poisson}
  T_n=\int_0^1(I-P_\rho)\frac{\rho^{\alpha-1}}{1-\rho}\,d\rho.
\end{equation}

\begin{lemma}[Basic spectral properties of \(T_n\)]\label{lem:Tn-basic-spectral}
The operator \(T_n\) is self-adjoint, nonnegative, and has compact resolvent.  Its kernel is the
constant functions.  For every real \(V\in L^\infty(\Sn)\), the operator \(T_n+V\) is self-adjoint with
compact resolvent, and its eigenvalues satisfy the min-max Lipschitz estimate
\begin{equation}\label{eq:Tn-eigenvalue-lipschitz}
  |\eta_j(T_n+V)-\eta_j(T_n+W)|\le \|V-W\|_\infty,
  \qquad j=0,1,2,\ldots .
\end{equation}
\end{lemma}

\begin{proof}
The multipliers \(\psi(\ell+\alpha)-\psi(\alpha)\) are nonnegative, vanish only for \(\ell=0\), and
satisfy \(\psi(\ell+\alpha)=\log\ell+O(\ell^{-1})\).  Since each spherical harmonic space is finite
dimensional and the multipliers tend to infinity, \(T_n\) has compact resolvent.  A bounded real
potential is a bounded self-adjoint perturbation, so \(T_n+V\) remains self-adjoint with compact
resolvent.  The estimate \eqref{eq:Tn-eigenvalue-lipschitz} follows directly from the min-max principle.
\end{proof}

\begin{lemma}[Positivity improving]\label{lem:positive-improving}
For every real \(V\in L^\infty(\Sn)\) and every \(t>0\), the semigroup
\(e^{-t(T_n+V)}\) is positivity improving on \(L^2(\Sn)\).  Hence the lowest eigenvalue of
\(T_n+V\) is simple.
\end{lemma}

\begin{proof}
For \(0<\rho<1\), the Poisson kernel on \(\Sn\) is strictly positive.  Write \(P_\rho=e^{-sA}\) with \(\rho=e^{-s}\), where \(A\) is the positive Poisson generator, \(AY=\ell Y\) on \(\calH_\ell\).  The multiplier of \(T_n\) is the Bernstein function
\[
  \Phi(\ell)=\psi(\ell+\alpha)-\psi(\alpha)
  =\int_0^\infty (1-e^{-s\ell})\,\frac{e^{-\alpha s}}{1-e^{-s}}\,ds .
\]
By the subordination theorem for Bernstein functions,
\[
  e^{-tT_n}=e^{-t\Phi(A)}=\int_0^\infty e^{-sA}\,d\mu_t(s),
\]
where \(\mu_t\) is a probability measure which has nonzero mass on \((0,\infty)\).  Since each \(e^{-sA}=P_{e^{-s}}\) has a strictly positive kernel for \(s>0\), the mixture has a strictly positive kernel.

It remains to check that adding a bounded potential does not destroy strict positivity.  Let
\(M=\|V\|_\infty\).  In the Trotter product
\[
  e^{-t(T_n+V)}=\lim_{m\to\infty}\left(e^{-tT_n/m}e^{-tV/m}\right)^m,
\]
we have \(e^{-tV/m}\ge e^{-tM/m}I\) as multiplication operators.  Since \(e^{-tT_n/m}\) is positivity
preserving, induction gives, for every \(f\ge0\),
\[
  \left(e^{-tT_n/m}e^{-tV/m}\right)^m f
  \ge e^{-tM}\left(e^{-tT_n/m}\right)^m f
  = e^{-tM}e^{-tT_n}f .
\]
Taking the strong limit preserves the order in the positive cone.  If \(f\ge0\) and
\(f\not\equiv0\), then \(e^{-tT_n}f>0\) a.e.; hence \(e^{-t(T_n+V)}f>0\) a.e.  The simplicity of the
ground state follows from the Perron-Frobenius-Jentzsch theorem for compact-resolvent self-adjoint
semigroups.
\end{proof}

\begin{lemma}[Uniform effective gap]\label{lem:uniform-effective-gap}
There exists \(\delta_n>0\) such that
\begin{equation}\label{eq:delta}
  \mathfrak g(V):=\eta_1(T_n+V)-\eta_0(T_n+V)\ge\delta_n
  \qquad\text{for all }V\in\calA_n.
\end{equation}
\end{lemma}

\begin{proof}
The class \(\calA_n\) is compact in \(C(\Sn)\).  By Lemma~\ref{lem:Tn-basic-spectral}, the first two
eigenvalues are continuous, in fact Lipschitz, under uniform perturbations of \(V\).  If
\(\inf_{\calA_n}\mathfrak g=0\), compactness gives a minimizer
\(V_*\in\calA_n\) with \(\eta_1(T_n+V_*)=\eta_0(T_n+V_*)\), contradicting the simplicity of the
ground state from Lemma~\ref{lem:positive-improving}.
\end{proof}

\begin{proof}[Proof of \cref{thm:main-lower}]
Assume, to the contrary, that there are horoconvex domains \(\Om_k\) with
\(D_k=\diam\Om_k\to\infty\) and \(D_k^3\Gap(\Om_k)\to0\).  Center \(\Om_k\) at a Chebyshev center and
let \(R_k\) be its circumradius.  By Proposition~\ref{prop:radial-compactness}, after passing to a subsequence the radial deficits
\(V_k=R_k-\rho_{\Omega_k}\) converge uniformly to some \(V\in\calA_n\), and \(D_k=2R_k+O(1)\).
Proposition~\ref{prop:horoconvex-transfer} gives
\[
  \Gap(\Om_k)=\frac{2\pi^2}{R_k^3}\mathfrak g(V)+o(R_k^{-3}).
\]
By Lemma~\ref{lem:uniform-effective-gap}, \(\mathfrak g(V)\ge\delta_n>0\), and therefore
\[
  D_k^3\Gap(\Om_k)\to16\pi^2\mathfrak g(V)\ge16\pi^2\delta_n>0,
\]
a contradiction.
\end{proof}

\section{The asymptotic variational formula}

\begin{proof}[Proof of \cref{thm:var-formula}]
The lower bound follows from the compactness argument in the proof of \cref{thm:main-lower}.  For
any sequence \(\Om_k\) with \(D_k\to\infty\), a subsequence has a limit \(V\in\calA_n\), and
\[
  \liminf_{k\to\infty}D_k^3\Gap(\Om_k)
  \ge16\pi^2\mathfrak g(V)
  \ge16\pi^2\inf_{\calA_n}\mathfrak g.
\]

For the upper bound, fix \(V\in\calA_n\).  The exact-diameter realization in
Proposition~\ref{prop:horoconvex-transfer} gives horoconvex bodies \(K_D(V)\) with
\(\diam K_D(V)=D\) and
\[
  D^3\Gap(K_D(V))\longrightarrow16\pi^2\mathfrak g(V).
\]
Hence
\[
  \limsup_{D\to\infty}D^3
  \inf_{\diam\Om=D}\Gap(\Om)
  \le16\pi^2\mathfrak g(V).
\]

Taking the infimum over \(V\in\calA_n\) proves \eqref{eq:main-constant}.  Positivity of the
infimum is Lemma~\ref{lem:uniform-effective-gap}.
\end{proof}

\section{Non-optimality of geodesic balls}

For \(V\equiv0\), the operator \(T_n\) has lowest eigenvalue \(0\), with constants as ground states,
and first excited eigenvalue
\[
  \psi(\alpha+1)-\psi(\alpha)=\frac1\alpha=\frac{2}{n-1},
\]
with eigenspace \(\calH_1\).  Hence
\begin{equation}\label{eq:g0}
  \mathfrak g(0)=\frac{2}{n-1}.
\end{equation}

Let \(e\in\mathbb R^n\) be a unit vector and set
\[
  W(\theta)=1-(e\cdot\theta)^2.
\]
By Lemma~\ref{lem:axial-admissible-curve}, the curve
\[
  q_\eps(\xi)=\exp[-\eps W(-\xi)]
\]
is normalized for all sufficiently small \(\eps>0\), and its admissible envelope satisfies
\[
  V_{q_\eps}(\theta)=\eps W(\theta)+O(\eps^2)
\]
uniformly.

Now apply first-order perturbation theory to \(T_n+\eps W\).  The first excited level of \(T_n\) has
multiplicity \(n\), so the right derivative of \(\eta_1(T_n+\eps W)\) is given by Kato's degenerate
perturbation formula: one compresses multiplication by \(W\) to the eigenspace \(\calH_1\) and takes
the smallest eigenvalue of that compression.  Let \(\sigma\) be normalized surface measure on \(\Sn\),
and write \(X=e\cdot\theta\).  Since \(\mathbb E X^2=1/n\),
\[
  \eta_0'(0)=\int_{\Sn}W\,d\sigma=1-\frac1n=\frac{n-1}{n}.
\]
On the first excited space \(\calH_1=\operatorname{span}\{\theta_1,\dots,\theta_n\}\), the first-order splitting is the matrix
\[
  M_{ij}=
  \frac{\int_{\Sn}(1-X^2)\theta_i\theta_j\,d\sigma}
       {\int_{\Sn}\theta_i^2\,d\sigma} .
\]
Choosing coordinates with \(e=e_1\), symmetry makes this matrix diagonal.  Using
\[
  \mathbb E X^4=\frac{3}{n(n+2)},\qquad
  \mathbb E[X^2\theta_j^2]=\frac{1}{n(n+2)}\quad(j\ne1),
\]
its axial eigenvalue is
\[
  1-\frac{3}{n+2}=\frac{n-1}{n+2},
\]
whereas every transverse eigenvalue is
\[
  1-\frac{1}{n+2}=\frac{n+1}{n+2}.
\]
Thus the smallest perturbation is obtained by the axial harmonic \(Y(\theta)=X\).  Therefore
\[
  \mathfrak g(\eps W)=\frac{2}{n-1}
  -\frac{2(n-1)}{n(n+2)}\eps+O(\eps^2).
\]
Finally \(V_{q_\eps}=\eps W+O(\eps^2)\) uniformly.  By the Lipschitz estimate
\eqref{eq:Tn-eigenvalue-lipschitz}, replacing \(\eps W\) by this admissible curve changes both
relevant effective eigenvalues, and hence \(\mathfrak g\), by \(O(\eps^2)\).  
This proves
\cref{thm:ball-not-optimal}.

\begin{example}[A reduced axial model in dimension three]\label{ex:n3-axial-numerics}
When \(n=3\), one has \(\alpha=1\) and
\[
  T_3Y_{\ell m}=H_\ell Y_{\ell m},
  \qquad H_\ell=1+\frac12+\cdots+\frac1\ell,
\]
with \(H_0=0\).  The axial tangent direction above is
\[
  W(\theta)=1-z^2,
  \qquad z=e\cdot\theta .
\]
Multiplication by \(W\) preserves the azimuthal quantum number \(m\) and couples only degrees
\(\ell\) and \(\ell\pm2\).  Truncating the axisymmetric blocks at \(\ell\le16\) gives the following
values for the reduced model \(T_3+\varepsilon W\):
\[
\begin{array}{c|ccc}
\varepsilon & \eta_0 & \eta_1 & \eta_1-\eta_0 \\
\hline
0    & 0       & 1       & 1       \\
0.05 & 0.03318 & 1.01980 & 0.98661 \\
0.10 & 0.06607 & 1.03919 & 0.97312 \\
0.20 & 0.13090 & 1.07678 & 0.94588 \\
0.50 & 0.31760 & 1.18063 & 0.86303
\end{array}
\]
The derivative of the last column at \(0\) is \(-4/15\), in agreement with
\cref{thm:ball-not-optimal} for \(n=3\).  The proof of \cref{thm:ball-not-optimal} is analytic; this finite calculation illustrates
how the reduced gap moves away from the ball profile in the simplest nonradial direction.
\end{example}

\section{The reduced sharp-constant problem}

The preceding theorem gives the sharp \(D^{-3}\) scale and identifies the remaining large-diameter constant through the
compact nonlocal variational problem
\[
  \inf_{V\in\calA_n}\mathfrak g(V),
  \qquad
  \mathfrak g(V)=\eta_1(T_n+V)-\eta_0(T_n+V).
\]
Every large horoconvex domain contributes, after normalization, one element of \(\calA_n\), and every
element of \(\calA_n\) is realized by horoconvex domains with calibrated diameter.  Thus the
optimizer geometry has been separated from the scale theorem.  The first-variation calculation gives
\[
  \mathfrak g(0)=\frac{2}{n-1},
  \qquad
  \left.\frac{d}{d\eps}\mathfrak g(V_{q_\eps})\right|_{\eps=0}
  =-\frac{2(n-1)}{n(n+2)}<0,
\]
so the admissible axial curve gives a strict first-order decrease from the round profile.

The next observation records what an unconstrained interior minimizer would have to satisfy.  The
scale theorem does not use it; it identifies the remaining minimization as a constrained
support-envelope problem.

\begin{proposition}[Necessary condition for unconstrained criticality]\label{prop:unconstrained-criticality}
Let \(\mathcal U\subset L^\infty(\Sn)\) be open, let \(V_*\in\mathcal U\), and suppose the first
excited eigenvalue of \(T_n+V_*\) is simple.  If \(V_*\) is a local minimizer of \(\mathfrak g\) in
\(\mathcal U\), then, for normalized real eigenfunctions \(\phi_0\) and \(\phi_1\) associated with
\(\eta_0(T_n+V_*)\) and \(\eta_1(T_n+V_*)\), one has
\begin{equation}\label{eq:unconstrained-critical-density}
  |\phi_1(\theta)|^2=|\phi_0(\theta)|^2
  \qquad\text{for a.e. }\theta\in\Sn .
\end{equation}
Consequently, any local minimizer for which the density difference
\(|\phi_1|^2-|\phi_0|^2\) is not identically zero must lie on the boundary of the admissible
constraint set rather than in an unconstrained interior regime.
\end{proposition}

\begin{proof}
For \(W\in L^\infty(\Sn)\) and \(|t|\) small, \(V_*+tW\in\mathcal U\).  Since the first two levels are
simple at \(t=0\), analytic perturbation theory gives
\[
  \left.\frac d{dt}\eta_i(T_n+V_*+tW)\right|_{t=0}
  =\int_{\Sn}W(\theta)|\phi_i(\theta)|^2\,d\sigma(\theta),
  \qquad i=0,1.
\]
The first derivative of \(\mathfrak g(V_*+tW)\) at \(t=0\) must vanish for every \(W\).  Hence
\[
  \int_{\Sn}W\bigl(|\phi_1|^2-|\phi_0|^2\bigr)d\sigma=0
  \qquad\text{for every }W\in L^\infty(\Sn),
\]
which is exactly \eqref{eq:unconstrained-critical-density}.
\end{proof}

\begin{remark}[Finite-contact Euler inequality]\label{rem:finite-contact-euler}
A finite-contact envelope gives a concrete form of the constrained variation cone.  Suppose that, near a candidate profile, only finitely many supports are active, so that
\[
  V(\theta)=\max_{1\le i\le N}
  \Phi_i(\theta),
  \qquad
  \Phi_i(\theta)=\log\frac{1-\theta\cdot\xi_i}{2q_i}.
\]
If the support heights are varied by \(q_i(t)=q_i e^{-t a_i}\), then the right derivative of the
envelope is
\[
  W_a(\theta)=\max_{i\in I(\theta)}a_i,
  \qquad
  I(\theta)=\{i:\Phi_i(\theta)=V(\theta)\}.
\]
Thus, at a simple first excited level, the one-sided first variation along this finite-contact cone is
\[
  \left.\frac d{dt}\mathfrak g(V_t)\right|_{t=0+}
  =\int_{\Sn}W_a(\theta)
  \bigl(|\phi_1(\theta)|^2-|\phi_0(\theta)|^2\bigr)\,d\sigma(\theta).
\]
A constrained minimizer of this type would therefore satisfy nonnegativity of this quantity for all
height variations \(a=(a_1,\ldots,a_N)\) preserving the normalization constraints.  This finite-contact
condition is the support-envelope analogue of the usual Euler inequality for an obstacle problem.

The condition \eqref{eq:unconstrained-critical-density} is much stronger than the usual Euler--Lagrange
equation for a constrained potential problem.  In the present geometry the admissible variations are
a cone determined by the horospherical support envelope.
The negative axial variation and \cref{prop:unconstrained-criticality} identify the optimizer problem as a boundary problem for the support-envelope constraint.  An explicit sharp optimizer, if one can be described, should be governed by the active-contact structure of \(\calA_n\).
\end{remark}

\section*{Future directions}

The variational formula in \cref{thm:var-formula} leaves a compact nonlocal support-envelope
optimization problem on \(\calA_n\).  The axial decrease in \cref{thm:ball-not-optimal} and the
finite-contact Euler inequality in \cref{rem:finite-contact-euler} indicate that the remaining
optimizer problem is governed by the active-contact geometry of the horospherical support envelope.

Several tasks are separate from the first-band asymptotic theorem proved here: selecting the first
excited sector in the reduced problem, reducing general admissible profiles to lower-dimensional or
axially saturated support envelopes, and promoting first-order stability conditions to global
minimality in \(\calA_n\).  These questions form a constrained nonlocal shape-optimization problem on
the sphere and will be studied separately.

\section*{Acknowledgements}

The author is grateful to the researchers whose work on fundamental gap estimates, horoconvexity,
and hyperbolic spectral geometry provided the context for this paper.

\section*{Declaration of generative AI and AI-assisted technologies in the manuscript preparation process}

This manuscript was prepared with extensive assistance from OpenAI GPT-5.5 Pro, used through
ChatGPT.  The tool assisted in drafting and rewriting substantial portions of the prose,
reorganizing the presentation of lemmas and proofs, formulating explanatory passages, identifying
possible referee concerns, and planning revisions.  The author reviewed, edited, and checked the
final manuscript, including the mathematical statements, proof arguments, constants, references,
claims, and attributions, and takes full responsibility for its content.  No AI system is credited as
an author or as a source of mathematical authority.

\clearpage
\appendix
\section{Uniform threshold phase expansion}\label{app:phase-expansion}

This appendix proves the one-dimensional estimate used in Lemma~\ref{lem:low-energy-phase}.  The
constant \(b(\beta)\) found at zero energy is the slope of the positive-energy
phase at threshold.  Thus the constants \(b_\ell\) used in the main text are the threshold phase
slopes determined by the Friedrichs branch at the regular singular endpoint.

\begin{proposition}[Threshold phase expansion including the Friedrichs endpoint]\label{prop:appendix-phase}
Let \(I\subset[1/2,\infty)\) be compact.  For \(\beta\in I\), let \(\sigma_\beta(r,k)\) be the
Friedrichs regular positive-energy branch of
\[
  -u''+\frac{\beta(\beta-1)}{\sinh^2 r}u=k^2u
\]
whose exterior sine amplitude is one.  Then, for \(0<k<k_0(I)\) and \(1\le r\le c_0/k\),
\begin{equation}\label{eq:appendix-phase-function}
  \sigma_\beta(r,k)
  =\sin\bigl(k(r+b(\beta))\bigr)
   +O_I\bigl(k^3+k(1+r)e^{-2r}\bigr),
  \qquad b(\beta)=-\gamma_E-\psi(\beta),
\end{equation}
and
\begin{equation}\label{eq:appendix-phase-derivative}
  \partial_r\sigma_\beta(r,k)
  =k\cos\bigl(k(r+b(\beta))\bigr)
   +O_I\bigl(k^4+k(1+r)e^{-2r}\bigr).
\end{equation}
Moreover \(k^{-1}\sigma_\beta(\cdot,k)\to y_\beta\) locally uniformly, where
\(y_\beta(r)=r+b(\beta)+O(re^{-2r})\) is the zero-energy regular solution.  All constants are
uniform for \(\beta\in I\), including the endpoint \(\beta=1/2\) with the Friedrichs convention.
\end{proposition}

\begin{proof}
We use the exact hypergeometric regular branch.  Put
\[
  z=\tanh^2 r,
  \qquad
  a_k=\frac{\beta-ik}{2},
  \qquad
  b_k=\frac{\beta+1-ik}{2},
  \qquad
  c=\beta+\frac12 .
\]
A direct substitution gives a solution regular at \(r=0\):
\begin{equation}\label{eq:positive-energy-regular-branch}
  U_{\beta,k}(r)
  =z^{\beta/2}(1-z)^{-ik/2}
   {}_2F_1(a_k,b_k;c;z).
\end{equation}
At \(k=0\) this is the zero-energy regular branch used in Lemma~\ref{lem:scattering}.  The formula
is uniform for \(\beta\in I\); when \(\beta=1/2\), the local behavior is \(r^{1/2}(1+O(r^2))\), so
it is exactly the Friedrichs branch.

We next move this solution to infinity.  The connection formula for \({}_2F_1\) at \(z=1\), with
\(c-a_k-b_k=ik\), gives
\begin{equation}\label{eq:connection-positive-energy}
  U_{\beta,k}(r)
  =C_\beta(k)e^{ikr}G^+_\beta(r,k)
   +\overline{C_\beta(k)}e^{-ikr}G^-_\beta(r,k),
\end{equation}
where \(G^-_\beta=\overline{G^+_\beta}\) for real \(k\), and
\begin{equation}\label{eq:Cbeta}
  C_\beta(k)
  =2^{-ik}
    \frac{\Gamma(c)\Gamma(ik)}
    {\Gamma\bigl((\beta+ik)/2\bigr)
     \Gamma\bigl((\beta+1+ik)/2\bigr)} .
\end{equation}
The functions \(G^\pm_\beta\) are the hypergeometric factors in the variable \(1-z=\operatorname{sech}^2r\), together with the explicit factor comparing \((1-z)^{\mp ik/2}\) with \(2^{\mp ik}e^{\pm ikr}\).  Hence, for \(j=0,1\),
\begin{equation}\label{eq:Gpm-bounds}
  \left|\partial_r^j(G^\pm_\beta(r,k)-G^\pm_\beta(r,0))\right|
  \le C_I k(1+r)e^{-2r},
  \qquad
  \left|\partial_r^j(G^\pm_\beta(r,0)-1)\right|
  \le C_I(1+r)e^{-2r}.
\end{equation}
These estimates follow by differentiating the hypergeometric series at \(1-z=0\).  The only possible
loss is the derivative of \((1-z)^{\mp ik/2}e^{\mp ikr}2^{\pm ik}\), and it is
\(O(k e^{-2r})\).  Uniformity in \(\beta\) is immediate because the parameters stay in a compact set
away from the poles of the Gamma and hypergeometric coefficients.

Let \(\arg C_\beta(k)\) be the continuous branch with \(\arg\Gamma(ik)\to-\pi/2\) as
\(k\downarrow0\), and set
\[
  \delta_\beta(k):=\arg C_\beta(k)+\frac\pi2 .
\]
Dividing \eqref{eq:connection-positive-energy} by \(2|C_\beta(k)|\) gives the real regular solution
with exterior sine amplitude one:
\begin{equation}\label{eq:sigma-normalized-appendix}
  \sigma_\beta(r,k)=\sin(kr+\delta_\beta(k))+E_\beta(r,k),
\end{equation}
where \eqref{eq:Gpm-bounds} implies
\begin{equation}\label{eq:appendix-spatial-error}
  |E_\beta(r,k)|\le C_I k(1+r)e^{-2r},
  \qquad
  |\partial_rE_\beta(r,k)|\le C_I k(1+r)e^{-2r}.
\end{equation}
The extra factor \(k(1+r)\) is important.  At \(k=0\), the normalized solution collapses to zero;
therefore the static amplitude correction \(G^\pm_\beta(r,0)-1\) is multiplied by
\(|\sin(kr+\delta_\beta(k))|\le C k(1+r)\) on the range \(r\le c_0/k\).

It remains to identify the threshold phase.  Taking the imaginary part of the logarithm of
\eqref{eq:Cbeta}, and using Taylor expansion of \(\log\Gamma\) on compact subsets of the right
half-plane, gives
\[
\begin{aligned}
  \arg C_\beta(k)
  &=-\frac\pi2
    +k\left[-\gamma_E-\log2
      -\frac12\psi\!\left(\frac\beta2\right)
      -\frac12\psi\!\left(\frac{\beta+1}{2}\right)
\right]
    +O_I(k^3).
\end{aligned}
\]
There is no quadratic term in the phase: around positive real arguments the imaginary part of
\(\log\Gamma(x+ik/2)\) contains only odd powers of \(k\), and the same is true for
\(\log\Gamma(ik)\) after the singular \(-\log(ik)\) term has supplied the constant \(-\pi/2\).  By
the duplication identity for the digamma function,
\[
  \psi\!\left(\frac\beta2\right)
  +\psi\!\left(\frac{\beta+1}{2}\right)
  =2\psi(\beta)-2\log2,
\]
we obtain
\begin{equation}\label{eq:delta-threshold-appendix}
  \delta_\beta(k)=k[-\gamma_E-\psi(\beta)]+O_I(k^3)=kb(\beta)+O_I(k^3).
\end{equation}
Combining \eqref{eq:sigma-normalized-appendix}, \eqref{eq:appendix-spatial-error}, and
\eqref{eq:delta-threshold-appendix} gives \eqref{eq:appendix-phase-function}.  Differentiating the
same representation gives \eqref{eq:appendix-phase-derivative}.  Finally, on compact \(r\)-intervals,
\[
  k^{-1}\sin\bigl(k(r+b(\beta))\bigr)\to r+b(\beta),
\]
and the spatial error divided by \(k\) is \(O_I((1+r)e^{-2r})\).  This recovers the zero-energy
branch from Lemma~\ref{lem:scattering}, including the stated tail estimate, and completes the proof.
\end{proof}

\end{document}